%% file: compositionality-Traced_Monads.tex
\documentclass[a4paper,onecolumn, superscriptaddress,10pt,accepted=2023-09-12,issue=10, volume=5, shorttitle=papers]{compositionalityarticle}
 \pdfoutput=1
\usepackage[utf8]{inputenc}
\usepackage[english]{babel}
\usepackage[T1]{fontenc}
\usepackage{hyperref}
\usepackage{enumerate,xspace}
\usepackage{amsmath,amssymb,wasysym}
\usepackage[all]{xy}
\usepackage{pict2e} 
\usepackage{tikz}
\input{fig-def.tex}
\usepackage[numbers,sort&compress]{natbib}

\newtheorem{observation}{Remark}[section]
\newtheorem{lemma}[observation]{Lemma}  
\newtheorem{theorem}[observation]{Theorem}
\newtheorem{definition}[observation]{Definition}
\newtheorem{example}[observation]{Example}

\newtheorem{proposition}[observation]{Proposition} 
\newtheorem{corollary}[observation]{Corollary}

\allowdisplaybreaks
\begin{document}

\title{Traced Monads and Hopf Monads}
\date{}
\author{Masahito Hasegawa}
\email{hassei@kurims.kyoto-u.ac.jp}
\homepage{https://www.kurims.kyoto-u.ac.jp/~hassei/}
\orcid{0000-0003-3460-8615}
\thanks{Supported by JSPS KAKENHI Grant No. JP18K11165, 
JP21K11753 and JST ERATO Grant No. JPMJER1603.}
\affiliation{Research Institute for Mathematical Sciences, Kyoto University, Kyoto, Japan}
\author{Jean-Simon Pacaud Lemay}
\email{js.lemay@mq.edu.au}
\homepage{https://sites.google.com/view/jspl-personal-webpage/}
\orcid{0000-0003-4124-3722}
\thanks{This 
research was financially supported by a JSPS Postdoctoral Fellowship -- Short Term, Award \#: PE19765, and a JSPS Postdoctoral Fellowship, Award \#: P21746.}
\affiliation{School of Mathematical and Physical Sciences, Macquarie University, Sydney, New South Wales, Australia}
\maketitle

\begin{abstract}
 A traced monad is a monad on a traced symmetric monoidal category that lifts the traced symmetric monoidal structure to its Eilenberg-Moore category. A long-standing question has been to provide a characterization of traced monads without explicitly mentioning the Eilenberg-Moore category. On the other hand, a symmetric Hopf monad is a symmetric bimonad whose fusion operators are invertible. For compact closed categories, symmetric Hopf monads are precisely the kind of monads that lift the compact closed structure to their Eilenberg-Moore categories. Since compact closed categories and traced symmetric monoidal categories are closely related, it is a natural question to ask what is the relationship between Hopf monads and traced monads. In this paper, we introduce trace-coherent Hopf monads on traced monoidal categories, which can be characterized without mentioning the Eilenberg-Moore category. The main theorem of this paper is that a symmetric Hopf monad is a traced monad if and only if it is a trace-coherent Hopf monad. We provide many examples of trace-coherent Hopf monads, such as those induced by cocommutative Hopf algebras or any symmetric Hopf monad on a compact closed category. We also explain how for traced Cartesian monoidal categories, trace-coherent Hopf monads can be expressed using the Conway operator, while for traced coCartesian monoidal categories, any trace-coherent Hopf monad is an idempotent monad. We also provide separating examples of traced monads that are not Hopf monads, as well as symmetric Hopf monads that are not trace-coherent. 
\end{abstract}


\section{Introduction}

Traced monoidal categories (Definition \ref{TSMC}) were introduced by Joyal, Street and Verity
in the 1990s \cite{joyal_street_verity_1996}.
Briefly, they are balanced (or often symmetric) monoidal categories equipped with a trace operator. In mathematics, the most common interpretation of the trace operator is as a generalization of the classical notion of the trace of matrices in linear algebra. Surprisingly, the trace operator also has many interpretations in various other areas such as the braid closure in knot theory, (quantum) traces in representation theory of (quantum) algebras, and feedback operators in control theory. Traced monoidal categories are also an important tool for theoretical computer scientists working on semantics of programming languages. Key applications of traced monoidal categories in this area include semantics of recursive programs and iterative programs via the trace-fixpoint correspondence \cite{Hasegawa97recursionfrom,simpson2000complete}, models of cyclic data structures and networks \cite{hasegawa2012models,stefanescu2000network}, and the Geometry of Interaction \cite{haghverdi2000unique,abramsky2002geometry}.

Being a structure with many such applications, one substantial problem with traced monoidal 
categories is that there are only a few known general constructions of them. Often we want to construct a new traced monoidal category from an existing one so that the new category inherits the traced monoidal structure from the original category while enjoying new additional properties or structures. Perhaps the most well-known construction on traced monoidal categories is the $\mathsf{INT}$-construction \cite{joyal_street_verity_1996}, which turns a traced monoidal category into a ribbon or compact closed category (Section \ref{sec:compactclosed}). The $\mathsf{INT}$-construction has several important applications but tends not to preserve the structures or properties that the original traced monoidal category possessed, including most limits and colimits. Another fairly general construction for traced monoidal categories is using uniformity of trace \cite{hasegawa2004uniformity}, which generalizes the construction of admissible predicates and the Scott induction principle in domain theory. However, this construction crucially depends on the presence of useful uniformity, which seems very rare in most examples of traced monoidal categories outside of domain theory. There is also a construction via a lax monoidal functor on the category of sets, due to Selinger \cite{selinger1999categorical}, which is an instance of the direct image construction from enriched category theory
\cite{benabou1967introduction}, which is not specific to traced monoidal categories.

In this paper, we focus on the problem of lifting traced monoidal structures 
to the category of algebras of a monad, which is often called the Eilenberg-Moore category of a monad (Definition \ref{def:Talg}). A good solution to this lifting problem would give us a new class of useful constructions on traced monoidal categories. This type of question of lifting the structure or property of an underlying category to the category of algebras of a monad has a long history. Most of the existing 
work on lifting problems is largely divided into two different contexts: simply asking to give the 
structure on the category of algebras, or asking in addition that the canonical forgetful functor (or more generally the monadic functor) preserves the structure. The former approach, of not assuming preservation, is often more suitable for questions of lifting properties which are unique. However, a trace operator on a monoidal category is not unique in general. Therefore for the story of this paper, we take the latter approach of asking for the preservation of structure by the forgetful functor, so that the traced monoidal structure on the category of algebras of a monad is the intended one related to that of the underlying category. 

We define a traced monad (Definition \ref{def:tracedmonad}) on a traced (symmetric) monoidal category to be a monad which lifts the trace operator to its Eilenberg-Moore category, that is, the Eilenberg-Moore category is a traced (symmetric) monoidal category such that the forgetful functor preserves the traced (symmetric) monoidal structure. However, often when defining monads that lift certain structures or properties, it is desirable to characterize such monads without mentioning their Eilenberg-Moore categories. This is usually achieved by equipping a monad with extra structure, and then proving that this extra structure is equivalent to a solution to the lifting problem. Unfortunately, for traced monads, it is still an open question whether it is possible to characterize them without explicitly mentioning the Eilenberg-Moore category and asking for preservation of the traced monoidal structure. In this paper, we provide a partial solution to this problem via the notion of Hopf monads. 

It is well-known that to lift a monoidal structure to the category of algebras 
of a monad $T$ is the same as to give a comonoidal structure on the monad, that is, to assume
coherence morphisms $T(I) \to I$ and ${T(A\otimes B) \to T(A)\otimes T(B)}$ subject to compatibility axioms with the monoidal structure of the underlying category \cite{moerdijk2002monads}. Following the terminology used in \cite{bruguieres2011hopf, bruguieres2007hopf}, monads with a comonoidal structure will be called bimonads (Definition \ref{def:bimonad}). For lifting symmetric monoidal structure, one must in addition ask that the comonoidal structure is symmetric (and similarly for lifting braided or balanced monoidal structure). One of the most remarkable results in this context of lifting monoidal structure is the discovery of the notion of Hopf monads, originally introduced by Bruguières and Virelizier \cite{bruguieres2007hopf}, and later generalized by Bruguières, Lack, and Virelizier \cite{bruguieres2011hopf}. Briefly, a Hopf monad (Definition \ref{Hopfmonaddef}) is a bimonad $T$ such that the bimonad's fusion operators, which are canonical natural transformations of type ${T(A\otimes T(B)) \to T(A) \otimes T(B)}$ and $T(T(A)\otimes B) \to T(A) \otimes T(B)$ which all bimonads have, are natural isomorphisms. Hopf monads characterize precisely monads that lift monoidal closed structure and also autonomous structure to their Eilenberg-Moore categories. In particular, for compact closed categories, a symmetric bimonad lifts the compact closed structure if and only if it is a Hopf monad (Proposition \ref{EMcompact}).

This result of Hopf monads lifting compact closed structure is relevant to the question of lifting traced symmetric monoidal structure. Indeed, every compact closed category is a traced symmetric monoidal category (Proposition \ref{compacttrace}), and conversely, as mentioned above, any traced symmetric monoidal category arises as a full monoidal subcategory of a compact closed category via the $\mathsf{INT}$-construction. It is then natural to expect that monads which lift traced symmetric monoidal structure are closely related to Hopf monads. In this paper, we will show that this intuition is valid to some extent, by providing necessary and sufficient conditions on a symmetric Hopf monad to lift traced symmetric monoidal structure, that is, to be a traced monad. Thus, we define a trace-coherent Hopf monad (Definition \ref{tracecoherent}) to be a symmetric Hopf monad on a traced symmetric monoidal category such that the trace operator and the monad's endofunctor are compatible via the fusion operators. The main result of this paper is that a symmetric Hopf monad is a traced monad if and only if it is trace-coherent (Theorem \ref{mainthm}). Therefore, the Eilenberg-Moore category of a trace-coherent Hopf monad is a traced symmetric monoidal category such that the forgetful functor preserves the structure strictly. The main advantage of trace-coherent Hopf monads is that they are characterized without explicitly mentioning the Eilenberg-Moore category, as desired. We also show that every symmetric Hopf monad on a compact closed category is trace-coherent (Corollary \ref{tracecoherentcompactclosed}), and also how every symmetric Hopf monad induced by a cocommutative Hopf algebra is also trace-coherent (Proposition \ref{prop:hopfalgtrace}), and so the trace operator lifts to the category of modules of a cocommutative Hopf algebra (Corollary \ref{cor:tracemodules}). We also explain how for traced Cartesian monoidal categories, trace-coherent Hopf monads can be expressed using the Conway operator (Proposition \ref{prop:hopftracefix}), while for traced coCartesian monoidal categories, any trace-coherent Hopf monad is an idempotent monad (Corollary \ref{tracecoherentcocart}). At the same time, we also show that the situation is not as simple as one might expect, by giving examples of traced monads which are not Hopf, and also a symmetric Hopf monad that is not trace-coherent, and thus does not lift the trace (Section \ref{counterexamplesec}).  

It is worth mentioning that in this paper we focus on traced {\em symmetric} monoidal categories rather than general traced {\em balanced} (or braided) monoidal categories, for the following reasons. First, handling balanced monoidal structure makes the technical details considerably more complex than just working on symmetric monoidal structure, while the consideration of trace is mostly the same in the balanced and symmetric cases. So, while most of our results (including the main theorem) are actually valid in the balanced cases, we shall concentrate on the symmetric case to make the story reasonably simple and short. Second, most of the important examples of traced monoidal categories are symmetric ones, and so are all of our (counter)examples used in this paper. Moreover, the lifting problem for ribbon categories was already solved in terms of ribbon Hopf monads \cite{bruguieres2007hopf}. Hence we believe that presenting this work in the balanced case is not really a pressing issue. Nevertheless, we are sure that readers familiar with balanced monoidal categories should not have too much difficulty in generalizing our results in this paper to the balanced cases.

\section{Traced Symmetric Monoidal Categories}\label{sec:traced}

In this background section, we review traced \emph{symmetric} monoidal categories. Traced monoidal categories were introduced by Joyal, Street, and Verity \cite{joyal_street_verity_1996}, and are balanced monoidal categories equipped with a trace operator. In this paper, we will work with the symmetric version of traced monoidal categories, and so will provide the definition of traced symmetric monoidal categories as found in \cite{Hasegawa97recursionfrom}, except that we omit an axiom (Vanishing for unit) which is known to be redundant \cite[Appendix A]{hasegawa2009traced}. Alternate, but equivalent, axiomatizations of traced symmetric monoidal categories can be found in \cite[Section 3]{hasegawa2009traced} and \cite[Section 2.1]{hasegawa2008finite}. 

As the name suggests, the underlying category of a traced symmetric monoidal category is a symmetric monoidal category, which is a category equipped with a tensor product. For an in-depth introduction to (symmetric) monoidal categories, and their axioms written out in commutative diagrams, we refer the reader to \cite[Section 3]{selinger2010survey}.  

\begin{definition} \cite[Section 3.1, 3.3, \& 3.5]{selinger2010survey} A \textbf{monoidal category} is a sextuple $(\mathbb{X}, \otimes, I, \alpha, \lambda, \rho)$ consisting of a category $\mathbb{X}$, a functor ${\otimes: \mathbb{X} \times \mathbb{X} \to \mathbb{X}}$ called the \textbf{monoidal product}, an object $I$ called the \textbf{monoidal unit}, a natural isomorphism $\alpha_{A,B,C}: A \otimes (B \otimes C) \xrightarrow{\cong} (A \otimes B) \otimes C$ called the associativity isomorphism, a natural isomorphism $\lambda_A: I \otimes A \xrightarrow{\cong}  A$ called the left unit isomorphism, and a natural isomorphism $\rho_A: A \otimes I \xrightarrow{\cong}  A $ called the right unit isomorphism, such that the following equalities hold: 
\begin{equation}\label{monoidalaxiom}\begin{gathered}
\alpha_{A \otimes B, C, D} \circ \alpha_{A,B,C \otimes D} = (\alpha_{A,B,C} \otimes 1_D) \circ \alpha_{A,B \otimes C, D} \circ (1_A \otimes \alpha_{B,C,D}) \\
(1_A \otimes \lambda_B) \circ \alpha_{A,I,B} = \rho_A \otimes 1_B  
\end{gathered}\end  {equation}
A \textbf{symmetric monoidal category} is a septuple $(\mathbb{X}, \otimes, I, \alpha, \lambda, \rho, \sigma)$ consisting of a monoidal category $(\mathbb{X}, \otimes, I, \alpha, \lambda, \rho)$ and a natural isomorphism $\sigma_{A,B}: A \otimes B \xrightarrow{\cong}  B \otimes A$ called the symmetry isomorphism, such that the following equalities hold:
\begin{equation}\label{symmonoidalaxiom}\begin{gathered}
\sigma_{B,A} \circ \sigma_{A,B} = 1_{A \otimes B} \\
(1_B \circ \sigma_{A,C}) \circ \alpha_{B,A,C} \circ (\sigma_{A,B} \otimes 1_C) = \alpha_{B,C,A} \circ \sigma_{A,B \otimes C} \circ \alpha_{A,B,C} 
\end{gathered}\end  {equation}
\end{definition}

A traced symmetric monoidal category is a symmetric monoidal category that comes equipped with a trace operator. The axioms of the trace operator are best understood using the graphical calculus for (symmetric) monoidal categories, which we will also make use of in Section \ref{sec:hopfcase}. We will not review in full the graphical calculus here, and we refer the reader to \cite{selinger2010survey} for an in-depth introduction. We use more-or-less the same conventions as found in \cite{selinger2010survey}, so, in particular, our string diagrams should compositionally be read horizontally from left to right, and with the wires read from bottom to top. Furthermore, while equationally we work in the non-strict setting (to provide a complete general story), following the standard convention, graphically we allow ourselves to work in a strict setting and so omit the associativity and unit isomorphisms in our string diagrams -- which is justified by the coherence theorem for symmetric monoidal categories. 

\begin{definition} \label{TSMC} \cite[Definition 2.1]{Hasegawa97recursionfrom}  A \textbf{traced symmetric monoidal category} is an octuple $(\mathbb{X}, \otimes, I, \alpha, \lambda, \rho, \sigma, \mathsf{Tr})$ consisting of a symmetric monoidal category $(\mathbb{X}, \otimes, I, \alpha, \lambda, \rho, \sigma)$ equipped with a \textbf{trace operator} $\mathsf{Tr}$, which is a family of operators (indexed by triples of objects $X, A, B \in \mathbb{X}$),  $\mathsf{Tr}^X_{A,B}: \mathbb{X}(A \otimes X, B \otimes X) \to \mathbb{X}(A,B)$, which is drawn in the graphical calculus as follows: 
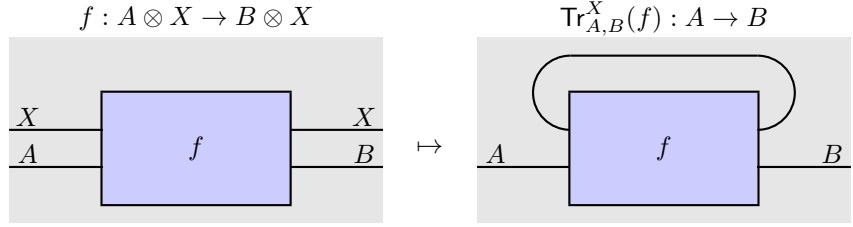
\begin{center}
\unitlength=.7pt
\begin{picture}(200,120)
\thicklines
\graybox{0}{0}{200}{100}{}%
\lineseg{0}{50}{200}{50}
\lineseg{0}{30}{200}{30}
\violetfbox{50}{10}{100}{60}{$f$}%
\put(100,110){\makebox(0,0){$f:A\otimes X\to B\otimes X$}}
\put(10,57){\makebox(0,0){$X$}}
\put(10,37){\makebox(0,0){$A$}}
\put(190,57){\makebox(0,0){$X$}}
\put(190,37){\makebox(0,0){$B$}}
\end{picture}
\begin{picture}(40,120)
\put(20,40){\makebox(0,0){$\mapsto$}}
\end{picture}
\begin{picture}(200,120)
\thicklines
\graybox{0}{0}{200}{100}{}%
\trace{50}{70}{100}{20}
\lineseg{0}{30}{200}{30}
\violetfbox{50}{10}{100}{60}{$f$}%
\put(100,110){\makebox(0,0){$\mathsf{Tr}^X_{A,B}(f):A\to B$}}
\put(10,37){\makebox(0,0){$A$}}
\put(190,37){\makebox(0,0){$B$}}
\end{picture}
\end{center}
such that the following axioms are satisfied: 
\begin{description}
\item[\textbf{[Tightening]:}] For every map $f: A \otimes X \to B \otimes X$ and $g: A^\prime \to A$, the following equality holds: 
 \begin{equation}\label{tightening1}\begin{gathered}\mathsf{Tr}^X_{A^\prime,B}\left( f \circ (g \otimes 1_X)  \right) = \mathsf{Tr}^X_{A,B}(f) \circ g \\
\unitlength=.7pt
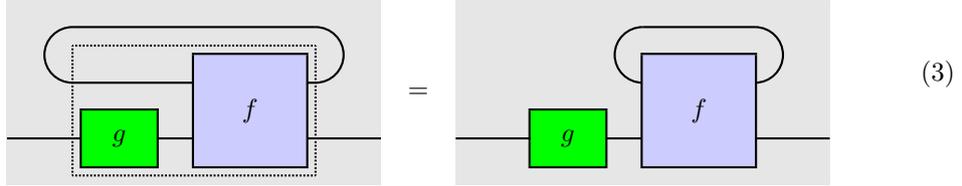
\begin{picture}(200,100)
\thicklines
\graybox{0}{0}{200}{100}{}%
\lineseg{0}{25}{200}{25}
\trace{35}{70}{130}{15}
\lineseg{35}{55}{165}{55}
\greenfbox{40}{10}{40}{30}{$g$}%
\violetfbox{100}{10}{60}{60}{$f$}%
\put(35,5){\dashbox(130,70){}}
\end{picture}
\begin{picture}(40,100)
\put(20,50){\makebox(0,0){$=$}}
\end{picture}
\begin{picture}(200,100)
\thicklines
\graybox{0}{0}{200}{100}{}%
\lineseg{0}{25}{200}{25}
\trace{100}{70}{60}{15}
\greenfbox{40}{10}{40}{30}{$g$}%
\violetfbox{100}{10}{60}{60}{$f$}%
\end{picture}
   \end{gathered}\end  {equation}
and for every map $f: A \otimes X \to B \otimes X$ and $h: B \to B^\prime$, the following equality holds: 
 \begin{equation}\label{tightening2}\begin{gathered} \mathsf{Tr}^X_{A,B^\prime}\left( (h \otimes 1_X) \circ f  \right) = h \circ \mathsf{Tr}^X_{A,B}(f) \\
 \unitlength=.7pt
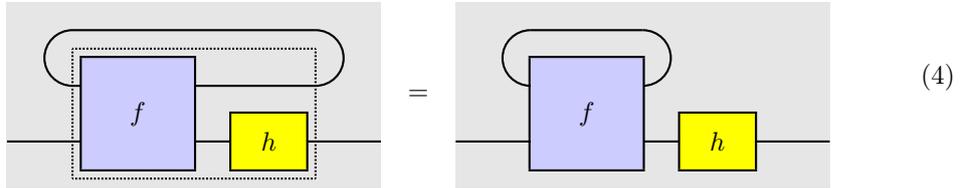
\begin{picture}(200,100)
\thicklines
\graybox{0}{0}{200}{100}{}%
\lineseg{0}{25}{200}{25}
\trace{35}{70}{130}{15}
\lineseg{35}{55}{165}{55}
\yellowfbox{120}{10}{40}{30}{$h$}%
\violetfbox{40}{10}{60}{60}{$f$}%
\put(35,5){\dashbox(130,70){}}
\end{picture}
\begin{picture}(40,100)
\put(20,50){\makebox(0,0){$=$}}
\end{picture}
\begin{picture}(200,100)
\thicklines
\graybox{0}{0}{200}{100}{}%
\lineseg{0}{25}{200}{25}
\trace{40}{70}{60}{15}
\yellowfbox{120}{10}{40}{30}{$h$}%
\violetfbox{40}{10}{60}{60}{$f$}%
\end{picture}
   \end{gathered}\end  {equation}
\item[\textbf{[Sliding]:}] For every map $f: A \otimes X \to B \otimes X^\prime$ and $k: X^\prime \to X$, the following equality holds: 
 \begin{equation}\label{sliding}\begin{gathered} \mathsf{Tr}^X_{A,B}\left(f \circ (1_A \otimes k) \right) = \mathsf{Tr}^{X^\prime}_{A,B}\left((1_B \otimes k)  \circ f \right) \\
 \unitlength=.7pt
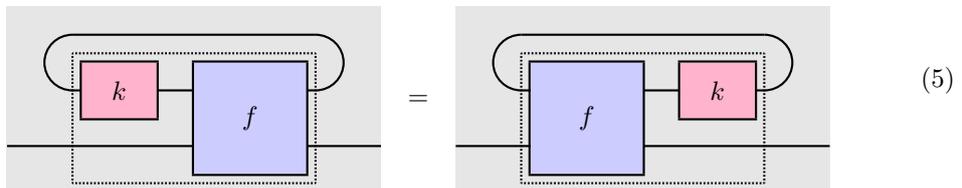
\begin{picture}(200,100)
\thicklines
\graybox{0}{0}{200}{100}{}%
\lineseg{0}{25}{200}{25}
\trace{35}{70}{130}{15}
\lineseg{35}{55}{165}{55}
\redfbox{40}{40}{40}{30}{$k$}%
\violetfbox{100}{10}{60}{60}{$f$}%
\put(35,5){\dashbox(130,70){}}
\end{picture}
\begin{picture}(40,100)
\put(20,50){\makebox(0,0){$=$}}
\end{picture}
\begin{picture}(200,100)
\thicklines
\graybox{0}{0}{200}{100}{}%
\lineseg{0}{25}{200}{25}
\trace{35}{70}{130}{15}
\lineseg{35}{55}{165}{55}
\redfbox{120}{40}{40}{30}{$k$}%
\violetfbox{40}{10}{60}{60}{$f$}%
\put(35,5){\dashbox(130,70){}}
\end{picture}
 \end{gathered}\end  {equation}
\item[\textbf{[Vanishing]:}] For every map $f: A \otimes (X \otimes Y) \to B \otimes (X \otimes Y)$, the following equality holds: 
 \begin{equation}\label{vanishing1}\begin{gathered}  \mathsf{Tr}^{X \otimes Y}_{A,B}(f) =     \mathsf{Tr}^X_{A,B} \left( \mathsf{Tr}^Y_{A \otimes X,B \otimes X}\left( \alpha_{B,X,Y} \circ f \circ \alpha^{-1}_{A,X,Y} \right) \right) \\
\unitlength=.7pt
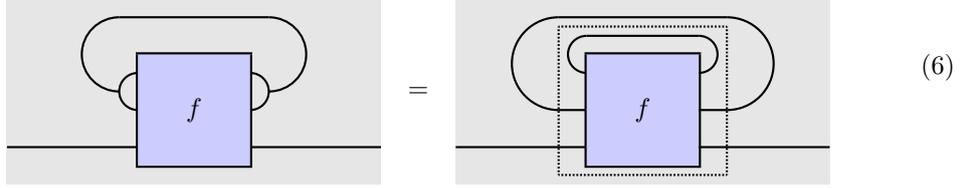
\begin{picture}(200,100)
\thicklines
\graybox{0}{0}{200}{100}{}%
\lineseg{0}{20}{200}{20}
\trace{60}{70}{80}{20}
\put(70,50){\circle{20}}
\put(130,50){\circle{20}}
\violetfbox{70}{10}{60}{60}{$f$}%
\end{picture}
\begin{picture}(40,100)
\put(20,50){\makebox(0,0){$=$}}
\end{picture}
\begin{picture}(200,100)
\thicklines
\graybox{0}{0}{200}{100}{}%
\lineseg{0}{20}{200}{20}
\trace{70}{70}{60}{10}
\trace{55}{65}{90}{25}
\lineseg{55}{40}{145}{40}
\violetfbox{70}{10}{60}{60}{$f$}%
\put(55,5){\dashbox(90,80){}}
\end{picture}
  \end{gathered}\end  {equation}
\item[\textbf{[Superposing]:}] For every map $f: A \otimes X \to B \otimes X$ the following equality holds:
 \begin{equation}\label{superposing}\begin{gathered} \mathsf{Tr}^X_{C \otimes A,C \otimes B}\left( \alpha_{C,A,X}  \circ (1_C \otimes f) \circ  \alpha^{-1}_{C,A,X} \right) = 1_C \otimes  \mathsf{Tr}^X_{A,B}(f) \\
   \unitlength=.7pt
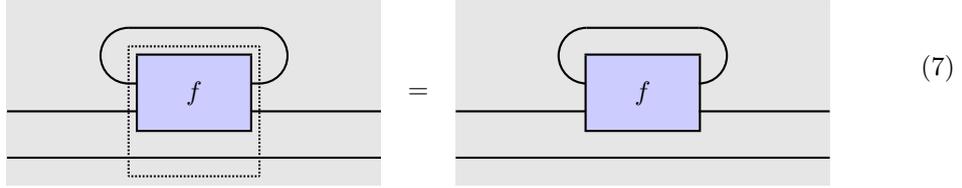
\begin{picture}(200,100)
\thicklines
\graybox{0}{0}{200}{100}{}%
\lineseg{0}{15}{200}{15}
\lineseg{0}{40}{200}{40}
\trace{65}{70}{70}{15}
\lineseg{65}{55}{135}{55}
\violetfbox{70}{30}{60}{40}{$f$}%
\put(65,5){\dashbox(70,70){}}
\end{picture}
\begin{picture}(40,100)
\put(20,50){\makebox(0,0){$=$}}
\end{picture}
\begin{picture}(200,100)
\thicklines
\graybox{0}{0}{200}{100}{}%
\lineseg{0}{15}{200}{15}
\lineseg{0}{40}{200}{40}
\trace{70}{70}{60}{15}
\violetfbox{70}{30}{60}{40}{$f$}%
\end{picture}
  \end{gathered}\end{equation}
\item[\textbf{[Yanking]:}] For every object $X$, the following equality holds: 
\begin{equation}\label{yanking}\begin{gathered} \mathsf{Tr}^X_{X,X}( \sigma_{X,X} ) = 1_X  \\
\unitlength=.7pt
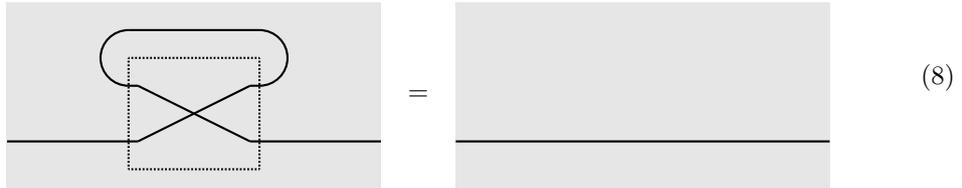
\begin{picture}(200,100)
\thicklines
\graybox{0}{0}{200}{100}{}%
\lineseg{0}{25}{70}{25}
\lineseg{130}{25}{200}{25}
\trace{65}{70}{70}{15}
\lineseg{65}{55}{70}{55}
\lineseg{130}{55}{135}{55}
\lineseg{70}{25}{130}{55}
\lineseg{70}{55}{130}{25}
\put(65,10){\dashbox(70,60){}}
\end{picture}
\begin{picture}(40,100)
\put(20,50){\makebox(0,0){$=$}}
\end{picture}
\begin{picture}(200,100)
\thicklines
\graybox{0}{0}{200}{100}{}%
\lineseg{0}{25}{200}{25}
\end{picture}
  \end{gathered}\end  {equation}
  \end{description}
For a map $f: A \otimes X \to B \otimes X$, the map $\mathsf{Tr}^X_{A,B}(f): A \to B$ is called the \textbf{trace} of $f$. 
\end{definition}

There are numerous interesting examples of traced symmetric monoidal categories in the literature, see \cite[Section 5]{abramsky2002geometry} for a good list of examples. In this paper, we will discuss compact closed categories (Section \ref{sec:compactclosed}), whose trace operator captures the trace of matrices from linear algebra, traced Cartesian monoidal categories (Section \ref{carttracesec}), where the trace operator is given by fixed points, and traced coCartesian monoidal categories (Section \ref{coCarttracemonadsec}), where the trace operator is given by iterations. We also consider some concrete examples of traced symmetric monoidal categories in Section \ref{counterexamplesec}. 

\section{Traced Monads}\label{tracedmonadsec}

In this section, we introduce the notion of \emph{traced monads}, which are precisely the sort of monads on traced symmetric monoidal categories that lift the traced monoidal structure. In order to properly define traced monads, we will also quickly review the intermediate notions of monads and bimonads, as well as their algebras and Eilenberg-Moore categories. We refer the reader to \cite[Chapter 4]{borceux1994handbook} for a detailed introduction to monads. 

\begin{definition} \cite[Definition 4.1.1]{borceux1994handbook} A \textbf{monad} on a category $\mathbb{X}$ is a triple $(T, \mu, \eta)$ consisting of a functor ${T: \mathbb{X} \to \mathbb{X}}$, a natural transformation $\mu_A: TT(A) \to T(A)$, called the monad multiplication, and a natural transformation ${\eta_A: A \to T(A)}$, called the monad unit, such that the following equalities hold: 
\begin{align}\label{monadeq}
\mu_A \circ T(\eta_A) = 1_{T(A)} = \mu_A \circ \eta_{T(A)} && \mu_A \circ T(\mu_A) = \mu_A \circ \mu_{T(A)}   
\end{align}
\end{definition}

We now review algebras of a monad, which were called modules of a monad in \cite{bruguieres2011hopf,bruguieres2007hopf}. 

\begin{definition}\label{def:Talg} \cite[Definition 4.1.2]{borceux1994handbook} Let $(T, \mu, \eta)$ be a monad on a category $\mathbb{X}$. 
\begin{enumerate}[{\em (i)}]
\item A \textbf{$T$-algebra} is a pair $(A, a)$ consisting of an object $A$ and a map $a: T(A) \to A$ of $\mathbb{X}$, called the $T$-algebra structure, such that the following equalities hold: 
\begin{align}\label{Talg}
a \circ \eta_A = 1_A && a \circ \mu_A = a \circ T(a)     
\end{align}
\item For each object $A$, the \textbf{free $T$-algebra} over $A$ is the $T$-algebra $(T(A), \mu_A)$. 
\item If $(A,a)$ and $(B,b)$ are $T$-algebras, then a \textbf{$T$-algebra morphism} $f: (A, a)\to (B, b)$ is a map in $\mathbb{X}$ between the underlying objects $f: A\to B$ such that the following equality holds: 
\begin{align}\label{Talgmap}
f \circ a = b \circ T(f)     
\end{align}
\end{enumerate}
The \textbf{Eilenberg-Moore category} of the monad $(T, \mu, \eta)$ is the category $\mathbb{X}^T$ whose objects are $T$-algebras and whose maps are $T$-algebra morphisms between them. Define the functor $U^{T}: \mathbb{X}^T \to \mathbb{X}$, called the \textbf{forgetful functor}, on objects as $U(A, a)= A$ and on maps as $U(f)=f$.
\end{definition}

The following well-known lemma about algebra morphisms whose domain is a free algebra will be particularly important for the proof of the main result of this paper (Theorem \ref{mainthm}): 

\begin{lemma} \label{Talglemma4} Let $(T, \mu, \eta)$ be a monad on a category $\mathbb{X}$, and let $(B, b)$ be a $T$-algebra, and let ${f: (T(A), \mu_A) \to (B,b)}$ and $g: (T(A), \mu_A) \to (B,b)$ be $T$-algebra morphisms. Then $f=g$ if and only if $f \circ \eta_A = g \circ \eta_A$. 
\end{lemma}

In order to lift traced symmetric monoidal structure, one must first be able to lift symmetric monoidal structure. As such, we now review the notion of a (symmetric) bimonad, which is precisely the kind of monad on a (symmetric) monoidal category which lifts the (symmetric) monoidal structure to its Eilenberg-Moore category. Bimonads were first introduced by Moerdijk \cite{moerdijk2002monads} under the name Hopf monads. However, Brugui{\`e}res and Virelizier \cite{bruguieres2007hopf} renamed Moerdijk's Hopf monads to bimonads since every bimonoid induces a bimonad. Note that bimonads are also sometimes called opmonoidal monads \cite{mccrudden2002opmonoidal} or comonoidal monads \cite{hasegawa2018linear}. A bimonad is defined as a monad on a monoidal category that comes equipped with extra structure that is compatible with the monoidal structure. The symmetric version of bimonads, called symmetric bimonads, also requires compatibility with the symmetry. 

\begin{definition} \label{def:bimonad} \cite[Definition 1.1]{moerdijk2002monads} A \textbf{bimonad} on a monoidal category $(\mathbb{X}, \!\otimes, I, \alpha, \lambda, \rho)$ is a quintuple $(T, \mu, \eta, m, m_{I})$ consisting of a monad $(T, \mu, \eta)$ equipped with a natural transformation $m_{A,B}: T(A \otimes B) \to T(A) \otimes T(B)$, and a map ${m_I: T(I) \to I}$ such that:
\begin{enumerate}[{\em (i)}]
\item $(T, m, m_I)$ is a \textbf{comonoidal endofunctor}, that is, the following equalities hold: 
\begin{equation}\label{smcendo}\begin{gathered} 
\alpha_{T(A), T(B), T(C)} \circ (1_{T(A)} \otimes m_{B,C}) \circ m_{A, B \otimes C} = m_{A,B} \otimes 1_{T(C)} \circ m_{A \otimes B, C} \circ T(\alpha_{A,B,C}) \\
\lambda_{T(A)} \circ (m_I \otimes 1_{T(A)}) \circ m_{I,A} \circ T(\lambda^{-1}_A) = 1_{T(A)} = \rho_{T(A)} \circ (1_{T(A)} \otimes m_I) \circ m_{A,I} \circ T(\rho^{-1}_A)
\end{gathered}\end  {equation}
 \item $\mu$ and $\eta$ are \textbf{comonoidal transformations}, that is, the following equalities hold: 
  \begin{equation}\label{symbimonad}\begin{gathered}
  m_{A,B} \circ \mu_{A  \otimes B} = (\mu_A \otimes \mu_B) \circ m_{T(A),T(B)} \circ T(m_{A,B}) \quad \quad \quad \quad m_I \circ \mu_I = m_I \circ T(m_I) \\ 
  m_{A,B} \circ \eta_{A \otimes B} = \eta_A \otimes \eta_B \quad \quad \quad \quad m_I \circ \eta_I = 1_I
  \end{gathered}\end{equation}
\end{enumerate}
\end{definition}

\begin{definition} \cite[Section 3]{moerdijk2002monads} A \textbf{symmetric bimonad} on a symmetric monoidal category \\
$(\mathbb{X}, \!\otimes, \!I, \alpha, \lambda, \rho, \sigma)$ is a bimonad $(T, \mu, \eta, m, m_{I})$ on the underlying monoidal category $(\mathbb{X}, \!\otimes, \!I, \alpha, \lambda, \rho)$ such that $(T, m, m_I)$ is also a \textbf{symmetric comonoidal endofunctor}, that is, the following equality holds:
 \begin{equation}\label{smcendo2}\begin{gathered} 
 m_{B,A} \circ T(\sigma_{A,B}) = \sigma_{T(A), T(B)} \circ m_{A,B}
 \end{gathered}\end  {equation}
\end{definition}

An important class of examples of (symmetric) bimonads are those induced by (cocommutative) bimonoids, which we review in Section \ref{sec:hopfcase}. In Section \ref{carttracesec}, we will see how any monad on a Cartesian monoidal category is a symmetric bimonad. Other explicit examples of symmetric bimonads can be found in Section \ref{counterexamplesec}. 

As first observed by Moerdijk \cite[Theorem 7.1]{moerdijk2002monads}, the Eilenberg-Moore category of a bimonad is a monoidal category such that the forgetful functor preserves the monoidal structure strictly (which we discuss a special case in Proposition \ref{EMTSMC} below). Similarly, the Eilenberg-Moore category of a symmetric bimonad is a symmetric monoidal category such that the forgetful functor preserves the symmetric monoidal structure strictly, that is, a symmetric bimonad lifts symmetric monoidal structure \cite[Proposition 3.2]{moerdijk2002monads}. In fact, for a monad on a (symmetric) monoidal category, (symmetric) comonoidal structures on the monad are in bijective correspondence with (symmetric) monoidal structures on the Eilenberg-Moore category which are strictly preserved by the forgetful functor. Thus a monad on a (symmetric) monoidal category is a (symmetric) bimonad if and only if the (symmetric) monoidal structure lifts to its Eilenberg-Moore category. For the purposes of this paper, let us explicitly review how the Eilenberg-Moore category of a symmetric bimonad is a symmetric monoidal category. 

\begin{proposition} \label{EMMCprop}\cite[Proposition 3.2]{moerdijk2002monads} Let $(T, \mu, \eta, m, m_I)$ be a symmetric bimonad of a symmetric monoidal category $(\mathbb{X}, \!\otimes, I, \alpha, \!\lambda, \rho, \sigma)$. Then $\left(\mathbb{X}^T, \otimes^T, (I, m_I), \alpha^T, \lambda^T, \rho^T, \sigma^T \right)$ is a symmetric monoidal category where:  
\begin{enumerate}[{\em (i)}]
\item The tensor product $\otimes^T$ of a pair of $T$-algebras $(A, a)$ and $(B, b)$, is the $T$-algebra defined as:
\[(A, a) \otimes^T (B, b) := (A \otimes B, (a \otimes b) \circ m_{A,B})\] 
while the tensor product of $T$-algebra morphisms $f$ and $g$ is defined as $f \otimes^{T} g := f \otimes g$;
\item The unit is the pair $(I, m_I)$;
\item The four natural isomorphisms:
\begin{gather*}
\alpha^T_{(A,a), (B, b), (C. c)}: (A,a) \otimes^T \left( (B,b) \otimes^T (C,c) \right) \xrightarrow{\cong} \left( (A,a) \otimes^T (B,b) \right) \otimes^T (C,c)\\
\lambda^T_{(A,a)}: (I,m_I) \otimes^T (A,a) \xrightarrow{\cong} (A,a) \\
\rho^T_{(A,a)}: (A,a) \otimes^T  (I,m_I) \xrightarrow{\cong} (A,a) \\
\sigma^T_{(A,a), (B, b)}: (A,a) \otimes^T (B,b) \xrightarrow{\cong} (B,b) \otimes^T (A,a)
\end{gather*}
and their inverses are respectively defined as follows:
\begin{align*}
\alpha^T_{(A,a), (B, b), (C. c)} = \alpha_{A,B,C} &&  \lambda^T_{(A,a)} = \lambda_A && \rho^T_{(A,a)} = \rho_A && \sigma^T_{(A,a), (B, b)} = \sigma_{A,B} \\
{\alpha^T}^{-1}_{(A,a), (B, b), (C. c)} = {\alpha}^{-1}_{A,B,C} &&  {\lambda^T}^{-1}_{(A,a)} = {\lambda}^{-1}_A && {\rho^T}^{-1}_{(A,a)} = {\rho}^{-1}_A && {\sigma^T}^{-1}_{(A,a), (B, b)} = {\sigma}^{-1}_{A,B} 
\end{align*}
\end{enumerate}
Furthermore the forgetful functor $U^T: \left(\mathbb{X}^T, \otimes^T, (I, m_I), \alpha^T, \lambda^T, \rho^T  \right) \to (\mathbb{X}, \otimes, I, \alpha, \lambda, \rho)$ is a strict symmetric monoidal functor, that is, the following equalities hold: 
\begin{gather*}
U^T(- \otimes^T -) =U^T(-) \otimes U^T(-) \quad \quad \quad U^T(I,m_I) = I \quad \quad \quad U^T( \alpha^T_{-,-,-}) = \alpha_{U^T(-), U^T(-), U^T(-)} \\  U^T(\lambda^T_{-}) = \lambda_{U^T(-)} \quad \quad \quad U^T(\rho^T_{-}) = \rho_{U^T(-)} \quad \quad \quad U^T( \sigma^T_{-,-}) = \sigma_{U^T(-), U^T(-)} 
\end{gather*}
\end{proposition}

We are now in a position to properly define what a traced monad is. In short, a traced monad is a symmetric bimonad such that the trace of an algebra morphism is again an algebra morphism.

\begin{definition} \label{def:tracedmonad} A \textbf{traced monad} on a traced symmetric monoidal category $(\mathbb{X}, \otimes, I, \alpha, \lambda, \rho, \sigma, \mathsf{Tr})$ is a symmetric bimonad $(T, \mu, \eta, m, m_I)$ on the underlying symmetric monoidal category \\ $(\mathbb{X}, \otimes, I, \alpha, \lambda, \rho, \sigma)$ such that for every triple of $T$-algebras $(X,x)$, $(A,a)$, and $(B,b)$, if:
\[f: (A,a) \otimes^T (X,x) \to (B,b) \otimes^T (X,x)\]
is a $T$-algebra morphism, then $\mathsf{Tr}^X_{A,B}(f): (A,a) \to (B,b)$ is a $T$-algebra morphism. Explicitly, if the equality on the left holds, then the equality on the right holds: 
\begin{equation}\label{tracedmonad}\begin{gathered}
f \circ (a \otimes x) \circ m_{A,X} = (b \otimes x) \circ m_{B,X} \circ T(f) \Longrightarrow  \mathsf{Tr}^X_{A,B}(f) \circ a = b \circ T\left( \mathsf{Tr}^X_{A,B}(f) \right)
\end{gathered}\end  {equation}
\end{definition}

Examples of traced monads can be found below throughout the remainder of this paper, with some concrete examples in Section \ref{counterexamplesec}. We now explicitly state how the Eilenberg-Moore category of a traced monad is a traced symmetric monoidal category, such that the forgetful functor preserves the traced symmetric monoidal structure strictly. 

\begin{proposition}\label{EMTSMC} Let $(T, \mu, \eta, m, m_I)$ be a traced monad on a traced symmetric monoidal category $(\mathbb{X}, \!\otimes, I, \alpha, \!\lambda, \rho, \sigma, \!\mathsf{Tr})$. Then $\left(\mathbb{X}^T, \otimes^T, (I, m_I), \alpha^T, \lambda^T, \rho^T, \sigma^T, \mathsf{Tr}^T \right)$ is a traced symmetric monoidal category where the underlying symmetric monoidal $\left(\mathbb{X}^T, \otimes^T, (I, m_I), \alpha^T, \lambda^T, \rho^T, \sigma^T \right)$ is defined as in Proposition \ref{EMMCprop}, and the trace operator $\mathsf{Tr}^T$ is defined as follows for a $T$-algebra morphism $f: (A,a) \otimes^T (X,x) \to (B,b) \otimes^T (X,x)$: 
\[ {\mathsf{Tr}^T}^{(X,x)}_{(A,a), (B,b)}(f) = \mathsf{Tr}^X_{A,B}(f) : (A,a) \to (B,b) \]
Furthermore the forgetful functor $U^T: \left(\mathbb{X}^T, \otimes^T, (I, m_I), \alpha^T, \lambda^T, \rho^T, \mathsf{Tr}^T  \right) \to (\mathbb{X}, \otimes, I, \alpha, \lambda, \rho, \mathsf{Tr})$ a strict traced symmetric monoidal functor, that is, the equalities in Proposition \ref{EMMCprop} hold and the following equality also holds: 
\begin{gather*}
U^T\left(\mathsf{Tr}^T(-)  \right) = \mathsf{Tr}\left( U^T (-) \right)
\end{gather*}
\end{proposition} 
\begin{proof} This is an extension of Proposition \ref{EMMCprop}, so it remains to show the trace part of the story is well-behaved. The trace operator $\mathsf{Tr}^T$ is well-defined by (\ref{tracedmonad}) and satisfies the necessary axioms since $\mathsf{Tr}$ is a trace operator. Lastly, it is automatic by definition that $U^T\left(\mathsf{Tr}^T(-)  \right) = \mathsf{Tr}\left( U^T (-) \right)$. 

\phantom{ }\hfill \end{proof}

Observe that, unlike the definition of a traced monad, the definition of a (symmetric) bimonad does not mention algebras or algebra morphisms. It is an open question about traced monads whether it is possible to characterize them without explicitly mentioning their algebras, such as is done with bimonads. The next section provides a solution to this for a special class of bimonads called Hopf monads. 

\section{Trace-Coherent Hopf Monads}\label{tracecoherentsec}

Originally, Hopf monads were introduced by Bruguières and Virelizier \cite{bruguieres2007hopf} as monads on autonomous categories (monoidal categories with duals) that lift autonomous structure to their Eilenberg-Moore category. Together with Lack \cite{bruguieres2011hopf}, they generalized Hopf monads to arbitrary monoidal categories, and as the name suggests, every Hopf monoid induces a Hopf monad. In regards to lifting structure, for monoidal \emph{closed} categories, Hopf monads are precisely the monads that lift monoidal closed structure to their Eilenberg-Moore category.  Of particular interest for this paper, is the case of compact closed categories (which are both monoidal closed and autonomous), where symmetric Hopf monads are precisely the monads that lift compact closed structure to their Eilenberg-Moore category. In this section, we introduce the main novel concept of this paper: trace-coherent Hopf monads, which are precisely the kinds of symmetric Hopf monads that are also traced monads. The advantage of the definition of a trace-coherent Hopf monad is that it does not reference the algebras of the monad or the Eilenberg-Moore category, and yet still results in a lifting of the traced symmetric monoidal structure. 

Let us first review Hopf monads on monoidal categories as introduced by Bruguières, Lack, and Virelizier \cite{bruguieres2011hopf}. Hopf monads are bimonads such that natural transformations that can be defined for any bimonad, called fusion operators, are natural isomorphisms.  

\begin{definition}\cite[Section 2.6 and 2.7]{bruguieres2011hopf} \label{Hopfmonaddef} Let $(T, \mu, \eta, m, m_I)$ be a bimonad on a monoidal category $(\mathbb{X}, \otimes, I, \alpha, \lambda, \rho)$. The \textbf{left fusion operator} of $(T, \mu, \eta, m, m_I)$ is the natural transformation
\[\mathsf{h}^l_{A,B}:  T(A \otimes T (B)) \to T(A) \otimes T(B)\]
defined as follows: 
\begin{align}
    \mathsf{h}^l_{A,B} :=   \xymatrixcolsep{5pc}\xymatrix{ T\left(A \otimes T(B) \right) \ar[r]^-{m_{A,T(B)}} & T(A) \otimes TT(B)\ar[r]^-{1_{T(A)} \otimes \mu_B} & T(A) \otimes T(B) }
\end{align}
and the \textbf{right fusion operator} of $(T, \mu, \eta, m, m_I)$ is the natural transformation
\[\mathsf{h}^r_{A,B}:  T(T(A) \otimes B) \to T(A) \otimes T(B)\]
defined as follows:
\begin{align}
    \mathsf{h}^r_{A,B} :=   \xymatrixcolsep{5pc}\xymatrix{T\left(T(A) \otimes B \right) \ar[r]^-{m_{T(A),B}} & TT(A) \otimes T(B)\ar[r]^-{\mu_A \otimes 1_{T(B)}} & T(A) \otimes T(B)}
\end{align}
A \textbf{Hopf monad} on a monoidal category $(\mathbb{X}, \otimes, I, \alpha, \lambda, \rho)$ is a bimonad $(T, \mu, \eta, m, m_I)$ on \\
$(\mathbb{X}, \otimes, I, \alpha, \lambda, \rho)$ such that the fusion operators $\mathsf{h}^l_{A,B}$ and $\mathsf{h}^r_{A,B}$ are natural isomorphisms, so in particular $T(A \otimes T(B)) \cong T(A) \otimes T(B) \cong T(T(A) \otimes B)$. 
    \end{definition}
    
A list of identities that the fusion operators satisfy can be found in \cite[Proposition 2.6]{bruguieres2011hopf}, and a list of identities that the inverses of the fusion operators satisfy can be found in \cite[Lemma 4.2]{hasegawa2018linear}. In particular, the fusion operators and their inverses are algebra morphisms. 

\begin{lemma}\label{fusiontalg} Let $(T, \mu, \eta, m, m_I)$ be a bimonad on a monoidal category $(\mathbb{X}, \otimes, I, \alpha, \lambda, \rho)$. Then: 
\begin{enumerate}[{\em (i)}]
\item $\mathsf{h}^l_{A,B}:  (T(A \otimes T (B)), \mu_{A \otimes T(B)}) \to (T(A), \mu_A) \otimes^T (T(B), \mu_B)$ is a $T$-algebra morphism; 
\item $\mathsf{h}^r_{A,B}: (T(T(A) \otimes B), \mu_{T(A) \otimes B}) \to (T(A), \mu_A) \otimes^T (T(B), \mu_B)$ is a $T$-algebra morphism. 
\end{enumerate}
Furthermore, if $(T, \mu, \eta, m, m_I)$ is a Hopf monad, then: 
\begin{enumerate}[{\em (i)}]
\item ${\mathsf{h}^l}^{-1}_{A,B}: (T(A), \mu_A) \otimes^T (T(B), \mu_B) \to (T(A \otimes T (B)), \mu_{A \otimes T(B)})$ is a $T$-algebra morphism; 
\item ${\mathsf{h}^r}^{-1}_{A,B}: (T(A), \mu_A) \otimes^T (T(B), \mu_B) \to (T(T(A) \otimes B), \mu_{T(A) \otimes B})$ is a $T$-algebra morphism. 
\end{enumerate}
\end{lemma}

We now turn our attention to the symmetric version of Hopf monads. In the symmetric setting, the fusion operators can be defined from one another using the symmetry, and similarly for their inverses. Therefore, the invertibility of one fusion operator implies the invertibility of the other.    

    \begin{definition}\label{symHopfmonaddef} A \textbf{symmetric Hopf monad} on a symmetric monoidal category $(\mathbb{X}, \!\otimes, \!I, \alpha, \lambda, \rho, \sigma)$ is a Hopf monad $(T, \mu, \eta, m, m_I)$ on the underlying monoidal category $(\mathbb{X}, \otimes, I, \alpha, \lambda, \rho)$ such that $(T, \mu, \eta, m, m_I)$ is also a symmetric bimonad on $(\mathbb{X}, \otimes, I, \alpha, \lambda, \rho, \sigma)$. 
    \end{definition}

\begin{lemma} Let $(T, \mu, \eta, m, m_I)$ be a symmetric bimonad on a symmetric monoidal category $(\mathbb{X}, \otimes, I, \alpha, \lambda, \rho, \sigma)$. Then the following are equivalent: 
\begin{enumerate}[{\em (i)}]
\item $(T, \mu, \eta, m, m_I)$ is a symmetric Hopf monad;
\item The left fusion operator $\mathsf{h}^l_{A,B}:  T(A \otimes T (B)) \to T(A) \otimes T(B)$ is a natural isomorphism;
\item The right fusion operator $\mathsf{h}^r_{A,B}: T(T(A) \otimes B) \to T(A) \otimes T(B)$ is a natural isomorphism. 
\end{enumerate}
\end{lemma}
\begin{proof} It suffices to show that $(ii) \Leftrightarrow (iii)$. First observe that for any symmetric bimonad, we can compute that: 
\begin{align*}
  \sigma_{T(A), T(B)} \circ \mathsf{h}^l_{A,B} &=~ \sigma_{T(A), T(B)} \circ (1_{T(A)} \otimes \mu_B) \circ m_{A, T(B)} \\
  &=~ (\mu_B \otimes 1_{T(A)}) \circ \sigma_{T(A), TT(B)} \circ m_{A, T(B)} \tag{Nat. of $\sigma$} \\
  &=~ (\mu_B \circ 1_{T(A)}) \circ m_{T(B), A} \circ T(\sigma_{A, T(B)}) \tag{Sym. Bimonad -- (\ref{smcendo2})} \\
  &=~ \mathsf{h}^r_{B,A} \circ T(\sigma_{A, T(B)})
\end{align*}
Therefore it follows that: 
\begin{align*}
\mathsf{h}^r_{B,A} =   \sigma_{T(A), T(B)} \circ \mathsf{h}^l_{A,B} \circ T(\sigma_{T(B), A}) && \mathsf{h}^l_{A,B} =  \sigma_{T(B), T(A)} \circ \mathsf{h}^r_{B,A} \circ T(\sigma_{A, T(B)})
\end{align*}
Since $\sigma$ is always an isomorphism, if $\mathsf{h}^l$ is an isomorphism then $\mathsf{h}^r$ is also an isomorphism, and if $\mathsf{h}^r$ is an isomorphism then $\mathsf{h}^l$ is also isomorphism. 
\end{proof}

A trace-coherent Hopf monad is a symmetric Hopf monad on traced symmetric monoidal category satisfying a compatibility axiom between the trace operator and the monad's functor using the left fusion operator\footnote{Of course, one could also reformulate this paper in terms of the right fusion operator.} and its inverse. So suppose that we have a traced symmetric monoidal category $(\mathbb{X}, \otimes, I, \alpha, \lambda, \rho, \sigma, \mathsf{Tr})$ and a symmetric Hopf monad $(T, \mu, \eta, m, m_I)$ on its underlying symmetric monoidal category. Consider a map of type:
\[f: A \otimes T(X) \to B \otimes T(X)\]
We can first take its trace and obtain a map of the following type: 
\[ \mathsf{Tr}^{T(X)}_{A,B}\left( f \right): A \to B \]
We can then apply the functor $T$ to the trace and obtain a map of type $T(A) \to T(B)$:
\[ T\left( \mathsf{Tr}^{T(X)}_{A,B}\left( f \right) \right) : T(A) \to T(B) \]
Alternatively, we could have first applied the functor $T$ to $f$ to obtain a map of type: 
\[ T(f): T(A \otimes T(X)) \to T(B \otimes T(X)) \]
We cannot take the trace of $T(f)$ when it is of this form. However, since $T$ is a Hopf monad, we can obtain a map of type $T(A) \otimes T(X) \to T(B) \otimes T(X)$ by post-composing by the left fusion operator and pre-composing by its inverse: 
 \[ \xymatrixcolsep{4pc}\xymatrixrowsep{1pc}\xymatrix{ T(A) \otimes T(X) \ar[r]^-{\mathsf{h}^{l^{-1}}_{A,X}  } & T(A \otimes T(X)) \ar[r]^-{T(f)} & T(B \otimes T(X)) \ar[r]^-{\mathsf{h}^l_{B,X}} &  T(B) \otimes T(X) } \]
 We can then take the trace of this composition to obtain a map of type $T(A) \to T(B)$:
 \[ \mathsf{Tr}^{T(X)}_{T(A),T(B)}\left( \mathsf{h}^l_{B,X} \circ  T(f) \circ \mathsf{h}^{l^{-1}}_{A,X}   \right) : T(A) \to T(B) \] 
To say that $T$ is trace-coherent is to require that these two maps of type $T(A) \to T(B)$ are equal. 

\begin{definition} A \textbf{trace-coherent Hopf monad} on a traced symmetric monoidal category \\
$(\mathbb{X}, \otimes, I, \alpha, \lambda, \rho, \sigma, \mathsf{Tr})$ is a symmetric Hopf monad $(T, \mu, \eta, m, m_I)$ on the underlying symmetric monoidal category $(\mathbb{X}, \otimes, I, \alpha, \lambda, \rho, \sigma)$ such that for every map $f: A \otimes T(X) \to B \otimes T(X)$ the following equality holds: 
 \begin{equation}\label{tracecoherent}\begin{gathered}
T\left( \mathsf{Tr}^{T(X)}_{A,B}\left( f \right) \right) =  \mathsf{Tr}^{T(X)}_{T(A),T(B)}\left( \mathsf{h}^l_{B,X} \circ  T(f) \circ \mathsf{h}^{l^{-1}}_{A,X}   \right) 
 \end{gathered}\end  {equation}  
\end{definition}

We now state and prove the main result of this paper: that trace-coherent Hopf monads are precisely symmetric Hopf monads that are also traced monads. Therefore, the Eilenberg-Moore category of a trace-coherent Hopf monad is again a traced symmetric monoidal category. 

\begin{theorem}\label{mainthm} Let $(\mathbb{X}, \otimes, I, \alpha, \lambda, \rho, \sigma, \mathsf{Tr})$ be a traced symmetric monoidal category and \\
$(T, \mu, \eta, m, m_I)$ a symmetric Hopf monad on the underlying symmetric monoidal category \\
$(\mathbb{X}, \otimes, I, \alpha, \lambda, \rho, \sigma)$. Then $(T, \mu, \eta, m, m_I)$ is a traced monad on $(\mathbb{X}, \otimes, I, \alpha, \lambda, \rho, \sigma, \mathsf{Tr})$ if and only if $(T, \mu, \eta, m, m_I)$ is a trace-coherent Hopf monad on $(\mathbb{X}, \otimes, I, \alpha, \lambda, \rho, \sigma, \mathsf{Tr})$.  
\end{theorem}
\begin{proof} $\Leftarrow$: Suppose that $(T, \mu, \eta, m, m_I)$ is traced-coherent. Let $(X,x)$, $(A,a)$, and $(B,b)$ be $T$-algebras and $f: (A,a) \otimes^T (X,x) \to (B,b) \otimes^T (X,x)$ be a $T$-algebra morphism, that is, the following equality holds:
 \begin{equation}\label{fTalg}\begin{gathered}
f \circ (a \otimes x) \circ m_{A,X}  = (b \otimes x) \circ m_{B,X} \circ T(f)
 \end{gathered}\end  {equation}  
First, observe that we can compute the following:
\begin{align*}
f \circ (a \otimes x) &=~ f \circ (a \otimes x) \circ \mathsf{h}^l_{A,X} \circ \mathsf{h}^{l^{-1}}_{A,X} \\
&=~ f \circ (a \otimes x) \circ (1_{T(A)} \otimes \mu_X) \circ m_{A,T(X)} \circ \mathsf{h}^{l^{-1}}_{A,X} \\ 
&=~ f \circ (a \otimes x) \circ (1_{T(A)} \otimes T(x)) \circ m_{A,T(X)} \circ \mathsf{h}^{l^{-1}}_{A,X} \tag{$(X,x)$ is a $T$-alg. -- (\ref{Talg})} \\
&=~ f \circ (a \otimes x) \circ m_{A,C} \circ T(1_A \otimes x) \circ \mathsf{h}^{l^{-1}}_{A,X} \tag{Nat. of $m$} \\
&=~ (b \otimes x) \circ m_{B,X} \circ T(f) \circ T(1_A \otimes x) \circ \mathsf{h}^{l^{-1}}_{A,X} \tag{$f$ is a $T$-alg. morph. -- (\ref{fTalg})} 
\end{align*}
So the following equality holds: 
 \begin{equation}\label{fTalg2}\begin{gathered}
f \circ (a \otimes x) = (b \otimes x) \circ m_{B,X} \circ T(f) \circ T(1_A \otimes x) \circ \mathsf{h}^{l^{-1}}_{A,X}
 \end{gathered}\end  {equation}  
The left fusion operator $\mathsf{h}^l$ satisfies the following identity \cite[Proposition 2.6]{bruguieres2011hopf}:
\begin{align} \label{h1}
\mathsf{h}^l_{B,X} \circ T(1_B \otimes \eta_X) = m_{B,X} 
\end{align}
and the inverse of the left fusion operator ${\mathsf{h}^l}^{-1}$ satisfies the following identity \cite[Lemma 4.2]{hasegawa2018linear}: 
\begin{align} \label{h3}
T(1_A \otimes \mu_X) \circ {\mathsf{h}^l}^{-1}_{A,T(X)} = {\mathsf{h}^l}^{-1}_{A,T(X)} \circ (1_{T(A)} \otimes \mu_X)    
\end{align}
Using the above identities, we can show that $\mathsf{Tr}^X_{A,B}(f)$ is a $T$-algebra morphism: 
\begin{align*}
 &\mathsf{Tr}^X_{A,B}(f) \circ a =~  \mathsf{Tr}^X_{T(A),B}\left( f \circ (a \otimes 1_X) \right) \tag{\textbf{[Tightening]} -- (\ref{tightening1})} \\
 &=~ \mathsf{Tr}^X_{T(A),B}\left( f \circ (a \otimes 1_X) \circ (1_{T(A)} \otimes x) \circ (1_{T(A)} \otimes \eta_X) \right) \tag{$(X,x)$ is a $T$-alg. -- (\ref{Talg})} \\
 &=~ \mathsf{Tr}^X_{T(A),B}\left( f \circ (a \otimes x) \circ (1_{T(A)} \otimes \eta_X) \right) \\ 
 &=~ \mathsf{Tr}^X_{T(A),B}\left( (b \otimes x) \circ m_{B,X} \circ T(f) \circ T(1_A \otimes x) \circ \mathsf{h}^{l^{-1}}_{A,X} \circ (1_{T(A)} \otimes \eta_X) \right) \tag{$f$ is a $T$-alg. morph. + (\ref{fTalg2})}  \\ 
  &=~ \mathsf{Tr}^X_{T(A),B}\left( (b \otimes 1_X) \circ (1_{T(B)} \otimes x) \circ m_{B,X} \circ T(f) \circ T(1_A \otimes x) \circ \mathsf{h}^{l^{-1}}_{A,X} \circ (1_{T(A)} \otimes \eta_X) \right)  \\
&=~ b \circ \mathsf{Tr}^X_{T(A),T(B)}\left(  (1_{T(B)} \otimes x) \circ m_{B,X} \circ T(f) \circ T(1_A \otimes x) \circ \mathsf{h}^{l^{-1}}_{A,X} \circ (1_{T(A)} \otimes \eta_X) \right) \tag{\textbf{[Tightening]} -- (\ref{tightening2})} \\ 
&=~ b \circ \mathsf{Tr}^{T(X)}_{T(A),T(B)}\left( m_{B,X} \circ T(f) \circ T(1_A \otimes x) \circ \mathsf{h}^{l^{-1}}_{A,X} \circ (1_{T(A)} \otimes \eta_X) \circ  (1_{T(A)} \otimes x) \right) \tag{\textbf{[Sliding]} -- (\ref{sliding})}\\
&=~ b \circ \mathsf{Tr}^{T(X)}_{T(A),T(B)}\left( m_{B,X} \circ T(f) \circ T(1_A \otimes x) \circ \mathsf{h}^{l^{-1}}_{A,X} \circ (1_{T(A)} \otimes T(x)) \circ (1_{T(A)} \otimes \eta_{T(X)})  \right) \tag{Nat. of $\eta$} \\ 
&=~ b \circ \mathsf{Tr}^{T(X)}_{T(A),T(B)}\left( m_{B,X} \circ T(f) \circ T(1_A \otimes x) \circ \mathsf{h}^{l^{-1}}_{A,X} \circ (T(1_A) \otimes T(x)) \circ (1_{T(A)} \otimes \eta_{T(X)})  \right) \tag{$T$ is a functor} \\ 
&=~ b \circ \mathsf{Tr}^{T(X)}_{T(A),T(B)}\left( m_{B,X} \circ T(f) \circ T(1_A \otimes x) \circ T\left(1_A \otimes T(x) \right) \circ {\mathsf{h}^l}^{-1}_{A,T(X)} \circ (1_{T(A)} \otimes \eta_{T(X)})  \right) \tag{Nat. of ${\mathsf{h}^l}^{-1}$} \\
&=~ b \circ \mathsf{Tr}^{T(X)}_{T(A),T(B)}\left( m_{B,X} \circ T(f) \circ T\left ( (1_A \otimes x) \circ \left(1_A \otimes T(x)\right) \right) \circ {\mathsf{h}^l}^{-1}_{A,T(X)} \circ (1_{T(A)} \otimes \eta_{T(X)})  \right) \tag{$T$ is a functor} \\ 
&=~ b \circ \mathsf{Tr}^{T(X)}_{T(A),T(B)}\left( m_{B,X} \circ T(f) \circ T\left ( (1_A \otimes x) \circ \left(1_A \otimes \mu_X \right) \right) \circ {\mathsf{h}^l}^{-1}_{A,T(X)} \circ (1_{T(A)} \otimes \eta_{T(X)})  \right) \tag{$(X,x)$ is a $T$-alg. -- (\ref{Talg})}\\
&=~ b \circ \mathsf{Tr}^{T(X)}_{T(A),T(B)}\left( m_{B,X} \circ T(f) \circ T(1_A \otimes x) \circ T\left(1_A \otimes \mu_X \right) \circ {\mathsf{h}^l}^{-1}_{A,T(X)} \circ (1_{T(A)} \otimes \eta_{T(X)})  \right) \tag{$T$ is a functor}\\
&=~ b \circ \mathsf{Tr}^{T(X)}_{T(A),T(B)}\left( m_{B,X} \circ T(f) \circ T(1_A \otimes x) \circ \mathsf{h}^{l^{-1}}_{A,X} \circ (1_{T(A)} \otimes \mu_X) \circ (1_{T(A)} \otimes \eta_{T(X)})  \right) \tag{Inv. Fus. Op. Identity -- (\ref{h3})} \\
&=~ b \circ \mathsf{Tr}^{T(X)}_{T(A),T(B)}\left( m_{B,X} \circ T(f) \circ T(1_A \otimes x) \circ \mathsf{h}^{l^{-1}}_{A,X}  \right) \tag{Monad Identity -- (\ref{monadeq})} \\ 
&=~ b \circ \mathsf{Tr}^{T(X)}_{T(A),T(B)}\left( \mathsf{h}^l_{B,X} \circ  T(1_B \otimes \eta_X) \circ T(f) \circ T(1_A \otimes x) \circ \mathsf{h}^{l^{-1}}_{A,X}  \right) \tag{Fusion Op. Identity -- (\ref{h1})} \\ 
&=~ b \circ \mathsf{Tr}^{T(X)}_{T(A),T(B)}\left( \mathsf{h}^l_{B,X} \circ  T\left( (1_B \otimes \eta_X) \circ f \circ (1_A \otimes x) \right) \circ \mathsf{h}^{l^{-1}}_{A,X}  \right) \tag{$T$ is a functor} \\ 
&=~ b \circ  T\left( \mathsf{Tr}^{T(X)}_{A,B}\left( (1_A \otimes \eta_A) \circ f \circ (1_A \otimes x) \right) \right) \tag{Trace-Coherent -- (\ref{tracecoherent})} \\ 
&=~ b \circ  T\left( \mathsf{Tr}^X_{A,B}\left( (1_A \otimes x) \circ (1_A \otimes \eta_A) \circ f \right) \right) \tag{\textbf{[Sliding]} -- (\ref{sliding})} \\
&=~    b \circ  T\left( \mathsf{Tr}^X_{A,B}(f) \right) \tag{$(X,x)$ is a $T$-alg. -- (\ref{Talg})}
\end{align*}
So $\mathsf{Tr}^X_{A,B}(f) \circ a =  b \circ  T\left( \mathsf{Tr}^X_{A,B}(f) \right)$, and thus $\mathsf{Tr}^X_{A,B}(f): (A,a) \to (B,b)$ is a $T$-algebra morphism. Therefore we conclude that $(T, \mu, \eta, m, m_I)$ is a traced monad. 

$\Rightarrow$: Suppose that $(T, \mu, \eta, m, m_I)$ is a traced monad. Let $f: A \otimes T(X) \to B \otimes T(X)$ be an arbitrary morphism in the base category. First observe that:
\[T(f): \left( T(A \otimes T(X)), \mu_{A \otimes T(X)} \right) \to  \left( T(B \otimes T(X)), \mu_{B \otimes T(X)} \right)\] 
is a $T$-algebra morphism, and also recall that the fusion operator and its inverse:
\begin{gather*}
    \mathsf{h}^{l^{-1}}_{A,X}: (T(A), \mu_A) \otimes^T (T(X), \mu_X) \to \left( T(A \otimes T(X)), \mu_{A \otimes T(X)} \right) \\
    \mathsf{h}^l_{B,X}: \left( T(B \otimes T(X)), \mu_{B \otimes T(X)} \right) \to (T(B), \mu_B) \otimes^T (T(X), \mu_X)
\end{gather*}
are also $T$-algebra morphisms (Lemma \ref{fusiontalg}). So then their composite:
\[\mathsf{h}^l_{B,X} \circ  T(f) \circ \mathsf{h}^{l^{-1}}_{A,X}: (T(A), \mu_A) \otimes^T (T(X), \mu_X) \to (T(B), \mu_B) \otimes^T (T(X), \mu_X)\] 
is a $T$-algebra morphism. Since $(T, \mu, \eta, m, m_I)$ is a traced monad, it follows that its trace 
\[\mathsf{Tr}^{T(X)}_{T(A),T(B)}\left( \mathsf{h}^l_{B,X} \circ  T(f) \circ \mathsf{h}^{l^{-1}}_{A,X}   \right): (T(A), \mu_A) \to (T(B), \mu_B)\]
is also a $T$-algebra morphism. Also note that:
\[T\left( \mathsf{Tr}^{T(X)}_{A,B}\left( f \right) \right): (T(A), \mu_A) \to (T(B), \mu_B)\]
is a $T$-algebra morphism by naturality of $\mu$. Now to show that $\mathsf{Tr}^{T(X)}_{T(A),T(B)}\left( \mathsf{h}^l_{B,X} \circ  T(f) \circ \mathsf{h}^{l^{-1}}_{A,X}   \right)$ and $T\left( \mathsf{Tr}^{T(X)}_{A,B}\left( f \right) \right)$ are equal, since they are $T$-algebra morphisms of the same type whose domain are free $T$-algebras, we will use Lemma \ref{Talglemma4}. As such, we will first show that they are equal when pre-composed with $\eta$. To do so, first observe the left fusion operator $\mathsf{h}^l$ satisfies the following identity \cite[Proposition 2.6]{bruguieres2011hopf}:
\begin{align} \label{h2}
\mathsf{h}^l_{B,X} \circ \eta_{B \otimes T(X)}= \eta_B \otimes 1_{T(X)}     
\end{align}
and that the inverse of the left fusion operator ${\mathsf{h}^l}^{-1}$ satisfies the following identity \cite[Lemma 4.2]{hasegawa2018linear}: 
\begin{align} \label{h4}
 \eta_{A \otimes T(X)}= \mathsf{h}^{l^{-1}}_{A,X} \circ \eta_A \otimes 1_{T(X)} 
\end{align}
Using the above identities, we compute: 
\begin{align*}
 \mathsf{Tr}^{T(X)}_{T(A),T(B)}\left( \mathsf{h}^l_{B,X} \circ  T(f) \circ \mathsf{h}^{l^{-1}}_{A,X}   \right) \circ \eta_{A} &=~ \mathsf{Tr}^{T(X)}_{A,T(B)}\left( \mathsf{h}^l_{B,X} \circ  T(f) \circ \mathsf{h}^{l^{-1}}_{A,X} \circ (\eta_{A} \otimes  1_{T(X)} ) \right)  \tag{\textbf{Tightening}} \\
 &=~ \mathsf{Tr}^{T(X)}_{A,T(B)}\left( \mathsf{h}^l_{B,X} \circ  T(f) \circ \eta_{A \otimes T(X)} \right) \tag{Inv. Fus. Op. Identity -- (\ref{h4})} \\
 &=~ \mathsf{Tr}^{T(X)}_{A,T(B)}\left( \mathsf{h}^l_{B,X} \circ \eta_{B \otimes T(X)} \circ f \right) \tag{Nat. of $\eta$} \\ 
 &=~ \mathsf{Tr}^{T(X)}_{A,B}\left( (\eta_B \otimes 1_{T(X)}) \circ  f \right) \tag{Fus. Op. Identity -- (\ref{h2})} \\
 &=~ \eta_B \circ \mathsf{Tr}^{T(X)}_{A,B}\left( f \right) \tag{\textbf{Tightening}} \\
 &=~ T\left( \mathsf{Tr}^{T(X)}_{A,B}\left( f \right) \right) \circ \eta_A \tag{Nat. of $\eta$} 
\end{align*}
So $\mathsf{Tr}^{T(X)}_{T(A),T(B)}\left( \mathsf{h}^l_{B,X} \circ  T(f) \circ \mathsf{h}^{l^{-1}}_{A,X}   \right) \circ \eta_{A} = T\left( \mathsf{Tr}^{T(X)}_{A,B}\left( f \right) \right) \circ \eta_A$. Then by Lemma \ref{Talglemma4}, it follows that:
\[\mathsf{Tr}^{T(X)}_{T(A),T(B)}\left( \mathsf{h}^l_{B,X} \circ  T(f) \circ \mathsf{h}^{l^{-1}}_{A,X}   \right)=T\left( \mathsf{Tr}^{T(X)}_{A,B}\left( f \right) \right)\]
So we conclude that $(T, \mu, \eta, m, m_I)$ is a trace-coherent Hopf monad. 
\end{proof}

\section{Hopf Algebras}\label{sec:hopfcase}

The canonical source of examples of (symmetric) Hopf monads are those induced by (cocommutative) Hopf monoids, and these are called \emph{representable} (symmetric) Hopf monads \cite[Section 5]{bruguieres2011hopf}. The algebras of the induced monad are precisely the modules of the Hopf monoid. In this section, we will prove that the induced symmetric Hopf monad of a cocommutative Hopf monoid is a trace-coherent Hopf monad. In other words, every representable symmetric Hopf monad is a trace-coherent Hopf monad. Therefore, for a cocommutative Hopf monoid in a traced symmetric monoidal category, its category of modules is again a traced symmetric monoidal category. 

The notion of a Hopf monoid generalizes the notion of Hopf algebras from classical algebra. So in particular a Hopf monoid is both a monoid and a comonoid that together form a bimonoid and comes equipped with an antipode. Any bimonoid induces a bimonad \cite[Example 7.5.4]{turaev2017monoidal}, while the antipode is required to build the inverse of the fusion operators. In fact, a bimonoid is a Hopf monoid if and only if its induced bimonad is a Hopf monad. To obtain a symmetric Hopf monad, one requires a cocommutative Hopf monoid, that is, a Hopf monoid whose comultiplication is cocommutative.  

To help with readability and avoid confusion, we highlight the Hopf monoid's wire in blue so that it can be easily distinguished from the other wires. 

\begin{definition} \cite[Section 6.2]{turaev2017monoidal} In a symmetric monoidal category $(\mathbb{X}, \otimes, I, \alpha, \lambda, \rho, \sigma)$, a \textbf{cocommutative Hopf monoid} is a sextuple $(H, \nabla, u, \Delta, e, S)$ consisting of an object $H$, a map $\nabla: H \otimes H \to H$, called the multiplication, and a map $u: I \to H$, called the unit, a map $\Delta: H \to H \otimes H$, called the comultiplication, and a map ${e: H \to I}$, called the counit, and a map ${S: H \to H}$, called the antipode, which are drawn in the graphical calculus as follows: 
\begin{align*}
\unitlength=0.9pt
\begin{picture}(30,30)
\put(15,25){\makebox(0,0){$\nabla:H\otimes H\rightarrow H$}}
\graybox{0}{0}{30}{20}{}%
\thicklines
{\Hcolor
\lineseg{0}{5}{10}{5}
\lineseg{0}{15}{10}{15}
\lineseg{20}{10}{30}{10}
}%
\mul{10}{5}
\end{picture}
&& 
\unitlength=0.8pt
\begin{picture}(30,30)
\put(15,25){\makebox(0,0){$u:I\rightarrow H$}}
\graybox{0}{0}{30}{20}{}%
\thicklines
{\Hcolor
\lineseg{20}{10}{30}{10}
}%
\mulunit{10}{5}
\end{picture}    
&&
\unitlength=0.9pt
\begin{picture}(30,30)
\put(15,25){\makebox(0,0){$\Delta:H\rightarrow H\otimes H$}}
\graybox{0}{0}{30}{20}{}%
\thicklines
{\Hcolor
\lineseg{20}{5}{30}{5}
\lineseg{20}{15}{30}{15}
\lineseg{0}{10}{10}{10}
}%
\comul{10}{5}
\end{picture}
&&
\unitlength=0.9pt
\begin{picture}(30,30)
\put(15,25){\makebox(0,0){$e:H\rightarrow I$}}
\graybox{0}{0}{30}{20}{}%
\thicklines
{\Hcolor
\lineseg{0}{10}{10}{10}
}%
\comulunit{10}{5}
\end{picture} && \unitlength=0.9pt
\begin{picture}(30,30)
\put(15,25){\makebox(0,0){$S:H\rightarrow H$}}
\graybox{0}{0}{30}{20}{}%
\thicklines
{\Hcolor
\lineseg{0}{10}{10}{10}
\lineseg{20}{10}{30}{10}
}%
\bluefbox{10}{5}{10}{10}{$S$}%
\end{picture}  
\end{align*}
and such that:
\begin{enumerate}[{\em (i)}]
\item $(H, \nabla, u)$ is a monoid \cite[Section 6.1.1]{turaev2017monoidal}, that is, the following equalities hold: 
 \begin{equation}\label{monoideq}\begin{gathered}
 \unitlength=0.9pt
\begin{picture}(50,30)
\graybox{0}{0}{50}{30}{}%
\thicklines
{\Hcolor
\lineseg{0}{5}{10}{5}
\lineseg{0}{15}{10}{15}
\lineseg{0}{25}{20}{25}
\lineseg{20}{25}{30}{20}
\lineseg{20}{10}{30}{10}
\lineseg{40}{15}{50}{15}
}%
\mul{10}{5}%
\mul{30}{10}%
\end{picture}
\begin{picture}(20,30)
\put(10,15){\makebox(0,0){$=$}}
\end{picture}
\begin{picture}(50,30)
\graybox{0}{0}{50}{30}{}%
\thicklines
{\Hcolor
\lineseg{0}{5}{20}{5}
\lineseg{0}{15}{10}{15}
\lineseg{0}{25}{10}{25}
\lineseg{20}{20}{30}{20}
\lineseg{20}{5}{30}{10}
\lineseg{40}{15}{50}{15}
}%
\mul{10}{15}%
\mul{30}{10}%
\end{picture} \quad \quad \quad \quad \quad \quad
 \unitlength=0.9pt
\begin{picture}(50,30)
\graybox{0}{0}{50}{30}{}%
\thicklines
{\Hcolor
\lineseg{0}{20}{30}{20}
\lineseg{20}{10}{30}{10}
\lineseg{40}{15}{50}{15}
}%
\mulunit{10}{5}%
\mul{30}{10}%
\end{picture}
\begin{picture}(20,30)
\put(10,15){\makebox(0,0){$=$}}
\end{picture}
\begin{picture}(50,30)
\graybox{0}{0}{50}{30}{}%
\thicklines
{\Hcolor
\lineseg{0}{15}{50}{15}
}%
\end{picture}
\begin{picture}(20,30)
\put(10,15){\makebox(0,0){$=$}}
\end{picture}
\begin{picture}(50,30)
\graybox{0}{0}{50}{30}{}%
\thicklines
{\Hcolor
\lineseg{20}{20}{30}{20}
\lineseg{0}{10}{30}{10}
\lineseg{40}{15}{50}{15}
}%
\mulunit{10}{15}%
\mul{30}{10}%
\end{picture}     
\end{gathered} \end{equation}
\item $(H, \Delta, e)$ is a cocommutative comonoid \cite[Section 6.1.2]{turaev2017monoidal}, that is, the following equalities hold: 
\begin{equation}\label{comonoideq}\begin{gathered}
 \unitlength=0.9pt
\begin{picture}(50,30)
\graybox{0}{0}{50}{30}{}%
\thicklines
{\Hcolor
\lineseg{40}{5}{50}{5}
\lineseg{0}{15}{10}{15}
\lineseg{30}{25}{50}{25}
\lineseg{20}{20}{30}{25}
\lineseg{20}{10}{30}{10}
\lineseg{40}{15}{50}{15}
}%
\comul{10}{10}%
\comul{30}{5}%
\end{picture}
\begin{picture}(20,30)
\put(10,15){\makebox(0,0){$=$}}
\end{picture}
\begin{picture}(50,30)
\graybox{0}{0}{50}{30}{}%
\thicklines
{\Hcolor
\lineseg{30}{5}{50}{5}
\lineseg{0}{15}{10}{15}
\lineseg{40}{25}{50}{25}
\lineseg{20}{20}{30}{20}
\lineseg{20}{10}{30}{5}
\lineseg{40}{15}{50}{15}
}%
\comul{10}{10}%
\comul{30}{15}%
\end{picture} \quad \quad \quad \quad 
 \unitlength=0.9pt
\begin{picture}(50,30)
\graybox{0}{0}{50}{30}{}%
\thicklines
{\Hcolor
\lineseg{20}{20}{50}{20}
\lineseg{20}{10}{30}{10}
\lineseg{0}{15}{10}{15}
}%
\comulunit{30}{5}%
\comul{10}{10}%
\end{picture}
\begin{picture}(20,30)
\put(10,15){\makebox(0,0){$=$}}
\end{picture}
\begin{picture}(50,30)
\graybox{0}{0}{50}{30}{}%
\thicklines
{\Hcolor
\lineseg{0}{15}{50}{15}
}%
\end{picture}
\begin{picture}(20,30)
\put(10,15){\makebox(0,0){$=$}}
\end{picture}
\begin{picture}(50,30)
\graybox{0}{0}{50}{30}{}%
\thicklines
{\Hcolor
\lineseg{20}{20}{30}{20}
\lineseg{20}{10}{50}{10}
\lineseg{0}{15}{10}{15}
}%
\comulunit{30}{15}%
\comul{10}{10}%
\end{picture} \\
\unitlength=0.9pt
\begin{picture}(50,30)
\graybox{0}{0}{50}{30}{}%
\thicklines
{\Hcolor
\lineseg{0}{15}{10}{15}
\lineseg{20}{20}{30}{20}
\lineseg{20}{10}{30}{10}
\lineseg{30}{20}{40}{10}
\lineseg{30}{10}{40}{20}
\lineseg{40}{20}{50}{20}
\lineseg{40}{10}{50}{10}
}%
\comul{10}{10}%
\end{picture}
\begin{picture}(20,30)
\put(10,15){\makebox(0,0){$=$}}
\end{picture}
\begin{picture}(30,30)
\graybox{0}{0}{30}{30}{}%
\thicklines
{\Hcolor
\lineseg{0}{15}{10}{15}
\lineseg{20}{20}{30}{20}
\lineseg{20}{10}{30}{10}
}%
\comul{10}{10}%
\end{picture}    
\end{gathered} \end{equation}
\item $(H, \nabla, u, \Delta, e)$ is also a bimonoid \cite[Section 6.1.3]{turaev2017monoidal}, that is, the following equalities hold:
 \begin{equation}\label{bimonoideq}\begin{gathered}
 \unitlength=0.9pt
\begin{picture}(50,40)
\graybox{0}{0}{50}{40}{}%
\thicklines
{\Hcolor
\lineseg{0}{25}{10}{25}
\lineseg{0}{15}{10}{15}
\lineseg{20}{20}{30}{20}
\lineseg{40}{25}{50}{25}
\lineseg{40}{15}{50}{15}
}%
\mul{10}{15}%
\comul{30}{15}%
\end{picture}
\begin{picture}(20,40)
\put(10,20){\makebox(0,0){$=$}}
\end{picture}
\begin{picture}(60,40)
\graybox{0}{0}{60}{40}{}%
\thicklines
{\Hcolor
\lineseg{0}{30}{10}{30}
\lineseg{0}{10}{10}{10}
\lineseg{20}{35}{40}{35}
\lineseg{20}{25}{25}{25}
\lineseg{20}{15}{25}{15}
\lineseg{20}{5}{40}{5}
\lineseg{25}{25}{35}{15}
\lineseg{25}{15}{35}{25}
\lineseg{35}{25}{40}{25}
\lineseg{35}{15}{40}{15}
\lineseg{50}{30}{60}{30}
\lineseg{50}{10}{60}{10}
}%
\comul{10}{5}%
\comul{10}{25}%
\mul{40}{5}%
\mul{40}{25}%
\end{picture} \quad \quad \quad \quad \quad \begin{picture}(50,40)
\graybox{0}{0}{50}{40}{}%
\thicklines
{\Hcolor\lineseg{20}{20}{30}{20}}%
\comulunit{30}{15}%
\mulunit{10}{15}%
\end{picture}
\begin{picture}(20,40)
\put(10,20){\makebox(0,0){$=$}}
\end{picture}
\begin{picture}(30,40)
\graybox{0}{0}{30}{40}{}%
\thicklines
\end{picture} \\
\unitlength=.9pt
\begin{picture}(50,40)
\graybox{0}{0}{50}{40}{}%
\thicklines
{\Hcolor
\lineseg{20}{20}{30}{20}
\lineseg{40}{25}{50}{25}
\lineseg{40}{15}{50}{15}
}%
\mulunit{10}{15}%
\comul{30}{15}%
\end{picture}
\begin{picture}(20,40)
\put(10,20){\makebox(0,0){$=$}}
\end{picture}
\begin{picture}(30,40)
\graybox{0}{0}{30}{40}{}%
\thicklines
{\Hcolor
\lineseg{20}{30}{30}{30}
\lineseg{20}{10}{30}{10}
}%
\mulunit{10}{5}%
\mulunit{10}{25}%
\end{picture}
\quad \quad \quad \quad \quad 
\unitlength=.9pt
\begin{picture}(50,40)
\graybox{0}{0}{50}{40}{}%
\thicklines
{\Hcolor
\lineseg{20}{20}{30}{20}
\lineseg{0}{25}{10}{25}
\lineseg{0}{15}{10}{15}
}%
\comulunit{30}{15}%
\mul{10}{15}%
\end{picture}
\begin{picture}(20,40)
\put(10,20){\makebox(0,0){$=$}}
\end{picture}
\begin{picture}(30,40)
\graybox{0}{0}{30}{40}{}%
\thicklines
{\Hcolor
\lineseg{0}{30}{10}{30}
\lineseg{0}{10}{10}{10}
}%
\comulunit{10}{5}%
\comulunit{10}{25}%
\end{picture}
\end{gathered} \end{equation}
\item The following equalities also hold: 
 \begin{equation}\label{hopfeq}\begin{gathered}
    \unitlength=.9pt
\begin{picture}(70,30)
\graybox{0}{0}{70}{30}{}%
\thicklines
{\Hcolor
\lineseg{0}{15}{10}{15}
\lineseg{20}{10}{50}{10}
\lineseg{20}{20}{50}{20}
\lineseg{60}{15}{70}{15}
}%
\comul{10}{10}%
\bluefbox{30}{5}{10}{10}{$S$}%
\mul{50}{10}%
\end{picture}
\begin{picture}(20,40)
\put(10,15){\makebox(0,0){$=$}}
\end{picture}
\begin{picture}(50,30)
\graybox{0}{0}{50}{30}{}%
\thicklines
{\Hcolor
\lineseg{0}{15}{10}{15}
\lineseg{40}{15}{50}{15}
}%
\comulunit{10}{10}%
\mulunit{30}{10}%
\end{picture}
\begin{picture}(20,40)
\put(10,15){\makebox(0,0){$=$}}
\end{picture}
\begin{picture}(70,30)
\graybox{0}{0}{70}{30}{}%
\thicklines
{\Hcolor 
\lineseg{0}{15}{10}{15}
\lineseg{20}{10}{50}{10}
\lineseg{20}{20}{50}{20}
\lineseg{60}{15}{70}{15}
}%
\comul{10}{10}%
\bluefbox{30}{15}{10}{10}{$S$}%
\mul{50}{10}%
\end{picture}
 \end{gathered}\end  {equation}  
\end{enumerate}
\end{definition}

\begin{lemma}\label{hopfmonoidhopfmonad} \cite[Example 2.10]{bruguieres2011hopf} Let $(H, \nabla, u, \Delta, e, S)$ be a cocommutative Hopf monoid in a symmetric monoidal category $(\mathbb{X}, \otimes, I, \alpha, \lambda, \rho, \sigma)$. Then $(H \otimes -, \mu^\nabla, \eta^u, m^\Delta, m_I^e)$ is a symmetric Hopf monad where:
\begin{enumerate}[{\em (i)}]
\item The endofunctor $H \otimes -: \mathbb{X} \to \mathbb{X}$ is defined on objects as $(H \otimes -)(A) = H \otimes A$ and on maps as $(H \otimes -)(f) = 1_H \otimes f$;
\item The natural transformations $\mu_A^\nabla: H \otimes (H \otimes A) \to H \otimes A$ and $\eta^u_A: A \to H \otimes A$ are respectively defined as follows: 
\begin{align*}
\begin{array}[c]{c}\mu_A^\nabla  := \end{array}
\begin{array}[c]{c}
\unitlength=0.8pt
\begin{picture}(50,50)
\graybox{0}{0}{50}{40}{}%
\thicklines
\lineseg{0}{30}{50}{30}
{\Hcolor
\lineseg{0}{20}{20}{20}
\lineseg{0}{10}{20}{10}
\lineseg{30}{15}{50}{15}}%
\mul{20}{10}%
\put(-5,30){\makebox(0,0){\scriptsize$A$}}
\put(-5,20){\makebox(0,0){\scriptsize$H$}}
\put(-5,10){\makebox(0,0){\scriptsize$H$}}
\put(55,30){\makebox(0,0){\scriptsize$A$}}
\put(55,15){\makebox(0,0){\scriptsize$H$}}
\end{picture}
\end{array}
&& 
\begin{array}[c]{c}\eta_A^\nabla  := \end{array}
\begin{array}[c]{c}
\unitlength=0.8pt
\begin{picture}(50,50)
\graybox{0}{0}{50}{40}{}%
\thicklines
\lineseg{0}{30}{50}{30}
{\Hcolor\lineseg{30}{15}{50}{15}}%
\mulunit{20}{10}%
\put(-5,30){\makebox(0,0){\scriptsize$A$}}
\put(55,30){\makebox(0,0){\scriptsize$A$}}
\put(55,15){\makebox(0,0){\scriptsize$H$}}
\end{picture} 
\end{array}
\end{align*}
\item The natural transformation $m^\Delta_{A,B}: H \otimes (A \otimes B) \to (H \otimes A) \otimes (H \otimes B)$ and the map $m^e_I: H \otimes I \to I$ are respectively defined as follows: 
\begin{align*}
\begin{array}[c]{c}m^\Delta_{A,B} := \end{array}
\begin{array}[c]{c}
\unitlength=0.8pt
\begin{picture}(50,50)
\graybox{0}{0}{50}{40}{}%
\thicklines
{\Hcolor
\lineseg{0}{35}{50}{35}
\lineseg{0}{10}{10}{10}
\lineseg{40}{25}{50}{25}
\lineseg{30}{15}{40}{25}
\lineseg{20}{15}{30}{15}
\lineseg{20}{5}{50}{5}
}%
\lineseg{0}{25}{30}{25}
\lineseg{30}{25}{40}{15}
\lineseg{40}{15}{50}{15}
\comul{10}{5}%
\put(-5,35){\makebox(0,0){\scriptsize$B$}}
\put(-5,25){\makebox(0,0){\scriptsize$A$}}
\put(-5,10){\makebox(0,0){\scriptsize$H$}}
\put(55,35){\makebox(0,0){\scriptsize$B$}}
\put(55,25){\makebox(0,0){\scriptsize$H$}}
\put(55,15){\makebox(0,0){\scriptsize$A$}}
\put(55,5){\makebox(0,0){\scriptsize$H$}}
\end{picture}
\end{array}
&& 
\begin{array}[c]{c}m^e_I  := \end{array}
\begin{array}[c]{c}
\unitlength=0.8pt
\begin{picture}(50,50)
\graybox{0}{0}{50}{40}{}%
\thicklines
{\Hcolor\lineseg{0}{20}{20}{20}}%
\comulunit{20}{15}%
\put(-5,20){\makebox(0,0){\scriptsize$H$}}
\end{picture}
\end{array}
\end{align*}
\item The left fusion operator ${\mathsf{h}^l_{A,B}}: H \otimes \left( A \otimes \left( H \otimes B \right) \right) \to (H \otimes A) \otimes (H \otimes B)$ and its inverse are respectively given as follows: 
\begin{align*}
    \begin{array}[c]{c}{\mathsf{h}^l_{A,B}} := \end{array} \begin{array}[c]{c}\unitlength=.9pt
\begin{picture}(50,60)
\graybox{0}{0}{50}{50}{}%
\thicklines
{\Hcolor
\lineseg{0}{35}{30}{35}
\lineseg{0}{10}{10}{10}
\lineseg{20}{15}{30}{25}
\lineseg{20}{5}{50}{5}
\lineseg{40}{30}{50}{30}
}%
\lineseg{0}{45}{50}{45}
\lineseg{0}{25}{20}{25}
\lineseg{20}{25}{30}{15}
\lineseg{30}{15}{50}{15}
\comul{10}{5}%
\mul{30}{25}%
\put(-5,45){\makebox(0,0){\scriptsize$B$}}
\put(-5,35){\makebox(0,0){\scriptsize$H$}}
\put(-5,25){\makebox(0,0){\scriptsize$A$}}
\put(-5,10){\makebox(0,0){\scriptsize$H$}}
\put(55,45){\makebox(0,0){\scriptsize$B$}}
\put(55,30){\makebox(0,0){\scriptsize$H$}}
\put(55,15){\makebox(0,0){\scriptsize$A$}}
\put(55,5){\makebox(0,0){\scriptsize$H$}}
\end{picture}\end{array} && \begin{array}[c]{c}{\mathsf{h}^l}^{-1}_{A,B} := \end{array} \begin{array}[c]{c}\unitlength=.9pt
\begin{picture}(70,60)
\graybox{0}{0}{70}{50}{}%
\thicklines
{\Hcolor
\lineseg{0}{35}{50}{35}
\lineseg{0}{10}{10}{10}
\lineseg{20}{15}{30}{25}
\lineseg{20}{5}{70}{5}
\lineseg{40}{25}{50}{25}
\lineseg{60}{30}{70}{30}
}%
\lineseg{0}{45}{70}{45}
\lineseg{0}{25}{20}{25}
\lineseg{20}{25}{30}{15}
\lineseg{30}{15}{70}{15}
\bluefbox{30}{20}{10}{10}{$S$}%
\comul{10}{5}%
\mul{50}{25}%
\put(-5,45){\makebox(0,0){\scriptsize$B$}}
\put(-5,35){\makebox(0,0){\scriptsize$H$}}
\put(-5,25){\makebox(0,0){\scriptsize$A$}}
\put(-5,10){\makebox(0,0){\scriptsize$H$}}
\put(75,45){\makebox(0,0){\scriptsize$B$}}
\put(75,30){\makebox(0,0){\scriptsize$H$}}
\put(75,15){\makebox(0,0){\scriptsize$A$}}
\put(75,5){\makebox(0,0){\scriptsize$H$}}
\end{picture}\end{array}
\end{align*}
\end{enumerate}
\end{lemma}

We now prove the main result of this section: 

\begin{proposition} \label{prop:hopfalgtrace} Let $(\mathbb{X}, \otimes, I, \alpha, \lambda, \rho, \sigma, \mathsf{Tr})$ be a traced symmetric monoidal category and\\
$(H, \nabla, u, \Delta, e, S)$ be a cocommutative Hopf monoid in the underlying symmetric monoidal category $(\mathbb{X}, \otimes, I, \alpha, \lambda, \rho, \sigma)$. Then the induced symmetric Hopf monad $(H \otimes -, \mu^\nabla, \eta^u, m^\Delta, m_I^e)$ is a trace-coherent Hopf monad on $(\mathbb{X}, \otimes, I, \alpha, \lambda, \rho, \sigma, \mathsf{Tr})$, that is, for any map $f: A \otimes (H \otimes X) \to B \otimes (H \otimes X)$ the following equality holds (which is equation (\ref{tracecoherent}) drawn in the graphical calculus): 
\begin{align*}
\unitlength=.9pt
\begin{picture}(100,90)
\rgbbox{0}{0}{100}{90}{.9}{.9}{.9}{}%
\thicklines
\trace{40}{60}{20}{10}
{\Hcolor\trace{40}{60}{20}{25}}%
\rgbfbox{40}{15}{20}{40}{.9}{.8}{0}{$f$}%
\idline{0}{20}{40} \idline{60}{20}{40}
{\Hcolor\idline{0}{5}{100}}%
\end{picture}
\begin{picture}(20,90)
\put(10,40){\makebox(0,0){$=$}}
\end{picture}
\begin{picture}(230,90)
\rgbbox{0}{0}{230}{90}{.9}{.9}{.9}{}%
\thicklines
\trace{40}{60}{150}{10}
{\Hcolor\trace{40}{60}{150}{20}}%
\comul{40}{10}%
\bluefbox{70}{25}{10}{10}{$S$}%
\mul{90}{30}%
\rgbfbox{110}{15}{20}{40}{.9}{.8}{0}{$f$}%
\comul{140}{5}%
\mul{170}{25}%
{\Hcolor
\idline{40}{40}{50}
\qbezier(55,20)(55,20)(65,30)
\idline{65}{30}{5} \idline{80}{30}{10} \idline{100}{35}{10}
\idline{130}{35}{40} 
\idline{0}{5}{30} \put(30,5){\line(1,1){10}}\idline{50}{20}{5} 
\idline{50}{10}{90}
\idline{150}{15}{5} \put(155,15){\line(1,1){10}} \idline{165}{25}{5}
\idline{150}{5}{80}
\put(180,30){\line(1,1){10}}
}%
\idline{40}{50}{70} \idline{130}{50}{60}
\idline{0}{20}{30} \qbezier(30,20)(30,20)(35,25) \idline{35}{25}{20}
\qbezier(55,25)(55,25)(60,20)
\idline{60}{20}{50} \idline{130}{20}{25}
\qbezier(155,20)(155,20)(160,15) 
\idline{160}{15}{25} 
\qbezier(185,15)(185,15)(190,20)
\idline{190}{20}{40}
\end{picture}    
\end{align*}
Therefore, $(H \otimes -, \mu^\nabla, \eta^u, m^\Delta, m_I^e)$ is a traced monad on $(\mathbb{X}, \otimes, I, \alpha, \lambda, \rho, \sigma, \mathsf{Tr})$. 
\end{proposition}
\begin{proof} We leave it as an exercise for the reader to check that the above equality is indeed the trace-coherent axiom drawn in the graphical calculus. To prove it, using the graphical calculus we compute: 
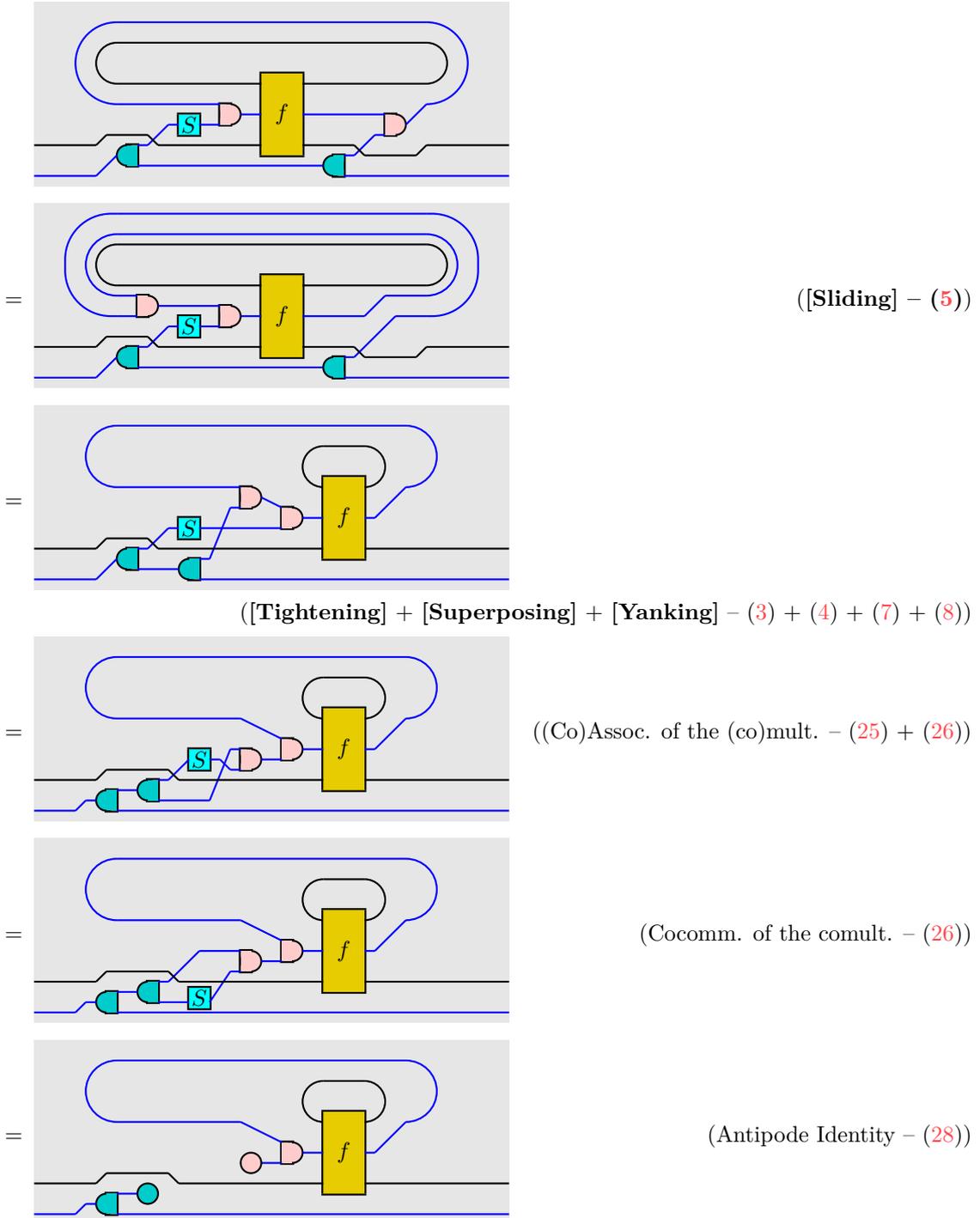
\begin{align*}
&\begin{array}[c]{c} \unitlength=.9pt\begin{picture}(230,90)
\rgbbox{0}{0}{230}{90}{.9}{.9}{.9}{}%
\thicklines
\trace{40}{60}{150}{10}
{\Hcolor\trace{40}{60}{150}{20}}%
\comul{40}{10}%
\bluefbox{70}{25}{10}{10}{$S$}%
\mul{90}{30}%
\rgbfbox{110}{15}{20}{40}{.9}{.8}{0}{$f$}%
\comul{140}{5}%
\mul{170}{25}%
{\Hcolor 
\idline{40}{40}{50}
\qbezier(55,20)(55,20)(65,30)
\idline{65}{30}{5} \idline{80}{30}{10} \idline{100}{35}{10}
\idline{130}{35}{40} 
\idline{0}{5}{30} \put(30,5){\line(1,1){10}}\idline{50}{20}{5} 
\idline{50}{10}{90}
\idline{150}{15}{5} \put(155,15){\line(1,1){10}} \idline{165}{25}{5}
\idline{150}{5}{80}
\put(180,30){\line(1,1){10}}
}%
\idline{40}{50}{70} \idline{130}{50}{60}
\idline{0}{20}{30} \qbezier(30,20)(30,20)(35,25) \idline{35}{25}{20}
\qbezier(55,25)(55,25)(60,20)
\idline{60}{20}{50} \idline{130}{20}{25}
\qbezier(155,20)(155,20)(160,15) 
\idline{160}{15}{25} 
\qbezier(185,15)(185,15)(190,20)
\idline{190}{20}{40}
\end{picture} 
 \end{array} \\
 = & \begin{array}[c]{c} \unitlength=.9pt \begin{picture}(230,90)
\rgbbox{0}{0}{230}{90}{.9}{.9}{.9}{}%
\thicklines
\trace{40}{60}{150}{10}
{\Hcolor
\trace{40}{60}{150}{15}
\trace{40}{60}{150}{25}
}%
\comul{40}{10}%
\bluefbox{70}{25}{10}{10}{$S$}%
\mul{90}{30}%
\rgbfbox{110}{15}{20}{40}{.9}{.8}{0}{$f$}%
\comul{140}{5}%
\mul{50}{35}%
{\Hcolor 
\idline{40}{45}{10}
\idline{40}{35}{10}
\idline{60}{40}{30}
\qbezier(55,20)(55,20)(65,30)
\idline{65}{30}{5} \idline{80}{30}{10} \idline{100}{35}{10}
\idline{130}{35}{30} \put(160,35){\line(1,1){10}} \idline{170}{45}{20}
\idline{0}{5}{30} \put(30,5){\line(1,1){10}}\idline{50}{20}{5}  
\idline{50}{10}{90}
\idline{150}{15}{5} \put(155,15){\line(1,1){20}} \idline{175}{35}{15}
\idline{150}{5}{80}
}%
\idline{40}{50}{70} \idline{130}{50}{60}
\idline{0}{20}{30} \qbezier(30,20)(30,20)(35,25) \idline{35}{25}{20}
\qbezier(55,25)(55,25)(60,20) \idline{60}{20}{50}
\idline{130}{20}{25} \qbezier(155,20)(155,20)(160,15) 
\idline{160}{15}{25} 
\qbezier(185,15)(185,15)(190,20) \idline{190}{20}{40}
\end{picture}  \end{array} \tag{\textbf{[Sliding] -- (\ref{sliding})}} \\
 = & \begin{array}[c]{c} \unitlength=.9pt \begin{picture}(230,90)
\rgbbox{0}{0}{230}{90}{.9}{.9}{.9}{}%
\thicklines
\trace{140}{60}{20}{10}
{\Hcolor\trace{40}{65}{140}{15}}%
\comul{40}{10}%
\bluefbox{70}{25}{10}{10}{$S$}%
\mul{120}{30}%
\rgbfbox{140}{15}{20}{40}{.9}{.8}{0}{$f$}%
\comul{70}{5}%
\mul{100}{40}%
{\Hcolor
\idline{40}{50}{60}
\idline{0}{5}{30} \put(30,5){\line(1,1){10}}
\idline{50}{20}{5}\put(55,20){\line(1,1){10}}\idline{65}{30}{5}
\idline{50}{10}{20}
\idline{80}{15}{5} \qbezier(85,15)(85,15)(95,40) \idline{95}{40}{5}
\idline{80}{30}{40} 
\qbezier(110,45)(110,45)(120,40)
\idline{130}{35}{10}
\idline{160}{35}{5} \qbezier(165,35)(165,35)(180,50)
\idline{80}{5}{150}
}%
\idline{0}{20}{30} \qbezier(30,20)(30,20)(35,25) \idline{35}{25}{20}
\qbezier(55,25)(55,25)(60,20) \idline{60}{20}{80}
\idline{160}{20}{70}
\end{picture}  \end{array} \tag{\textbf{[Tightening]} + \textbf{[Superposing]} + \textbf{[Yanking]} -- (\ref{tightening1}) + (\ref{tightening2}) + (\ref{superposing}) + (\ref{yanking})} \\
 = & \begin{array}[c]{c} \unitlength=.9pt \begin{picture}(230,90)
\rgbbox{0}{0}{230}{90}{.9}{.9}{.9}{}%
\thicklines
\trace{140}{60}{20}{10}
{\Hcolor\trace{40}{65}{140}{15}}%
\comul{30}{5}%
\bluefbox{75}{25}{10}{10}{$S$}%
\mul{120}{30}%
\rgbfbox{140}{15}{20}{40}{.9}{.8}{0}{$f$}%
\comul{50}{10}%
\mul{100}{25}%
{\Hcolor
\idline{40}{50}{60}
\idline{0}{5}{20} \qbezier(20,5)(20,5)(25,10) \idline{25}{10}{5}
\idline{60}{20}{5}\put(65,20){\line(1,1){10}}
\idline{60}{10}{25}
\qbezier(85,10)(85,10)(95,35) \idline{95}{35}{5}
\idline{85}{30}{5} \qbezier(90,30)(90,30)(95,25) \idline{95}{25}{5}
\qbezier(100,50)(100,50)(120,40)
\idline{110}{30}{10}
\idline{130}{35}{10}
\idline{160}{35}{5} \qbezier(165,35)(165,35)(180,50)
\idline{40}{5}{190}
\idline{40}{15}{10}
}%
\idline{0}{20}{30} \qbezier(30,20)(30,20)(35,25) \idline{35}{25}{30}
\qbezier(65,25)(65,25)(70,20) \idline{70}{20}{70}
\idline{160}{20}{70}
\end{picture} \end{array} \tag{(Co)Assoc. of the (co)mult. -- (\ref{monoideq}) + (\ref{comonoideq})}\\
 = & \begin{array}[c]{c} \unitlength=.9pt \begin{picture}(230,90)
\rgbbox{0}{0}{230}{90}{.9}{.9}{.9}{}%
\thicklines
\trace{140}{60}{20}{10}
{\Hcolor\trace{40}{65}{140}{15}}%
\comul{30}{5}%
\bluefbox{75}{7}{10}{10}{$S$}%
\mul{120}{30}%
\rgbfbox{140}{15}{20}{40}{.9}{.8}{0}{$f$}%
\comul{50}{10}%
\mul{100}{25}%
{\Hcolor
\idline{40}{50}{60}
\idline{0}{5}{20} \qbezier(20,5)(20,5)(25,10) \idline{25}{10}{5}
\idline{60}{20}{5}\put(65,20){\line(1,1){15}}
\idline{60}{10}{15}
\qbezier(85,10)(85,10)(95,25) \idline{80}{35}{20}
 \idline{95}{25}{5}
\qbezier(100,50)(100,50)(120,40)
\idline{110}{30}{10}
\idline{130}{35}{10}
\idline{160}{35}{5} \qbezier(165,35)(165,35)(180,50)
\idline{40}{5}{190}
\idline{40}{15}{10}
}%
\idline{0}{20}{30} \qbezier(30,20)(30,20)(35,25) \idline{35}{25}{30}
\qbezier(65,25)(65,25)(70,20) \idline{70}{20}{70}
\idline{160}{20}{70}
\end{picture} \end{array} \tag{Cocomm. of the comult. -- (\ref{comonoideq})} \\
 = & \begin{array}[c]{c} \unitlength=.9pt \begin{picture}(230,90)
\rgbbox{0}{0}{230}{90}{.9}{.9}{.9}{}%
\thicklines
\trace{140}{60}{20}{10}
{\Hcolor\trace{40}{65}{140}{15}}%
\comul{30}{5}%
\mul{120}{30}%
\rgbfbox{140}{15}{20}{40}{.9}{.8}{0}{$f$}%
{\Hcolor 
\idline{40}{50}{60}
\idline{0}{5}{20} \qbezier(20,5)(20,5)(25,10) \idline{25}{10}{5}
\qbezier(100,50)(100,50)(120,40)
\idline{110}{30}{10}
\idline{130}{35}{10}
\idline{160}{35}{5} \qbezier(165,35)(165,35)(180,50)
\idline{40}{5}{190}
\idline{40}{15}{10}
}%
\idline{0}{20}{30} \qbezier(30,20)(30,20)(35,25) \idline{35}{25}{30}
\qbezier(65,25)(65,25)(70,20) \idline{70}{20}{70}
\idline{160}{20}{70}
\mulunit{100}{25}%
\comulunit{50}{10}%
\end{picture} \end{array} \tag{Antipode Identity -- (\ref{hopfeq})} \\
 = & \begin{array}[c]{c} \unitlength=.9pt \begin{picture}(100,90)
\rgbbox{0}{0}{100}{90}{.9}{.9}{.9}{}%
\thicklines
\trace{40}{60}{20}{10}
{\Hcolor\trace{40}{60}{20}{25}}%
\rgbfbox{40}{15}{20}{40}{.9}{.8}{0}{$f$}%
\idline{0}{20}{40} \idline{60}{20}{40}
{\Hcolor\idline{0}{5}{100}}%
\end{picture} \end{array} \tag{(Co)unit identities -- (\ref{monoideq}) + (\ref{comonoideq})}
\end{align*}
So the desired equality holds. As such, we conclude that $(H \otimes -, \mu^\nabla, \eta^u, m^\Delta, m_I^e)$ is trace-coherent, and therefore also a traced monad. 
\end{proof}

Therefore the Eilenberg-Moore category of the induced monad of a cocommutative Hopf monoid will be a traced symmetric monoidal category. However, as mentioned, the algebras in this case correspond precisely to modules. 

\begin{definition} \cite[Section 6.1.1]{turaev2017monoidal} Let $(H, \nabla, u, \Delta, e, S)$ be a cocommutative Hopf monoid in a symmetric monoidal category $(\mathbb{X}, \otimes, I, \alpha, \lambda, \rho, \sigma)$. A (left) \textbf{module over $(H, \nabla, u, \Delta, e, S)$}, or simply a \textbf{$H$-module}, is a pair $(A, a)$ consisting of an object $A$ and a map ${a: H \otimes A \to A}$, called the action, drawn in the graphical calculus as:
\[ \unitlength=1pt
\begin{picture}(30,40)
\put(15,35){\makebox(0,0){$a: H \otimes A \to A$}}
\graybox{0}{0}{30}{30}{}%
\thicklines
\lineseg{0}{20}{10}{20}
{\Hcolor\lineseg{0}{10}{10}{10}}%
\lineseg{15}{15}{30}{15}
\greenfbox{10}{7.5}{10}{15}{$a$}%
\put(-5,20){\makebox(0,0){\scriptsize$A$}}
\put(-5,10){\makebox(0,0){\scriptsize$H$}}
\put(35,15){\makebox(0,0){\scriptsize$A$}}
\end{picture} \]
such that the following equalities hold: 
 \begin{equation}\label{moduleeq}\begin{gathered}
\unitlength=1pt
\begin{picture}(50,30)
\graybox{0}{0}{50}{30}{}%
\thicklines
\lineseg{0}{20}{30}{20}
{\Hcolor\lineseg{20}{10}{30}{10}}%
\lineseg{40}{15}{50}{15}
\mulunit{10}{5}%
\greenfbox{30}{7.5}{10}{15}{$a$}%
\end{picture}
\begin{picture}(20,30)
\put(10,15){\makebox(0,0){$=$}}
\end{picture}
\begin{picture}(50,30)
\graybox{0}{0}{50}{30}{}%
\thicklines
\lineseg{0}{15}{50}{15}
\end{picture}
\quad \quad \quad \quad \quad 
\begin{picture}(50,30)
\graybox{0}{0}{50}{30}{}%
\thicklines
\lineseg{0}{25}{30}{20}
{\Hcolor
\lineseg{0}{15}{10}{15}
\lineseg{0}{5}{10}{5}
\lineseg{20}{10}{30}{10}
}%
\lineseg{40}{15}{50}{15}
\mul{10}{5}%
\greenfbox{30}{7.5}{10}{15}{$a$}%
\end{picture}
\begin{picture}(20,30)
\put(10,15){\makebox(0,0){$=$}}
\end{picture}
\begin{picture}(50,30)
\graybox{0}{0}{50}{30}{}%
\thicklines
\lineseg{0}{25}{10}{25}
{\Hcolor
\lineseg{0}{15}{10}{15}
\lineseg{0}{5}{30}{7.5}
}%
\lineseg{20}{20}{30}{17.5}
\lineseg{40}{12.5}{50}{12.5}
\greenfbox{10}{12.5}{10}{15}{$a$}%
\greenfbox{30}{5}{10}{15}{$a$}%
\end{picture}
 \end{gathered} \end{equation}
If $(A,a)$ and $(B,b)$ are $H$-modules, then a \textbf{$H$-module morphism} $f: (A,a) \to (B,b)$ is a map $f: A \to B$ such that the following equality holds: 
 \begin{equation}\label{modulemapeq}\begin{gathered}
\unitlength=1.2pt
\begin{picture}(50,30)
\graybox{0}{0}{50}{30}{}%
\thicklines
\lineseg{0}{20}{10}{20}
{\Hcolor\lineseg{0}{10}{10}{10}}%
\lineseg{20}{15}{30}{15}
\lineseg{40}{15}{50}{15}
\violetfbox{30}{10}{10}{10}{$f$}%
\greenfbox{10}{7.5}{10}{15}{$a$}%
\end{picture}
\begin{picture}(20,30)
\put(10,15){\makebox(0,0){$=$}}
\end{picture}
\begin{picture}(50,30)
\graybox{0}{0}{50}{30}{}%
\thicklines
\lineseg{0}{20}{30}{20}
{\Hcolor\lineseg{0}{10}{30}{10}}%
\lineseg{40}{15}{50}{15}
\violetfbox{10}{15}{10}{10}{$f$}%
\yellowfbox{30}{7.5}{10}{15}{$b$}%
\end{picture}
 \end{gathered} \end{equation} 
Let $\mathbf{MOD}(H)$ be the category whose objects are $H$-modules and whose maps are $H$-module morphisms. 
\end{definition}

\begin{lemma} \cite[Example 7.5.4]{turaev2017monoidal} Let $(H, \nabla, u, \Delta, e, S)$ be a cocommutative Hopf monoid in a symmetric monoidal category $(\mathbb{X}, \otimes, I, \alpha, \lambda, \rho, \sigma)$. Then for the induced symmetric Hopf monad $(H \otimes -, \mu^\nabla, \eta^u, m^\Delta, m_I^e)$, the Eilenberg-Moore category of $(H \otimes -, \mu^\nabla, \eta^u)$ is precisely the category of $H$-modules, that is, $\mathbf{MOD}(H) = \mathbb{X}^{H\otimes -}$. Therefore $\left(\mathsf{MOD}(H), \otimes^{\Delta}, \big(
I,
\unitlength=.8pt
\begin{picture}(25,14)(0,2)
\graybox{0}{0}{25}{14}{}%
\thicklines
{\Hcolor\lineseg{15}{7}{25}{7}}%
\comulunit{5}{2}%
\end{picture}
\big), \alpha^H, \lambda^H, \rho^H, \sigma^H \right)$ is a symmetric monoidal category where:
\[\left(\mathsf{MOD}(H), \otimes^{\Delta}, \big(
I,
\unitlength=.8pt
\begin{picture}(25,14)(0,2)
\graybox{0}{0}{25}{14}{}%
\thicklines
{\Hcolor\lineseg{15}{7}{25}{7}}%
\comulunit{5}{2}%
\end{picture}
\big), \alpha^H, \lambda^H, \rho^H, \sigma^H \right) \!:=\! \left(\mathbb{X}^{H \otimes -}, \otimes^{H\otimes -}, (I, m^e_I), \alpha^{H\otimes -}, \lambda^{H\otimes -}, \rho^{H\otimes -}, \sigma^{H\otimes -} \right)\]
where the right-hand side is defined as in Proposition \ref{EMTSMC}. In particular, for $H$-modules $(A,a)$ and $(B,b)$, the action of the $H$-module $(A,a) \otimes^{H} (B,b)$ is drawn as follows in the graphical calculus: 
\[ \left( \begin{array}[c]{c} A \end{array}, \begin{array}[c]{c} \unitlength=1pt
\begin{picture}(30,30)
\graybox{0}{0}{30}{30}{}%
\thicklines
\lineseg{0}{20}{10}{20}
\lineseg{0}{10}{10}{10}
{\Hcolor\lineseg{15}{15}{30}{15}}%
\greenfbox{10}{7.5}{10}{15}{$a$}%
\end{picture} \end{array} \right) \otimes^H \left(\begin{array}[c]{c} B  \end{array}, \begin{array}[c]{c} \unitlength=1pt
\begin{picture}(30,30)
\graybox{0}{0}{30}{30}{}%
\thicklines
\lineseg{0}{20}{10}{20}
{\Hcolor\lineseg{0}{10}{10}{10}}%
\lineseg{15}{15}{30}{15}
\yellowfbox{10}{7.5}{10}{15}{$b$}%
\end{picture}  \end{array}\right) = \left(\begin{array}[c]{c} A \otimes B  \end{array}, 
\begin{array}[c]{c} \unitlength=1.2pt
\begin{picture}(60,50)
\graybox{0}{0}{60}{50}{}%
\thicklines
\lineseg{0}{40}{40}{40}
\lineseg{0}{30}{25}{30}
\lineseg{25}{30}{35}{20}
\lineseg{35}{20}{40}{20}
\lineseg{50}{35}{60}{35}
\lineseg{50}{15}{60}{15}
{\Hcolor
\lineseg{0}{15}{10}{15}
\lineseg{20}{20}{25}{20}
\lineseg{20}{10}{40}{10}
\lineseg{25}{20}{35}{30}
\lineseg{35}{30}{40}{30}
}%
\comul{10}{10}%
\yellowfbox{40}{27.5}{10}{15}{$b$}%
\greenfbox{40}{7.5}{10}{15}{$a$}%
\end{picture}  \end{array}\right) \] 
\end{lemma}

Therefore, we conclude that the category of modules of a cocommutative Hopf monoid will be a traced symmetric monoidal category (if the base category is also), which to the knowledge of the authors is a novel result. 

\begin{corollary} \label{cor:tracemodules} Let $(\mathbb{X}, \otimes, I, \alpha, \lambda, \rho, \sigma, \mathsf{Tr})$ be a traced symmetric monoidal category and \\$(H, \nabla, u, \Delta, e, S)$ be a cocommutative Hopf monoid in the underlying symmetric monoidal category $(\mathbb{X}, \otimes, I, \alpha, \lambda, \rho, \sigma)$. Then the category of modules over $H$ is a traced symmetric monoidal category, that is, $\left(\mathsf{MOD}(H), \otimes^{\Delta}, \big(
I,
\unitlength=.8pt
\begin{picture}(25,14)(0,2)
\graybox{0}{0}{25}{14}{}%
\thicklines
{\Hcolor\lineseg{15}{7}{25}{7}}%
\comulunit{5}{2}%
\end{picture}
\big), \alpha^H, \lambda^H, \rho^H, \sigma^H, \mathsf{Tr}^H \right)$ is a traced symmetric monoidal category where $\mathsf{Tr}^H(-) = \mathsf{Tr}(-)$. In other words, for an $H$-module morphism 
\[f: (A,a) \otimes^{H \otimes -} (X,x) \to (B,a) \otimes^{H \otimes -} (X,x)\] 
its trace $\mathsf{Tr}^X_{A,B}(f): (A,a) \to (B,b)$ is also an $H$-module morphism. Explicitly, if the equality on the left holds then the equality on the right holds: 
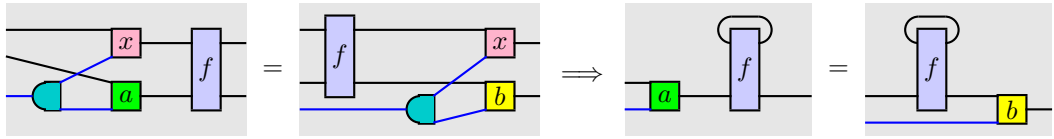
\[\begin{array}[c]{c} \setlength\unitlength{1pt}
\begin{picture}(90,50)
\rgbbox{0}{0}{90}{50}{.9}{.9}{.9}{}%
\thicklines
\comul{10}{10}%
\redfbox{40}{30}{10}{10}{$x$}%
\greenfbox{40}{10}{10}{10}{$a$}%
\violetfbox{70}{10}{10}{30}{$f$}%
\put(0,40){\line(1,0){40}}
\qbezier(0,30)(0,30)(40,20)
{\Hcolor
\put(0,15){\line(1,0){10}}
\qbezier(20,20)(20,20)(40,30)
\put(20,10){\line(1,0){20}}
}%
\put(50,35){\line(1,0){20}}
\put(50,15){\line(1,0){20}}
\put(80,35){\line(1,0){10}}
\put(80,15){\line(1,0){10}}
\end{picture}
\begin{picture}(20,50)
\put(10,25){\makebox(0,0){$=$}}
\end{picture}
\begin{picture}(90,50)
\rgbbox{0}{0}{90}{50}{.9}{.9}{.9}{}%
\thicklines
\violetfbox{10}{15}{10}{30}{$f$}%
\comul{40}{5}%
\redfbox{70}{30}{10}{10}{$x$}%
\yellowfbox{70}{10}{10}{10}{$b$}%
\put(0,40){\line(1,0){10}}
\put(0,20){\line(1,0){10}}
{\Hcolor
\put(0,10){\line(1,0){40}}
}%
\put(20,40){\line(1,0){50}}
\put(20,20){\line(1,0){50}}
{\Hcolor
\qbezier(50,15)(50,15)(70,30)
\qbezier(50,5)(50,5)(70,10)
}%
\put(80,35){\line(1,0){10}}
\put(80,15){\line(1,0){10}}
\end{picture}\end{array} \Longrightarrow \begin{array}[c]{c} \setlength\unitlength{1pt}
\begin{picture}(70,50)
\rgbbox{0}{0}{70}{50}{.9}{.9}{.9}{}%
\thicklines
\greenfbox{10}{10}{10}{10}{$a$}%
\trace{40}{40}{10}{5}
\violetfbox{40}{10}{10}{30}{$f$}%
\put(0,20){\line(1,0){10}}
{\Hcolor 
\put(0,10){\line(1,0){10}}
}%
\put(20,15){\line(1,0){20}}
\put(50,15){\line(1,0){20}}
\end{picture}
\begin{picture}(20,50)
\put(10,25){\makebox(0,0){$=$}}
\end{picture}
\begin{picture}(70,50)
\rgbbox{0}{0}{70}{50}{.9}{.9}{.9}{}%
\thicklines
\yellowfbox{50}{5}{10}{10}{$b$}%
\trace{20}{40}{10}{5}
\violetfbox{20}{10}{10}{30}{$f$}%
\put(0,15){\line(1,0){20}}
{\Hcolor
\put(0,5){\line(1,0){50}}
}%
\put(30,15){\line(1,0){20}}
\put(60,10){\line(1,0){10}}
\end{picture}\end{array} \] 
\end{corollary}

\section{Compact Closed Categories} \label{sec:compactclosed}

An important class of traced symmetric monoidal categories are \emph{compact closed categories}, which are particularly important in categorical quantum foundations \cite{abramsky2004categorical}. Compact closed categories are symmetric monoidal categories where every object has a dual. Every compact closed category comes equipped with a unique trace operator that captures the classical notion of (partial) trace for matrices, which is a fundamental operation for both classical quantum theory and categorical quantum foundations \cite{abramsky2009abstract}.  

\begin{definition} \label{compactex} \cite[Section 4.8]{selinger2010survey} A \textbf{compact closed category} is a symmetric monoidal category $(\mathbb{X}, \otimes, I, \alpha, \lambda, \rho, \sigma)$ such that for every object $X$, there is an object $X^\ast$, called the dual of $X$, and maps $\cup_X: X^\ast \otimes X \to I$, called the cup or evaluation map, and $\cap_X: I \to X \otimes X^\ast$, called the cap or coevaluation map, which are drawn in the graphical calculus as follows: 
\begin{align*}
\unitlength=1.2pt \begin{picture}(30,30)
\put(15,25){\makebox(0,0){$\cup_X: X^*\otimes X\rightarrow I$}}
\graybox{0}{0}{30}{20}{}%
\thicklines
\lineseg{0}{15}{10}{15}
\lineseg{0}{5}{10}{5}
\compactcounit{10}{5}%
\end{picture} && 
    \unitlength=1.2pt
\begin{picture}(30,30)
\put(15,25){\makebox(0,0){$\cap_X:I\rightarrow X\otimes X^*$}}
\graybox{0}{0}{30}{20}{}%
\thicklines
\lineseg{20}{15}{30}{15}
\lineseg{20}{5}{30}{5}
\compactunit{10}{5}%
\end{picture}  
\end{align*}
such that the following equality holds:
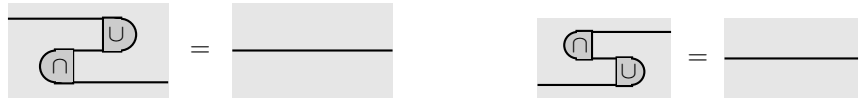
\begin{align*}
    \unitlength=1.2pt
\begin{picture}(50,30)
\graybox{0}{0}{50}{30}{}%
\thicklines
\lineseg{0}{25}{30}{25}
\lineseg{20}{15}{30}{15}
\lineseg{20}{5}{50}{5}
\compactunit{10}{5}%
\compactcounit{30}{15}%
\end{picture}
\begin{picture}(20,30)
\put(10,15){\makebox(0,0){$=$}}
\end{picture}
\begin{picture}(50,30)
\graybox{0}{0}{50}{30}{}%
\thicklines
\lineseg{0}{15}{50}{15}
\end{picture}
&&
  \unitlength=1pt
\begin{picture}(50,30)
\graybox{0}{0}{50}{30}{}%
\thicklines
\lineseg{0}{5}{30}{5}
\lineseg{20}{15}{30}{15}
\lineseg{20}{25}{50}{25}
\compactunit{10}{15}%
\compactcounit{30}{5}%
\end{picture}
\begin{picture}(20,30)
\put(10,15){\makebox(0,0){$=$}}
\end{picture}
\begin{picture}(50,30)
\graybox{0}{0}{50}{30}{}%
\thicklines
\lineseg{0}{15}{50}{15}
\end{picture}
\end{align*}
and these equalities are called the snake equations. 
\end{definition}

\begin{proposition} \cite[Proposition 3.1]{joyal_street_verity_1996} Let $(\mathbb{X}, \otimes, I, \alpha, \lambda, \rho, \sigma)$ be a compact closed category, with duals $(-)^\ast$, caps $\cap$, and cups $\cup$. For a map $f: A \otimes X \to B \otimes X$, its trace $\mathsf{Tr}^X_{A,B}(f): A \to B$ is defined as the following composite:
 \begin{equation*}\label{compacttrace}\begin{gathered} 
\begin{array}[c]{c} \mathsf{Tr}^X_{A,B}(f) := \end{array} \begin{array}[c]{c} \xymatrixcolsep{3pc}\xymatrixrowsep{1pc}\xymatrix{ A \ar[r]^-{\rho^{-1}_A} & A \otimes I \ar[r]^-{1_A \otimes \cap_X} & A \otimes (X \otimes X^\ast) \ar[r]^-{\alpha_{A,X,X^\ast}} &\\
(A \otimes X) \otimes X^\ast \ar[r]^-{f \otimes 1_{X^\ast}} & (B \otimes X) \otimes X^\ast \ar[r]^-{\alpha^{-1}_{B,X,X^\ast}} & B \otimes (X \otimes X^\ast) \ar[r]^-{1_B \otimes \sigma_{X,X^\ast}} & \\
B \otimes (X^\ast \otimes X) \ar[r]^-{1_B \otimes \cup_X} & B \otimes I \ar[r]^-{\rho_B} & B  } \end{array}
  \end{gathered}\end{equation*}
  which is drawn in the graphical calculus as follows: 
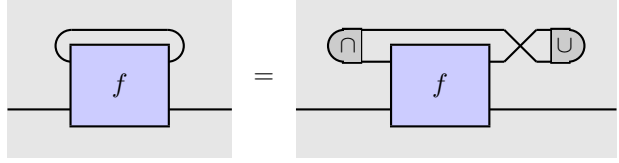
\begin{align*}
    \unitlength=1.2pt
\begin{picture}(70,50)
\graybox{0}{0}{70}{50}{}%
\thicklines
\trace{20}{35}{30}{5}
\lineseg{0}{15}{70}{15}
\violetfbox{20}{10}{30}{25}{$f$}%
\end{picture}
\begin{picture}(20,50)
\put(10,25){\makebox(0,0){$=$}}
\end{picture}
\begin{picture}(100,50)
\graybox{0}{0}{100}{50}{}%
\thicklines
\lineseg{20}{40}{65}{40}
\lineseg{20}{30}{65}{30}
\lineseg{65}{40}{75}{30}
\lineseg{65}{30}{75}{40}
\lineseg{75}{40}{80}{40}
\lineseg{75}{30}{80}{30}
\compactunit{10}{30}%
\compactcounit{80}{30}%
\lineseg{0}{15}{100}{15}
\violetfbox{30}{10}{30}{25}{$f$}%
\end{picture}
\end{align*}
Then $(\mathbb{X}, \otimes, I, \alpha, \lambda, \rho, \sigma, \mathsf{Tr})$ is a traced symmetric monoidal category. 
Furthermore, $\mathsf{Tr}$ is the \emph{unique} trace operator on $(\mathbb{X}, \otimes, I, \alpha, \lambda, \rho, \sigma)$ \cite[Section 3.2]{hasegawa2009traced}. 
\end{proposition}

It is already known that symmetric Hopf monads lift compact closed structure, that is, for a symmetric Hopf monad on a compact closed category, its Eilenberg-Moore category is also compact closed such that the forgetful functor preserves the compact closed structure strictly. The algebra structure of the dual is constructed using the inverse of the fusion operator. 

\begin{proposition} \label{EMcompact} \cite[Theorem 3.8]{bruguieres2007hopf} Let $(\mathbb{X}, \otimes, I, \alpha, \lambda, \rho, \sigma)$ be a compact closed category, with duals $(-)^\ast$, caps $\cap$, and cups $\cup$, and let $(T, \mu, \eta, m, m_I)$ be a symmetric Hopf monad on $(\mathbb{X}, \!\otimes, I, \alpha, \lambda, \rho, \sigma)$. Then the Eilenberg-Moore category $\left(\mathbb{X}^T, \otimes^T, (I, m_I), \alpha^T, \lambda^T, \rho^T, \sigma^T \right)$, which is defined as in Proposition \ref{EMTSMC}, is a compact closed category where: 
\begin{enumerate}[{\em (i)}]
\item The dual of a $T$-algebra $(A, a)$ is defined as follows: 
\[ (A,a)^\ast = \left(  \begin{array}[c]{c} A^\ast \end{array},    \begin{array}[c]{c} \xymatrixcolsep{5pc}\xymatrix{ T(A^\ast) \ar[r]^-{\rho^{-1}_{T(A^\ast)}} & T(A^\ast) \otimes I \ar[r]^-{1_{T(A^\ast)} \otimes \cap_A} & \\
T(A^\ast) \otimes (A \otimes A^\ast)\ar[r]^-{\alpha_{T(A^\ast),A,A^\ast}} & \\
\left( T(A^\ast) \otimes A \right) \otimes A^\ast \ar[r]^-{(1_{T(A^\ast)} \otimes \eta_A) \otimes 1_{A^\ast}} & \\
\left( T(A^\ast) \otimes T(A) \right) \otimes A^\ast \ar[r]^-{{\mathsf{h}^l}^{-1}_{A^\ast,A} \otimes 1_{A^\ast}} & \\
T(A^\ast \otimes T(A)) \otimes A^\ast \ar[r]^-{ T(1_{A^\ast} \otimes \eta_A) \otimes 1_{A^\ast}} & \\
T(A^\ast \otimes A) \otimes A^\ast \ar[r]^-{ T(\cap_A) \otimes 1_{A^\ast}} & \\
T(I) \otimes A^\ast \ar[r]^-{ m_I \otimes 1_{A^\ast}} & I \otimes A^\ast  \ar[r]^-{\lambda_{A^\ast}} & A^\ast}\end{array} \right)   \]
\item The cup $\cup^T_{(A,a)}: (A,a)^\ast \otimes^T (A,a) \to (I, m_I)$ is defined as $\cup^T_{(A,a)} = \cup_A$;
\item The cap $\cap^T_{(A,a)}: (I, m_I) \to (A,a) \otimes^T (A,a)^\ast$ is defined as $\cap^T_{(A,a)} = \cap_A$. 
\end{enumerate}
Furthermore, the forgetful functor $U^T\!\!:\!\left(\mathbb{X}^T, \otimes^T, (I, m_I), \alpha^T, \lambda^T, \rho^T, \sigma^T \right) \to (\mathbb{X}, \otimes, I, \alpha, \lambda, \rho, \sigma)$ is a strict compact closed functor, that is, the equalities in Proposition \ref{EMTSMC} hold and $U^T( (-)^\ast ) = U^T( - )^\ast$, ${U^T(\cup^T_{-}) = \cup_{U^T(-)}}$, and $U^T(\cap^T_{-}) = \cap_{U^T(-)}$.  
\end{proposition}

In particular, this implies that the Eilenberg-Moore category of a symmetric Hopf monad on a compact closed category is also a traced symmetric monoidal category. Since the forgetful functor is a strict compact closed functor, it is easy to see that it also preserves the trace operator strictly. Therefore we conclude that symmetric Hopf monads on a compact closed category also lift the traced symmetric monoidal structure to their Eilenberg-Moore categories, and as such are also traced monads.  

\begin{corollary}\label{tracecoherentcompactclosed} Let $(\mathbb{X}, \otimes, I, \alpha, \lambda, \rho, \sigma)$ be a compact closed category and let $(T, \mu, \eta, m, m_I)$ be a symmetric Hopf monad on $(\mathbb{X}, \otimes, I, \alpha, \lambda, \rho, \sigma)$. Then $(T, \mu, \eta, m, m_I)$ is a trace-coherent Hopf monad (and also a traced monad) on $(\mathbb{X}, \otimes, I, \alpha, \lambda, \rho, \sigma, \mathsf{Tr})$, where the trace operator is defined as in Example \ref{compactex}.   
\end{corollary}
\begin{proof} We will show that $(T, \mu, \eta, m, m_I)$ is a traced monad. Let $(X,x)$, $(A,a)$, and $(B,b)$ be $T$-algebras and $f: (A,a) \otimes^T (X,x) \to (B,b) \otimes^T (X,x)$ be a $T$-algebra morphism. By Proposition \ref{EMcompact}, it follows that all the maps in the definition of $\mathsf{Tr}^X_{A,B}(f)$, as defined in Example \ref{compactex}, are $T$-algebra morphisms. Thus since the composite of $T$-algebra morphisms is a $T$-algebra morphism, we have that $\mathsf{Tr}^X_{A,B}(f): (A,a) \to (B,b)$ is a $T$-algebra morphism. Explicitly, $\mathsf{Tr}^X_{A,B}(f)$ as a $T$-algebra morphism is given piece-wise as follows: 
 \[ \begin{array}[c]{c} \mathsf{Tr}^X_{A,B}(f) \end{array} :=\begin{array}[c]{c} \xymatrixcolsep{2.75pc}\xymatrixrowsep{1pc}\xymatrix{ (A, a) \ar[r]^-{\rho^{-1}_A} & (A,a) \otimes^T (I,m_I) \ar[r]^-{1_A \otimes \cup_X} & \\
 (A,a) \otimes^T \left( (X,x) \otimes^T (X,x)^\ast \right) \ar[r]^-{\alpha_{A,X,X^\ast}} & \left( (A,a) \otimes^T (X,x) \right) \otimes^T (X,x)^\ast \ar[r]^-{f \otimes 1_{X^\ast}} & \\
\left( (B,b) \otimes^T (X,x) \right) \otimes^T (X,x)^\ast \ar[r]^-{\alpha^{-1}_{B,X,X^\ast}} & (B,b) \otimes^T \left((X,x) \otimes^T (X,x)^\ast \right) \ar[r]^-{1_B \otimes \sigma_{X,X^\ast}} &\\
(B,b) \otimes^T \left((X,x)^\ast \otimes (X,x)\right) \ar[r]^-{1_B \otimes \cup_X} &\\
(B,b) \otimes^T (I,m_I) \ar[r]^-{\rho_B} & (B,b)  } \end{array} \]
So we conclude that $(T, \mu, \eta, m, m_I)$ is a traced monad. Thus by Theorem \ref{mainthm}, $(T, \mu, \eta, m, m_I)$ is a trace-coherent Hopf monad. 
\end{proof}

\section{Traced Cartesian Monoidal Categories}\label{carttracesec}

Any category with finite products is a symmetric monoidal category, called a Cartesian monoidal category. In a traced Cartesian monoidal category, the trace operator captures the notion of feedback via fixed points. In fact, for Cartesian monoidal categories, trace operators are in bijective correspondence with Conway operators, which are special kinds of fixed point operators (which is a result proved by Hyland and the first named author independently \cite{Hasegawa97recursionfrom}). In this section, we will show how for traced Cartesian monoidal categories, trace-coherent Hopf monads can also be characterized in terms of the Conway operator.

There are multiple equivalent ways of defining a category with finite products. Since we are interested in their induced symmetric monoidal structure, we will define a category with finite products from this perspective, that is, as a symmetric monoidal whose monoidal product is a product and whose monoidal unit is a terminal object. Of course, any category with finite products is a Cartesian monoidal category, where the symmetric monoidal structure is derived using the universal property of the product, and conversely, every Cartesian monoidal category is a category with finite products. Alternatively, a Cartesian monoidal category can be defined as a symmetric monoidal category with natural copy and delete maps satisfying the axioms found in \cite[Table 7]{selinger2010survey}. 

\begin{definition} \label{cartmondef}\cite[Section 6.1 \& Section 6.4]{selinger2010survey} A \textbf{Cartesian monoidal category} is a symmetric monoidal category $(\mathbb{X}, \times, \top, \alpha, \lambda, \rho, \sigma)$ whose monoidal structure is a given by \textbf{finite products}, that is: 
\begin{enumerate}[{\em (i)}]
\item The monoidal unit $\top$ is a \textbf{terminal object}, that is, for every object $A$ there exists a \emph{unique map} $t_A: A \to \top$;
\item For every pair of objects $A$ and $B$, $A \times B$ a \textbf{product} of $A$ and $B$ with \textbf{projection} maps $\pi_0: A \times B \to A$ and $\pi_1: A \times B \to B$ defined as following composites: 
\begin{align*}
\pi_0 :=   \xymatrixcolsep{3pc}\xymatrix{ A \times B \ar[r]^-{1_A \times t_B} & A \times \top \ar[r]^-{\rho_A} & A } &&   \pi_1 :=   \xymatrixcolsep{3pc}\xymatrix{ A \times B \ar[r]^-{t_A \times 1_B} & \top \times B \ar[r]^-{\lambda_B} & B } 
\end{align*}
that is, for every pair of maps $f_0: C \to A$ and $f_1: C \to B$, there is a \emph{unique} map ${\langle f_0, f_1 \rangle: C \to A \times B}$, called the \textbf{pairing} of $f_0$ and $f_1$, such that:
\begin{align}
    \pi_0 \circ \langle f_0, f_1 \rangle = f_0 && \pi_1 \circ \langle f_0, f_1 \rangle = f_1
\end{align}
\end{enumerate} 
A \textbf{traced Cartesian monoidal category} is a traced symmetric monoidal category \\
$(\mathbb{X}, \times, \top, \alpha, \lambda, \rho, \sigma, \mathsf{Tr})$ whose underlying symmetric monoidal category $(\mathbb{X}, \times, \top, \alpha, \lambda, \rho, \sigma)$ is a\\
Cartesian monoidal category. 
\end{definition}

Traced Cartesian monoidal categories can equivalently be defined as a Cartesian monoidal category equipped with a Conway operator. Here we provide the Conway operator axiomatization found in \cite{Hasegawa97recursionfrom}, but equivalent alternative axiomatizations can be found in \cite{hasegawa2004uniformity,simpson2000complete}. 

\begin{definition}\cite[Theorem 3.1]{Hasegawa97recursionfrom} A \textbf{Conway operator} on a Cartesian monoidal category \\
$(\mathbb{X}, \times, \top, \alpha, \lambda, \rho, \sigma)$ is a family of operators $\mathsf{Fix}^X_A: \mathbb{X}(A \times X, X) \to \mathbb{X}(A,X)$ such that: 
\begin{description} 
\item[\textbf{[Parametrized Fixed Point]}] For every map $f: A \times X \to X$ the following equality holds: 
 \begin{equation}\label{fixedpoint}\begin{gathered}
 \mathsf{Fix}^X_A(f) = f \circ \left \langle 1_A, \mathsf{Fix}^X_A(f) \right \rangle 
   \end{gathered}\end  {equation}
   \item[\textbf{[Naturality]}] For every map $f: A \times X \to X$ and every map $g: A^\prime \to A$ the following equality holds: 
 \begin{equation}\label{fpnat1}\begin{gathered}
 \mathsf{Fix}^X_A\left( f \circ (g \times 1_X) \right) =  \mathsf{Fix}^X_{A^\prime}\left( f \right) \circ g
   \end{gathered}\end  {equation}
   and for every map $f: A \times X \to X^\prime$ and every map $k: X^\prime \to X$ the following equality holds:
    \begin{equation}\label{fpnat2}\begin{gathered}
 \mathsf{Fix}^X_A\left( k \circ f \right) =  k \circ \mathsf{Fix}^{X^\prime}_{A}\left( f \circ (1_A \times k) \right) 
   \end{gathered}\end  {equation}
   \item[\textbf{[Beki\v{c} Lemma]}] For every map $f: A \times (X \times Y) \to X$ and $g: A \times (X \times Y) \to Y$ the following equality holds: 
 \begin{equation}\label{Bekic}\begin{gathered}
 \mathsf{Fix}^{X \times Y}_A(\langle f, g \rangle) =\\ \langle \pi_1, \mathsf{Fix}^{Y}_{A \times X}(g \circ \alpha^{-1}_{A,X,Y}) \rangle  \circ \left \langle 1_A, \mathsf{Fix}^{X}_{A}\left( f \circ \alpha^{-1}_{A,X,Y} \circ  \langle 1_{A \times X}, \mathsf{Fix}^{Y}_{A \times X}(g \circ \alpha^{-1}_{A,X,Y}) \rangle \right)   \right \rangle
   \end{gathered}\end  {equation}
\end{description}
For a map $f: A \times X \to X$, $\mathsf{Fix}^X_A(f): A \to X$ is called the \textbf{parametrized fixed point} of $f$. 
\end{definition}

\begin{proposition}\cite[Theorem 3.1]{Hasegawa97recursionfrom} \label{fixtraceprop} Let $(\mathbb{X}, \times, \top, \alpha, \lambda, \rho, \sigma)$ be a Cartesian monoidal category: 
\begin{enumerate}[{\em (i)}]
\item Let $\mathsf{Fix}$ be a Conway operator on $(\mathbb{X}, \times, \top, \alpha, \lambda, \rho, \sigma)$. Then for a map $f: A \times X \to B \times X$, its trace ${\mathsf{Tr}^X_{A,B}(f): A \to B}$ is defined as follows:
 \[ \mathsf{Tr}^X_{A,B}(f) := \xymatrixcolsep{5pc}\xymatrixrowsep{1pc}\xymatrix{ A \ar[r]^-{\langle 1_A, \mathsf{Fix}^X_A(\pi_1 \circ f) \rangle} & A \times X \ar[r]^-{f} & B \times X \ar[r]^-{\pi_0} &  B } \]
Then $(\mathbb{X}, \times, \top, \alpha, \lambda, \rho, \sigma, \mathsf{Tr})$ is a traced Cartesian monoidal category.
\item Let $(\mathbb{X}, \times, \top, \alpha, \lambda, \rho, \sigma, \mathsf{Tr})$ be a Cartesian monoidal category. For a map $f: A \times X \to X$, its parametrized fixed point $\mathsf{Fix}^X_{A}(f): A \to X$ is defined as follows:
\[\mathsf{Fix}^X_{A}(f) := \mathsf{Tr}^X_{A,X}(\langle f, f \rangle)  \]
Then $\mathsf{Fix}$ is a Conway operator on $(\mathbb{X}, \times, \top, \alpha, \lambda, \rho, \sigma)$.  
\end{enumerate}
Furthermore, these constructions are inverses of each other. 
\end{proposition}

For Cartesian monoidal categories, every monad has a canonical symmetric bimonad structure and the Eilenberg-Moore category is again a Cartesian monoidal category. So any monad will lift the Cartesian monoidal structure. This is a reformulation of the well-known fact in category theory that finite products lift to Eilenberg-Moore categories. In other words, for a monad on a category with finite products, its Eilenberg-Moore category will also have finite products which are strictly preserved by the forgetful functor \cite[Proposition 4.3.1]{borceux1994handbook}. 

\begin{lemma}\cite[Example 1.4]{moerdijk2002monads} \label{cartmonadlem} Let $(\mathbb{X}, \times, \top, \alpha, \lambda, \rho, \sigma)$ be a Cartesian monoidal category and let $(T, \mu, \eta)$ be a monad on the underlying category $\mathbb{X}$. Define the natural transformation $m_{A,B}: T(A \times B) \to T(A) \times T(B)$ and the map $m_\top: T(\top) \to \top$ as follows: 
\begin{align*}
m_{A,B} := \left \langle T(\pi_0), T(\pi_1) \right \rangle  && m_\top := t_\top
\end{align*}
Then $(T, \mu, \eta, m, m_I)$ is a symmetric bimonad on $(\mathbb{X}, \times, \top, \alpha, \lambda, \rho, \sigma)$, and it is the unique symmetric bimonad structure on the monad $(T, \mu, \eta)$. Furthermore, $\left(\mathbb{X}^T, \times^T, (\top, m_\top), \alpha^T, \lambda^T, \rho^T, \sigma^T \right)$, as defined in Proposition \ref{EMTSMC}, is a Cartesian monoidal category where: 
\begin{enumerate}[{\em (i)}]
\item \label{cartmonadlem1} The product of $T$-algebras $(A,a)$ and $(B,b)$ is the $T$-algebra \[(A,a) \times^T (B,b) = \left(A \times B, \langle a \circ T(\pi_0), b \circ T(\pi_1) \rangle \right)\] 
with projection maps defined as:
\begin{align*}
   \pi^T_0 = \pi_0 : (A,a) \times^T (B,b) \to (A,a) && \pi^T_1 = \pi_1: (A,a) \times^T (B,b) \to (B,b)  
\end{align*}
\item \label{cartmonadlem2} The pairing of $T$-algebra morphisms $f: (C,c) \to (A,a)$ and $g: (C,c) \to (B, b)$ is their pairing in the base category $\langle f,g \rangle^T = \langle f, g \rangle: (C,c) \to (A,a) \times (B,b)$. 
\item The terminal object is the $T$-algebra $(\top, m_I) = (\top, t_{T(\top)})$ and where the unique map to the terminal object is $t^T_{(A,a)} = t_A: (A,a) \to (\top, m_I)$. 
\end{enumerate}
\end{lemma}

Traced monads on traced Cartesian monoidal categories can also be characterized in terms of the Conway operator. Explicitly, a traced monad on a traced Cartesian monoidal category is equivalently a symmetric bimonad such that the parametrized fixed point of an algebra morphism is again an algebra morphism.

\begin{lemma} Let $(\mathbb{X}, \times, \top, \alpha, \lambda, \rho, \sigma, \mathsf{Tr})$ be a traced Cartesian monoidal category, with induced Conway operator $\mathsf{Fix}$, as defined in Proposition \ref{fixtraceprop}, and let $(T, \mu, \eta)$ be a monad on the underlying category $\mathbb{X}$. Then the induced symmetric bimonad $(T, \mu, \eta, m, m_I)$ is a traced monad on $(\mathbb{X}, \times, \top, \alpha, \lambda, \rho, \sigma, \mathsf{Tr})$ if and only if for every pair of $T$-algebras $(X,x)$ and $(A,a)$, if 
\[f: (A,a) \times^T (X,x) \to (X,x)\]
is a $T$-algebra morphism, then its parametrized fixed point ${\mathsf{Fix}^X_{A}(f): (A,a) \to (X,x)}$ is a $T$-algebra morphism. Explicitly, if the equality on the left holds, then the equality on the right holds: 
\begin{equation}\label{fixedmonad}\begin{gathered}
f \circ (a \times x) \circ m_{A,X} = x \circ T(f) \Longrightarrow  \mathsf{Fix}^X_{A}(f) \circ a = x \circ T\left( \mathsf{Fix}^X_{A}(f) \right)
\end{gathered}\end  {equation}
\end{lemma}
\begin{proof} $\Rightarrow:$ Suppose that $(T, \mu, \eta, m, m_I)$ is a traced monad. Let $f: (A,a) \times^T (X,x) \to (X,x)$ be a $T$-algebra morphism. By Lemma \ref{cartmonadlem}.(\ref{cartmonadlem2}), recall that the pairing of $T$-algebra morphisms is again a $T$-algebra morphism, so ${\langle f, f \rangle: (A,a) \times^T (X,x) \to (X,x) \times^T (X,x)}$ is a $T$-algebra morphism. Thus its trace ${\mathsf{Tr}^X_{A,X}(\langle f, f \rangle): (A,a) \to (X,x)}$ is also a $T$-algebra morphism. So we conclude that the parametrized fixed point of $f$, 
\[\mathsf{Fix}^X_{A}(f) := \mathsf{Tr}^X_{A,X}(\langle f, f \rangle): (A,a) \to (X,x)\]
is also a $T$-algebra morphism, as desired. 

$\Leftarrow:$ Suppose that the parametrized fixed point of a $T$-algebra morphism is again a $T$-algebra morphism.  Let $f: (A,a) \otimes^T (X,x) \to (B,b) \otimes^T (X,x)$ be a $T$-algebra morphism. By Lemma \ref{cartmonadlem}.(\ref{cartmonadlem1}), recall that the projection maps $\pi_0: (B,b) \times^T (X,x) \to (B,b)$ and $\pi_1: (B,b) \times^T (X,x) \to (X,x)$ are $T$-algebra morphisms. Therefore, the composite $\pi_1 \circ f:  (A,a) \otimes^T (X,x) \to (X,x)$ is a $T$-algebra morphism, and therefore so is its parametrized fixed point $\mathsf{Fix}^X_{A}(f): (A,a) \to (X,x)$. As such, it follows that the following is a composite of $T$-algebra morphisms:  
 \[ \begin{array}[c]{c} \mathsf{Tr}^X_{A,B}(f) \end{array} := \begin{array}[c]{c} \xymatrixcolsep{5pc}\xymatrixrowsep{1pc}\xymatrix{ (A,a) \ar[r]^-{\langle 1_A, \mathsf{Fix}^X_A(\pi_1 \circ f) \rangle} & (A,a) \times^T (X,x) \ar[r]^-{f} &\\ (B,b) \times^T (X,x) \ar[r]^-{\pi_0} &  (B ,b)}  \end{array} \]
Thus $\mathsf{Tr}^X_{A,B}(f): (A,a) \to (B,b)$ is a $T$-algebra morphism, and we conclude that  $(T, \mu, \eta, m, m_I)$ is a traced monad.~~~~~\end{proof}

Therefore traced monads also lift the Conway operator to the Eilenberg-Moore category. 

\begin{corollary}  Let $(T, \mu, \eta, m, m_I)$ be a traced monad on a traced Cartesian monoidal category $(\mathbb{X}, \times, \top, \alpha, \lambda, \rho, \sigma, \mathsf{Tr})$. Then $\left(\mathbb{X}^T, \times^T, (\top, m_\top), \alpha^T, \lambda^T, \rho^T, \sigma^T, \mathsf{Tr}^T \right)$, as defined in Proposition \ref{EMTSMC}, is a traced Cartesian monoidal category where the induced Conway operator $\mathsf{Fix}^T$ is defined as follows for a $T$-algebra morphism $f: (A,a) \times^T (X,x) \to (X,x)$:
\[ {\mathsf{Fix}^T}^{(X,x)}_{(A,a)}(f) =  \mathsf{Fix}^X_{A}(f) : (A,a) \to (X,x) \]
Furthermore, the forgetful functor:
\[U^T\!\!:\!\left(\mathbb{X}^T, \times^T, (I, m_I), \alpha^T, \lambda^T, \rho^T, \sigma^T, \mathsf{Tr}^T \right) \to (\mathbb{X}, \times, I, \alpha, \lambda, \rho, \sigma, \mathsf{Tr})\]
preserves the Conway operator, that is, $U^T\left(\mathsf{Fix}^T(-)  \right) = \mathsf{Fix}\left( U^T (-) \right)$.
\end{corollary}

Finally, trace-coherent Hopf monads on traced Cartesian monoidal categories can themselves also be characterized in terms of the Conway operator. Consider this time a map of type:
\[f: A \times T(X) \to T(X)\]
We can first obtain its parametrized fixed point: 
\[ \mathsf{Fix}^{T(X)}_{A}(f): A \to T(X) \] 
We can then apply the functor $T$ to the parametrized fixed point and obtain: 
\[ T\left( \mathsf{Fix}^{T(X)}_{A}(f) \right): T(A) \to TT(X) \] 
Post-composing with the monad multiplication, we obtain a map of type $T(A) \to T(X)$: 
\[\mu_X \circ T\left( \mathsf{Fix}^{T(X)}_{A}(f) \right): T(A) \to T(X) \] 
Alternatively, we could have first applied the functor $T$ to $f$: 
\[ T(f): T(A \times T(X)) \to TT(X) \]
We cannot take the parametrized fixed point of $T(f)$ when it is of this form. To obtain a map on which we can apply the Conway operator, we post-compose by monad multiplication and pre-compose by the inverse of the left fusion operator: 
 \[ \xymatrixcolsep{5pc}\xymatrixrowsep{1pc}\xymatrix{ T(A) \times T(X) \ar[r]^-{\mathsf{h}^{l^{-1}}_{A,X}  } & T(A \times T(X)) \ar[r]^-{T(f)} & TT(X) \ar[r]^-{\mu_X} &  T(X) } \]
 We can now take the parametrized fixed point of this map to obtain a map of type $T(A) \to T(X)$: 
 \[\mathsf{Fix}^{T(X)}_{T(A)}\left( \mu_X \circ T(f) \circ \mathsf{h}^{l^{-1}}_{A,X} \right) : T(A) \to T(X) \] 
Thus the trace-coherent axiom can alternatively be stated as saying that these two maps of type $T(A) \to T(X)$ are equal. 

\begin{proposition} \label{prop:hopftracefix} Let $(\mathbb{X}, \times, \top, \alpha, \lambda, \rho, \sigma, \mathsf{Tr})$ be traced Cartesian monoidal category with induced Conway operator $\mathsf{Fix}$ and let $(T, \mu, \eta, m, m_I)$ be a symmetric Hopf monad on the underlying symmetric monoidal category $(\mathbb{X}, \times, \top, \alpha, \lambda, \rho, \sigma)$. Then $(T, \mu, \eta, m, m_I)$ is a trace-coherent Hopf monad (or equivalently a traced monad) on $(\mathbb{X}, \times, \top, \alpha, \lambda, \rho, \sigma, \mathsf{Tr})$ if and only if for every map 
\[f: A \times T(X) \to T(X)\] 
the following equality holds: 
\begin{align} \label{fixcoherent}
\mu_X \circ T\left( \mathsf{Fix}^{T(X)}_{A}(f) \right) = \mathsf{Fix}^{T(X)}_{T(A)}\left( \mu_X \circ T(f) \circ \mathsf{h}^{l^{-1}}_{A,X} \right) 
\end{align}
\end{proposition}
\begin{proof} $\Rightarrow$: Suppose that $(T, \mu, \eta, m, m_I)$ is trace-coherent. First recall that in a Cartesian monoidal category, the pairing operation satisfies the following equality for any arbitrary maps $j: A^\prime \to A$, $h: A \to B$, $k: A \to C$, ${f: B \to B^\prime}$, and $g: C \to C^\prime$: 
\begin{align} \label{pairprop1}
 (f \times g) \circ \langle h, k \rangle \circ j = \left \langle f \circ h \circ j, g \circ k \circ j \right \rangle     
\end{align}
Expressing the Conway operator $\mathsf{Fix}$ in terms of the trace operator $\mathsf{Tr}$ and using the trace-coherence, we compute: 
\begin{align*}
&\mu_X \circ T\left( \mathsf{Fix}^{T(X)}_{A}(f) \right) =~ \mu_X \circ T\left( \mathsf{Tr}^{T(X)}_{A, T(X)}\left( \langle f, f \rangle \right) \right) \tag{Def. of $\mathsf{Fix}$ -- Prop. \ref{fixtraceprop}} \\ 
&=~ \mu_X \circ \mathsf{Tr}^{T(X)}_{T(A),T(X)}\left( \mathsf{h}^l_{X,X} \circ  T\left( \langle f, f \rangle \right) \circ \mathsf{h}^{l^{-1}}_{A,X}   \right) \tag{Trace-Coherent -- (\ref{tracecoherent})} \\ 
&=~  \mathsf{Tr}^{T(X)}_{T(A),T(X)}\left( (\mu_X \times 1_{T(X)}) \circ \mathsf{h}^l_{X,X} \circ  T\left( \langle f, f \rangle \right) \circ \mathsf{h}^{l^{-1}}_{A,X}   \right) \tag{\textbf{[Tightening]} -- (\ref{tightening2})}\\
&=~  \mathsf{Tr}^{T(X)}_{T(A),T(X)}\left( (\mu_X \times 1_{T(X)}) \circ (1_{TT(X)} \times \mu_X) \circ m_{T(X), T(X)} \circ  T\left( \langle f, f \rangle \right) \circ \mathsf{h}^{l^{-1}}_{A,X}   \right) \tag{Fus. Op. Def. -- Def. \ref{Hopfmonaddef}} \\
&=~ \mathsf{Tr}^{T(X)}_{T(A),T(X)}\left( (\mu_X \times \mu_X) \circ m_{T(X), T(X)} \circ  T\left( \langle f, f \rangle \right) \circ \mathsf{h}^{l^{-1}}_{A,X}   \right) \\ 
&=~ \mathsf{Tr}^{T(X)}_{T(A),T(X)}\left( (\mu_X \times \mu_X) \circ \langle T(\pi_0), T(\pi_1) \rangle \circ  T\left( \langle f, f \rangle \right) \circ \mathsf{h}^{l^{-1}}_{A,X}   \right) \tag{Def. of $m$ -- Lem. \ref{cartmonadlem}} \\ 
&=~ \mathsf{Tr}^{T(X)}_{T(A),T(X)}\left( \left \langle \mu_X \circ T(\pi_0) \circ  T\left( \langle f, f \rangle \right) \circ \mathsf{h}^{l^{-1}}_{A,X}, \mu_X \circ T(\pi_1) \circ  T\left( \langle f, f \rangle \right) \circ \mathsf{h}^{l^{-1}}_{A,X} \right \rangle    \right) \tag{Property of $\langle -,- \rangle$ -- (\ref{pairprop1})} \\ 
&=~ \mathsf{Tr}^{T(X)}_{T(A),T(X)}\left( \left \langle \mu_X \circ T\left(\pi_0 \circ \langle f, f \rangle \right) \circ \mathsf{h}^{l^{-1}}_{A,X}, \mu_X \circ T\left(\pi_1 \circ \langle f, f \rangle \right) \circ \mathsf{h}^{l^{-1}}_{A,X} \right \rangle    \right)  \tag{Functoriality of $T$}\\ 
&=~ \mathsf{Tr}^{T(X)}_{T(A),T(X)}\left( \left \langle \mu_X \circ T(f) \circ \mathsf{h}^{l^{-1}}_{A,X}, \mu_X \circ T(f) \circ \mathsf{h}^{l^{-1}}_{A,X}  \right \rangle \right) \tag{Property of $\langle -,- \rangle$ -- Def. \ref{cartmondef})} \\
&=~ \mathsf{Fix}^{T(X)}_{T(A)}\left( \mu_X \circ T(f) \circ \mathsf{h}^{l^{-1}}_{A,X} \right) \tag{Def. of $\mathsf{Fix}$ -- Prop. \ref{fixtraceprop}}
\end{align*}
So the desired equality holds. 

$\Leftarrow:$ Suppose that (\ref{fixcoherent}) holds. First observe that the fusion operator \[\mathsf{h}^l_{A,X}: T(A \times T(X)) \to T(A) \times T(X)\]
satisfies the following equality: 
\begin{align} \label{pi1fusion}
  \pi_1 \circ \mathsf{h}^l_{A,X} = \mu_X \circ T(\pi_1)   
\end{align}
and is also compatible with pairing in the sense that for maps $f: A^\prime \to A$ and $g: A^\prime \to T(X)$, the following equality holds: 
\begin{align}\label{pairfusion}
   \mathsf{h}^l_{A,X} \circ  T\left( \langle f, g \rangle \right) = \left \langle T(f), \mu_X \circ T(g) \right \rangle  
\end{align}
Then expressing the trace operator $\mathsf{Tr}$ in terms of the Conway operator $\mathsf{Fix}$, we compute: 
\begin{align*}
&\mathsf{Tr}^{T(X)}_{T(A),T(B)}\left( \mathsf{h}^l_{B,X} \circ  T(f) \circ \mathsf{h}^{l^{-1}}_{A,X}   \right) \\
&=~ \pi_0 \circ \mathsf{h}^l_{B,X} \circ  T(f) \circ \mathsf{h}^{l^{-1}}_{A,X} \circ \left \langle 1_{T(A)}, \mathsf{Fix}^{T(X)}_{T(A)}\left( \pi_1 \circ \mathsf{h}^l_{B,X} \circ T(f) \circ \mathsf{h}^{l^{-1}}_{A,X} \right) \right \rangle \tag{Def. of $\mathsf{Tr}$ -- Prop. \ref{fixtraceprop}} \\ 
&=~ T(\pi_0) \circ  T(f) \circ \mathsf{h}^{l^{-1}}_{A,X} \circ \left \langle 1_{T(A)}, \mathsf{Fix}^{T(X)}_{T(A)}\left( \mu_X \circ T(\pi_1) \circ T(f) \circ \mathsf{h}^{l^{-1}}_{A,X} \right) \right \rangle \tag{Fusion Operator Identity -- (\ref{pi1fusion})} \\ 
&=~ T(\pi_0) \circ  T(f) \circ \mathsf{h}^{l^{-1}}_{A,X} \circ \left \langle T(1_A), \mathsf{Fix}^{T(X)}_{T(A)}\left( \mu_X \circ T(\pi_1 \circ f) \circ \mathsf{h}^{l^{-1}}_{A,X} \right) \right \rangle \tag{Functoriality of $T$} \\ 
&=~ T(\pi_0) \circ  T(f) \circ \mathsf{h}^{l^{-1}}_{A,X} \circ \left \langle T(1_A), \mu_X \circ T\left( \mathsf{Fix}^{T(X)}_{A}(\pi_1 \circ f) \right) \right \rangle \tag{Fixed Point Coherent -- (\ref{fixcoherent})} \\ 
&=~ T(\pi_0) \circ  T(f) \circ \mathsf{h}^{l^{-1}}_{A,X} \circ {\mathsf{h}^l}_{A,X} \circ T\left( \left \langle 1_A, \mathsf{Fix}^{T(X)}_{A}(\pi_1 \circ f) \right \rangle \right) \tag{Fusion Operator Identity -- (\ref{pairfusion})} \\
&=~ T(\pi_0) \circ  T(f) \circ T\left( \left \langle 1_A, \mathsf{Fix}^{T(X)}_{A}(\pi_1 \circ f) \right \rangle \right) \tag{Fusion Operator Inverses} \\
&=~ T\left(\pi_0 \circ f \circ \left \langle 1_A, \mathsf{Fix}^{T(X)}_{A}(\pi_1 \circ f) \right \rangle \right) \tag{Functoriality of $T$} \\ 
&=~ T\left( \mathsf{Tr}^{T(X)}_{A,B}\left( f \right) \right)  \tag{Def. of $\mathsf{Tr}$ -- Prop. \ref{fixtraceprop}} 
\end{align*}
So the desired trace-coherent identity holds and we conclude that $(T, \mu, \eta, m, m_I)$ is trace-coherent. \\
\phantom{ } \hfill 
\end{proof}

\section{Idempotent Hopf Monads}

In order to discuss trace-coherent Hopf monads in a coCartesian setting, we must first discuss idempotent monads, which are monads whose monad multiplication is an isomorphism. We will show some interesting new identities about bimonads or Hopf monads that are also idempotent monads, and some key results about idempotent traced monads. In particular, we will give necessary and sufficient conditions for when a trace-coherent Hopf monad is also an idempotent monad, which will be crucial for characterizing trace-coherent Hopf monads on traced coCartesian monoidal categories in the next section. 

\begin{definition}\cite[Proposition 4.2.3]{borceux1994handbook} An \textbf{idempotent monad} on a category $\mathbb{X}$ is a monad $(T, \mu, \eta)$ on $\mathbb{X}$ such that $\mu_A: TT(A) \to T(A)$ is a natural isomorphism, so $TT(A) \cong T(A)$. 
\end{definition}

For idempotent monads, the inverse of the monad multiplication must be the monad unit. In fact, for an idempotent monad, its algebras are precisely the objects for which, this time, the monad unit is an isomorphism. As such, it follows that the Eilenberg-Moore category of an idempotent monad can be interpreted as a full subcategory of the base category. 

\begin{lemma} \label{idempotentalg}\cite[Proposition 4.2.3 \& Corollary 4.2.4]{borceux1994handbook}  Let $(T, \mu, \eta)$ be an idempotent monad on a category $\mathbb{X}$. Then: 
\begin{enumerate}[{\em (i)}]
\item ${\mu^{-1}_A \!=\! \eta_{T(A)} \!=\! T(\eta_A)}$
\item If $\eta_A: A \to T(A)$ is an isomorphism, then $(A, \eta^{-1}_A)$ is a $T$-algebra;
\item \label{idempotentalg3} If $(A,a)$ is a $T$-algebra then $a = \eta_A^{-1}$;
\item \label{idempotentalg4} If $(A, \eta^{-1}_A)$ and $(B, \eta^{-1}_B)$ are $T$-algebras, then for every map $f: A \to B$, we have that $f: (A, \eta^{-1}_A) \to (B, \eta^{-1}_B)$ is a $T$-algebra morphism.
\end{enumerate}
Therefore, every object $A$ has \emph{at most one} $T$-algebra structure if and only if $\eta_A: A \to T(A)$ is an isomorphism. Furthermore, the forgetful functor $U^{T}: \mathbb{X}^T \to \mathbb{X}$ is full and faithful, and so the Eilenberg-Moore category $\mathbb{X}^T$ is isomorphic to the full subcategory of $\mathbb{X}$ whose objects are those of $\mathbb{X}$ such that $\eta_A: A \to T(A)$ is an isomorphism. 
\end{lemma}

We now discuss idempotent bimonads, that is, bimonads whose underlying monad is idempotent. 

\begin{definition} An \textbf{idempotent bimonad} on a monoidal category $(\mathbb{X}, \otimes, I, \alpha, \lambda, \rho)$ is a bimonad $(T, \mu, \eta, m, m_I)$ on $(\mathbb{X}, \otimes, I, \alpha, \lambda, \rho)$ whose underlying monad $(T, \mu, \eta)$ is an idempotent monad. Similarly, a \textbf{symmetric idempotent bimonad} on a symmetric monoidal category \\
$(\mathbb{X}, \otimes, I, \alpha, \lambda, \rho, \sigma)$ is a symmetric bimonad $(T, \mu, \eta, m, m_I)$ on $(\mathbb{X}, \otimes, I, \alpha, \lambda, \rho, \sigma)$ which is also an idempotent bimonad. 
\end{definition}

We first observe that the Eilenberg-Moore category of an idempotent (symmetric) bimonad is essentially a full sub-(symmetric) monoidal category of the base category.

\begin{lemma}\label{Biidempotentlemma2} Let $(T, \mu, \eta, m, m_I)$ be an idempotent bimonad on a monoidal category $(\mathbb{X}, \!\otimes, \!I, \!\alpha, \lambda, \rho)$. Then if ${\eta_A: A \to T(A)}$ and $\eta_B: B \to T(B)$ are isomorphisms, $\eta_{A \otimes B}: A \otimes B \to T(A \otimes B)$ is an isomorphism with inverse $\eta^{-1}_{A \otimes B} = (\eta^{-1}_A \otimes \eta^{-1}_B) \circ m_{A,B}$. Therefore, 
\[(A, \eta^{-1}_A) \otimes^T (B, \eta^{-1}_B) = (A \otimes B, \eta^{-1}_{A \otimes B})\]
\end{lemma}
\begin{proof} Suppose that ${\eta_A: A \to T(A)}$ and $\eta_B: B \to T(B)$ are isomorphisms, so by Lemma \ref{idempotentalg}, $(A, \eta^{-1}_A)$ and $(B, \eta^{-1}_B)$ are $T$-algebras. By \cite[Theorem 7.1]{moerdijk2002monads}, we have that 
\[(A, \eta^{-1}_A) \otimes^T (B, \eta^{-1}_B) = (A \otimes B, (\eta^{-1}_A \otimes \eta^{-1}_B) \circ m_{A,B})\] 
is a $T$-algebra. However by Lemma \ref{idempotentalg}.(\ref{idempotentalg3}), it follows that $\eta_{A \otimes B}: A \otimes B \to T(A \otimes B)$ is an isomorphism with inverse $\eta^{-1}_{A \otimes B} = (\eta^{-1}_A \otimes \eta^{-1}_B) \circ m_{A,B}$ as desired. 
\end{proof}

Next, we observe that for idempotent bimonads, the monad unit at the monoidal unit is an isomorphism, which follows from the fact that the monoidal unit is an algebra of the monad. 

\begin{lemma}\label{Biidempotentlemma} Let $(T, \mu, \eta, m, m_I)$ be an idempotent bimonad on a monoidal category $(\mathbb{X}, \!\otimes, \!I, \!\alpha, \lambda, \rho)$. Then the following equality holds: $\eta_I \circ m_I = 1_{T(I)}$. Therefore $\eta_I: I \to T(I)$ is an isomorphism with inverse $m_I: T(I) \to I$, so $T(I) \cong I$.
\end{lemma}
\begin{proof} For a bimonad, $(I, m_I)$ is always a $T$-algebra. However, since $(T, \mu, \eta)$ is idempotent, by applying Lemma \ref{idempotentalg}.(\ref{idempotentalg3}), we have that $m_I = \eta_I^{-1}$. Thus $\eta_I$ is an isomorphism with inverse $m_I$, and so $\eta_I \circ m_I = 1_{T(I)}$ as desired. 
\end{proof}

The converse of Lemma \ref{Biidempotentlemma} is not always true. Indeed, there are bimonads such that $\eta_I \circ m_I = 1_{T(I)}$ and $T(I) \cong I$, but where the underlying monad is not idempotent. That said, in Lemma \ref{Hopfidempotentlemma} we will show that the converse is true for Hopf monads. Briefly, a Hopf monad is an idempotent monad if and only if $T(I) \cong I$. The proof is essentially that we have the following chain of isomorphisms: 
\[ TT(A) \cong T(I \otimes T(A)) \cong T(I) \otimes T(A) \cong I \otimes T(A) \cong T(A)  \]

\begin{definition} An \textbf{idempotent Hopf monad} on a monoidal category $(\mathbb{X}, \!\otimes, \!I\!, \!\alpha, \lambda, \rho)$ is a Hopf monad $(T, \mu, \eta, m, m_I)$ on $(\mathbb{X}, \otimes, I, \alpha, \lambda, \rho)$ whose underlying monad $(T, \mu, \eta)$ is an idempotent monad. Similarly, a \textbf{symmetric idempotent Hopf monad} on a symmetric monoidal category $(\mathbb{X}, \otimes, I, \alpha, \lambda, \rho, \sigma)$ is a symmetric Hopf monad $(T, \mu, \eta, m, m_I)$ on $(\mathbb{X}, \otimes, I, \alpha, \lambda, \rho, \sigma)$ which is also an idempotent Hopf monad. 
\end{definition}

  \begin{lemma} \label{Hopfidempotentlemma} Let $(T, \mu, \eta, m, m_I)$ be a Hopf monad on a monoidal category $(\mathbb{X}, \otimes, I, \alpha, \lambda, \rho)$. Then $(T, \mu, \eta, m, m_I)$ is an idempotent Hopf monad if and only if $\eta_I \circ m_I = 1_{T(I)}$. 
\end{lemma}
\begin{proof} By Lemma \ref{Biidempotentlemma}, we already have the $\Rightarrow$ direction. For $\Leftarrow$, first observe that for any bimonad, we have the following identity \cite[Proposition 2.6]{bruguieres2011hopf}: 
\begin{align} \label{hlmu1}
 \lambda_{T(A)} \circ (m_I \otimes 1_{T(A)}) \circ \mathsf{h}^{l}_{I,A} \circ T(\lambda^{-1}_A) = \mu_A
\end{align} 
Now suppose that $\eta_I \circ m_I = 1_{T(I)}$. By definition, we also have that $m_I \circ \eta_I = 1_I$ (\ref{symbimonad}), so $m_I$ is an isomorphism. Also, by the assumption of being a Hopf monad, $\mathsf{h}^{l}_{I,A}$ is also an isomorphism. Therefore by (\ref{hlmu1}), $\mu_A$ is equal to the composite of isomorphisms, and is therefore also an isomorphism. So we conclude that $(T, \mu, \eta)$ is an idempotent monad. 
\end{proof}

Let us now turn our attention to the traced setting. It turns out that every symmetric idempotent bimonad on a traced symmetric monoidal category is a traced monad, which we call idempotent traced monads. In this case, the Eilenberg-Moore category is a sub-traced symmetric monoidal category of the base category.  

\begin{definition} An \textbf{idempotent traced monad} on a traced symmetric monoidal category \\
$(\mathbb{X}, \otimes, I, \alpha, \lambda, \rho, \sigma, \mathsf{Tr})$ is a traced monad $(T, \mu, \eta, m, m_I)$ on $(\mathbb{X}, \otimes, I, \alpha, \lambda, \rho, \sigma, \mathsf{Tr})$ such that \\
$(T, \mu, \eta, m, m_I)$ is also a symmetric idempotent bimonad. 
\end{definition}

\begin{lemma} \label{idempotenttracelemma} Let $(\mathbb{X}, \otimes, I, \alpha, \lambda, \rho, \sigma, \mathsf{Tr})$ be a traced symmetric monoidal category, and let \\
$(T, \mu, \eta, m, m_I)$ be a symmetric idempotent bimonad on the underlying symmetric monoidal category $(\mathbb{X}, \otimes, I, \alpha, \lambda, \rho, \sigma)$. Then $(T, \mu, \eta, m, m_I)$ is an idempotent traced monad on the traced symmetric monoidal category $(\mathbb{X}, \otimes, I, \alpha, \lambda, \rho, \sigma, \mathsf{Tr})$. 
\end{lemma}
\begin{proof} Let $(X,\eta^{-1}_X)$, $(A,\eta^{-1}_A)$, and $(B,\eta^{-1}_B)$ be $T$-algebras and let
\[f: (A,\eta^{-1}_A) \otimes^T (X,\eta^{-1}_X) \to (B,\eta^{-1}_B) \otimes^T (X,\eta^{-1}_X)\] 
be a $T$-algebra morphism. By Lemma \ref{idempotentalg}.(\ref{idempotentalg4}), every map $A \to B$ in the base category gives a $T$-algebra morphism $(A,\eta^{-1}_A) \to (B,\eta^{-1}_B)$. Therefore, $\mathsf{Tr}^X_{A,B}(f): (A,\eta^{-1}_A) \to (B,\eta^{-1}_B)$ is a $T$-algebra morphism, and we conclude that $(T, \mu, \eta, m, m_I)$ is an idempotent traced monad. 
\end{proof}

Similarly, it turns out that every symmetric idempotent Hopf monad is trace-coherent.  

\begin{lemma}\label{idempotenthopftrace}  Let $(\mathbb{X}, \otimes, I, \alpha, \lambda, \rho, \sigma, \mathsf{Tr})$ be a traced symmetric monoidal category and \\
$(T, \mu, \eta, m, m_I)$ be a symmetric idempotent Hopf monad on the underlying symmetric monoidal category $(\mathbb{X}, \otimes, I, \alpha, \lambda, \rho, \sigma)$. Then $(T, \mu, \eta, m, m_I)$ is a trace-coherent Hopf monad on the traced symmetric monoidal category $(\mathbb{X}, \otimes, I, \alpha, \lambda, \rho, \sigma, \mathsf{Tr})$. 
\end{lemma}
\begin{proof} By Lemma \ref{idempotenttracelemma}, $(T, \mu, \eta, m, m_I)$ is a traced monad. Therefore, by Theorem \ref{mainthm}, we have that $(T, \mu, \eta, m, m_I)$ is a trace-coherent Hopf monad. 
\end{proof}

We now state the following crucial identity regarding idempotent traced monads. 

\begin{lemma}\label{traceidempotentlemma} Let $(T, \mu, \eta, m, m_I)$ be an idempotent traced monad on a traced symmetric monoidal category $(\mathbb{X}, \otimes, I, \alpha, \lambda, \rho, \sigma, \mathsf{Tr})$. Then $\mathsf{Tr}^{T(I)}_{I, T(I)}\left( m_{I,I} \circ T(\lambda^{-1}_I) \circ  \lambda_{T(I)} \right) = \eta_{I}$. 
\end{lemma}
\begin{proof} By Lemma \ref{Biidempotentlemma}, we know that $\eta_I: I \to T(I)$ is an isomorphism with inverse $m_I: T(I) \to I$. Using this fact, we can compute that: 
\begin{align*}
\mathsf{Tr}^{T(I)}_{I, T(I)}\left( m_{I,I} \circ T(\lambda^{-1}_I) \circ  \lambda_{T(I)} \right) &=~ \eta_{I} \circ m_I \circ \mathsf{Tr}^{T(I)}_{I, T(I)}\left( m_{I,I} \circ T(\lambda^{-1}_I) \circ  \lambda_{T(I)} \right) \tag{$\eta_{I} \circ m_I = 1_{T(I)}$ -- Lem.\ref{Biidempotentlemma}} \\
&=~ \eta_{I} \circ \mathsf{Tr}^{T(I)}_{I, I}\left( (m_I \otimes 1_{T(I)}) \circ m_{I,I} \circ T(\lambda^{-1}_I) \circ  \lambda_{T(I)} \right)\tag{\textbf{[Tightening]} -- (\ref{tightening2})} \\ 
&=~ \eta_{I} \circ \mathsf{Tr}^{T(I)}_{I, I}\left( \lambda^{-1}_{T(I)} \circ  \lambda_{T(I)} \right) \tag{Comonoidal Functor -- (\ref{smcendo})} \\ 
&=~ \eta_{I} \circ \mathsf{Tr}^{T(I)}_{I, I}\left( 1_{I \otimes T(I)} \right) \\
&=~ \eta_{I} \circ \mathsf{Tr}^{T(I)}_{I, I}\left( (1_I \otimes \eta_I) \circ (1_I \otimes m_I) \right) \tag{$\eta_{I} \circ m_I = 1_{T(I)}$ -- Lem.\ref{Biidempotentlemma}} \\
&=~ \eta_{I} \circ \mathsf{Tr}^{T(I)}_{I, I}\left( (1_I \otimes m_I) \circ (1_I \otimes \eta_I) \right) \tag{\textbf{[Sliding]} -- (\ref{sliding})} \\ 
&=~ \eta_{I} \circ \mathsf{Tr}^{T(I)}_{I, I}\left( 1_{I \otimes I} \right) \tag{Sym. Bimonad -- (\ref{symbimonad})} \\
&=~ \eta_{I} \circ \mathsf{Tr}^{T(I)}_{I, I}\left( \sigma_{I,I} \right) \tag{Sym. Mon. Cat. Coherence that $1_{I \otimes I} = \sigma_{I,I}$} \\
&=~ \eta_{I} \circ 1_I \tag{\textbf{[Yanking]} -- (\ref{yanking})} \\ 
&=~ \eta_I 
\end{align*} 
So we conclude that the desired equality holds. 
\end{proof}

The converse of Lemma \ref{traceidempotentlemma} is not always true, that is, there are non-idempotent traced monads where the mentioned equality holds. However, the converse is true for trace-coherent Hopf monads. 

\begin{proposition}\label{tracehopfidempotent} Let $(T, \mu, \eta, m, m_I)$ be a trace-coherent Hopf monad of a traced symmetric monoidal category $(\mathbb{X}, \otimes, I, \alpha, \lambda, \rho, \sigma, \mathsf{Tr})$. Then $(T, \mu, \eta, m, m_I)$ is a symmetric idempotent Hopf monad if and only if the following equality holds:
\begin{align} \label{tracemeta}
\mathsf{Tr}^{T(I)}_{I, T(I)}\left( m_{I,I} \circ T(\lambda^{-1}_I) \circ  \lambda_{T(I)} \right) = \eta_{I}
\end{align}
\end{proposition}
\begin{proof} By Lemma \ref{traceidempotentlemma}, we already have the $\Leftarrow$ direction. For the $\Rightarrow$ direction, assume that (\ref{tracemeta}) holds. By Lemma \ref{Hopfidempotentlemma}, to show that $(T, \mu, \eta)$ is idempotent, it is sufficient to show that $\eta_I \circ m_I = 1_{T(I)}$. First, observe that:
\begin{align*}
\lambda_{T(I)}:  (I, m_I) \otimes (T(I), \mu_I) \to (T(I), \mu_I) &&  T(\lambda^{-1}_I): (T(I), \mu_I) \to (T(I \otimes I), \mu_{I \otimes I})
\end{align*}
\[ m_{I,I}: (T(I \otimes I), \mu_{I \otimes I}) \to (T(I), \mu_I) \otimes^T (T(I), \mu_I)  \]
are all $T$-algebra morphisms, and so therefore their composite:
\[m_{I,I} \circ T(\lambda^{-1}_I) \circ  \lambda_{T(I)}: (I, m_I) \otimes^T (T(I), \mu_I) \to (T(I), \mu_I) \otimes^T (T(I), \mu_I)\]
is also a $T$-algebra morphism. Since $(T, \mu, \eta, m, m_I)$ is a trace-coherent Hopf monad, and therefore a traced monad, the trace of the above composite
\[\mathsf{Tr}^{T(I)}_{I, T(I)}\left( m_{I,I} \circ T(\lambda^{-1}_I) \circ  \lambda_{T(I)} \right): (I, m_I) \to (T(I), \mu_I)\] 
is also a $T$-algebra morphism. Therefore by (\ref{tracemeta}), we have that $\eta_I: (I, m_I) \to (T(I), \mu_I)$ is a $T$-algebra morphism. Explicitly, this means that the following equality holds: 
\begin{align*}
  \eta_I \circ m_I = \mu_I \circ T(\eta_I)  
\end{align*}
However, by the monad identities (\ref{monadeq}), the right-hand side of the above equality is always equal to the identity, $\mu_I \circ T(\eta_I)  =1_{T(I)}$. So we have that $\eta_I \circ m_I = 1_{T(I)}$. So by Lemma \ref{Hopfidempotentlemma}, we conclude that $(T, \mu, \eta)$ is idempotent. 
\end{proof}
    
\section{Traced CoCartesian Monoidal Categories}\label{coCarttracemonadsec}

Any category with finite coproducts is a symmetric monoidal category, called a coCartesian monoidal category. Traced coCartesian monoidal categories are the dual notion of traced Cartesian monoidal categories, in the sense that the opposite category of one is a form of the other. In a traced coCartesian monoidal category, the trace operator captures the notion of feedback via iteration. In this section, we will explain why every trace-coherent Hopf monad on a traced coCartesian monoidal category is in fact idempotent. 

As for the product case, we are interested in the symmetric monoidal structure induced by finite coproducts. So we will define coCartesian monoidal categories as symmetric monoidal categories whose monoidal product is a coproduct and whose monoidal unit is an initial object. Once again, any category with finite coproducts is a coCartesian monoidal category where the symmetric monoidal structure can be fully derived from the universal property of the coproduct, and conversely, every coCartesian monoidal category is a category with finite coproducts. 

\begin{definition} \cite[Section 6.2 \& Section 6.4]{selinger2010survey} A \textbf{coCartesian monoidal category} is a symmetric monoidal category $(\mathbb{X}, \oplus, \bot, \alpha, \lambda, \rho, \sigma)$ whose monoidal structure is a given by \textbf{finite coproducts}, that is: 
\begin{enumerate}[{\em (i)}]
\item The monoidal unit $\bot$ is an \textbf{initial object}, that is, for every object $A$ there exists a unique map $i_A: \bot \to A$. 
\item For every pair of objects $A$ and $B$, $A \oplus B$ a \textbf{coproduct} of $A$ and $B$ with \textbf{injection} maps $\iota_0: A \to A \oplus B$ and $\iota_1: B \to A \oplus B$ defined as following composites: 
\begin{align*}
\iota_0 :=   \xymatrixcolsep{3pc}\xymatrix{ A \ar[r]^-{\rho^{-1}_A} & A \oplus \bot \ar[r]^-{1_A \times i_B} & A \oplus B } &&   \iota_1 :=   \xymatrixcolsep{3pc}\xymatrix{ B \ar[r]^-{\lambda^{-1}_B} & \bot \oplus A \ar[r]^-{i_A \times 1_B} & A \oplus B }
\end{align*}
that is, for every pair of maps $f_0: A \to C$ and $f_1: B \to C$, there is a \emph{unique} map ${[ f_0, f_1 ]: A \oplus B \to C}$, called the \textbf{copairing} of $f_0$ and $f_1$, such that:
\begin{align*}
    [f_0, f_1] \circ \iota_0 = f_0 && [f_0, f_1] \circ \iota_1 = f_1
\end{align*}
\end{enumerate} 
A \textbf{traced coCartesian monoidal category} is a traced symmetric monoidal category \\
$(\mathbb{X}, \oplus, \bot, \alpha, \lambda, \rho, \sigma, \mathsf{Tr})$ whose underlying symmetric monoidal category $(\mathbb{X}, \oplus, \bot, \alpha, \lambda, \rho, \sigma)$ is a coCartesian monoidal category. 
\end{definition}

A traced coCartesian monoidal category can equivalently be characterized as a coCartesian monoidal category $\mathbb{X}$ equipped with an iteration operator, which is a family of operators $\mathsf{Iter}^X_A: \mathbb{X}(X, A \oplus X) \to \mathbb{X}(X, A)$, such that the dual axioms of a Conway operator hold. Since iteration operators do not play a crucial role in this paper, we will not review the full definition here and invite curious readers to see \cite[Section 6.4]{selinger2010survey}. 

Now, it is first important to mention that unlike in the Cartesian case, not every monad on a coCartesian monoidal category is necessarily a (symmetric) bimonad. Furthermore, given a symmetric bimonad on a coCartesian monoidal category, while the Eilenberg-Moore category will be a symmetric monoidal category, it is not automatically the case that the Eilenberg-Moore category is again a coCartesian monoidal category. As such, for a traced monad on a traced coCartesian monoidal category, the Eilenberg-Moore category will be a traced symmetric monoidal category, but not necessarily a traced coCartesian monoidal category. So in particular, an arbitrary traced monad does not necessarily lift the iteration operator. Therefore, in general, traced monads on traced coCartesian monoidal categories cannot be described in terms of the induced iteration operator. 

Now let us prove that trace-coherent Hopf monads on traced coCartesian monoidal categories must be idempotent. First, let us prove a more general result in the case when the monoidal unit is an initial object. 

\begin{lemma} Let $(\mathbb{X}, \otimes, I, \alpha, \lambda, \rho, \sigma, \mathsf{Tr})$ be a traced symmetric monoidal category such that the unit $I$ is an initial object, and let $(T, \mu, \eta, m, m_I)$ be a symmetric Hopf monad on the underlying symmetric monoidal category $(\mathbb{X}, \otimes, I, \alpha, \lambda, \rho, \sigma)$. Then $(T, \mu, \eta, m, m_I)$ is a trace-coherent Hopf monad (or equivalently a traced monad) on $(\mathbb{X}, \otimes, I, \alpha, \lambda, \rho, \sigma, \mathsf{Tr})$ if and only if $(T, \mu, \eta, m, m_I)$ is a symmetric idempotent Hopf monad.  
\end{lemma}
\begin{proof} The $\Leftarrow$ direction follows from Lemma \ref{idempotenthopftrace}. For the $\Rightarrow$ direction, assume that $I$ is an initial object. By the universal property of an initial object, there is a unique map of type $I \to T(I)$. Therefore, it follows that $\mathsf{Tr}^{T(I)}_{I, T(I)}\left( m_{I,I} \circ T(\lambda^{-1}_I) \circ  \lambda_{T(I)} \right) = \eta_{I}$. Then by Proposition \ref{tracehopfidempotent}, it follows that $(T, \mu, \eta, m, m_I)$ is a symmetric idempotent Hopf monad. 
\end{proof}

By definition, the monoidal unit of a traced coCartesian monoidal category is an initial object. Therefore, we conclude that:  

\begin{corollary}\label{tracecoherentcocart} Let $(\mathbb{X}, \oplus, \bot, \alpha, \lambda, \rho, \sigma, \mathsf{Tr})$ be a traced coCartesian monoidal category, and let \\
$(T, \mu, \eta, m, m_I)$ be a symmetric Hopf monad on the underlying symmetric monoidal category \\
$(\mathbb{X}, \oplus, \bot, \alpha, \lambda, \rho, \sigma)$. Then $(T, \mu, \eta, m, m_I)$ is a trace-coherent Hopf monad (or equivalently a traced monad) on $(\mathbb{X}, \oplus, \bot, \alpha, \lambda, \rho, \sigma, \mathsf{Tr})$ if and only if $(T, \mu, \eta, m, m_I)$ is a symmetric idempotent Hopf monad. 
\end{corollary}

\section{Separating Examples}\label{counterexamplesec}

In this section we provide separating examples to show that not all traced monads are Hopf monads, and also that not all symmetric Hopf monads are trace-coherent (or equivalently traced monads).  

Our first example is a traced monad on a compact closed category that lifts the traced monoidal structure but not the compact closed structure, and is therefore not a Hopf monad. 

\begin{example}\label{Zexsep} \normalfont Let $\mathbb{Z}$ be the set of integers. Let $\mathbb{Z}_{\leq}$ be the standard poset category, that is, the category whose objects are integers $n \in \mathbb{Z}$, where there is a map $n \to m$ if and only if $n \leq m$. $\mathbb{Z}_{\leq}$ is a compact closed category where the monoidal product is given by addition, so $n \otimes m = n +m$ and $I = 0$, the dual is given by the negative $n^\ast = -n$, the cap is the equality $0 = n + (-n)$, and the cup is the equality $(-n) + n = 0$. The induced trace operator records the fact that if $n + x \leq m + x$ then $n \leq m$. Now define the endofunctor $N: \mathbb{Z}_{\leq} \to \mathbb{Z}_{\leq}$ which maps negative integers to zero and non-negative integers to themselves: 
\[N(n) = \begin{cases}
n & \text{if $n \geq 0$} \\
0  & \text{if $n < 0$}
\end{cases} \]
Then $N$ is an idempotent monad since $n \leq N(n)$, which gives the unit, and $NN(n) = N(n)$, which gives the multiplication that is an isomorphism (in this case the identity). $N$ is also a symmetric idempotent bimonad since $N(n+m) \leq N(n) + N(m)$ and $N(0) = 0$. Therefore by Lemma \ref{idempotenttracelemma}, $N$ is an idempotent traced monad. However, $N$ is not a Hopf monad since $N(n + N(m))$ does not equal $N(n) + N(m)$ for all $n, m \in \mathbb{Z}$. Indeed, setting $n= -2$ and $m = 1$, we have that:
\[N(-2) + N(1) = 0 + 1 = 1\] 
but on the other hand, we have that:
\[N(-2 + N(1)) = N(-2 +1) = N(-1) = 0\] 
So $N(-2) + N(1) \neq N(-2 + N(1))$, and so $N$ is not a Hopf monad. So we conclude that $N$ is a traced monad which is not a Hopf monad.
\end{example}

Our next example is a traced monad induced by a cocommutative bimonoid which is not a Hopf monoid, and therefore not a Hopf monad. This example is particularly important in domain theory: 

\begin{example}\normalfont Let $\omega\text{-}\mathsf{CPPO}$ be the category whose objects are $\omega$-complete partial orders with a bottom element ($\omega$-cppo) and whose maps are (Scott) continuous functions between them. $\omega\text{-}\mathsf{CPPO}$ is a traced Cartesian (closed) monoidal category where the monoidal product is given by the Cartesian product, $X \otimes Y = X \times Y$ and $I = \lbrace \bot \rbrace$, and where the Conway operator and trace operator are induced by the parametrized Tarski least fixed-point operator. Indeed, for any continous function $f: X \to X$, there exists a least fixed point for $f$, that is, there exists an $\mathsf{fix}(f) \in X$ such that $f\left( \mathsf{fix}(f) \right) = \mathsf{fix}(f)$ and for every $x \in X$ such that $f(x) = x$, we have that $\mathsf{fix}(f) \leq x$. So for a continuous function $g: A \times X \to X$, its parametrized fixed point $\mathsf{Fix}^X_A(g): A \to X$ is defined as the least fixed point of the curry of $g$, that is, $\mathsf{Fix}^X_A(g)(a) = \mathsf{fix}\left( g(a,-) \right)$. Then for a continuous function $h: A \times X \to B \times X$, where $h(a,x) = \left( h_0(a,x)_B, h_1(a,x) \right)$ for continuous functions $h_0: A \times X \to B$ and $h_1: A \times X \to X$, its trace ${\mathsf{Tr}^X_{A,B}(h): A \to B}$ is defined as follows:
\[ \mathsf{Tr}^X_{A,B}(h)(a) = h_0\left( a, \mathsf{fix}\left( h_1(a,-) \right) \right)  \]
Now consider the Sierpinski space $\Sigma$, which is the two-element $\omega$-cppo where one is the bottom element and the other is the top element, $\Sigma=\{\bot\leq\top\}$. The Sierpinski space $\Sigma$ is a monoid where the multiplication $(-)\wedge(-):\Sigma\times\Sigma\rightarrow\Sigma$ is defined as follows:
\begin{align*}
    \top\wedge\top=\top && \top\wedge\bot= \bot && \bot\wedge\top=\bot && \bot\wedge\bot=\bot
\end{align*}
and the unit is the top element $\top$ (that is, $u: \lbrace \bot \rbrace \to \Sigma$ is defined as $u(\bot) = \top$). Now recall that a cocommutative bimonoid can be defined as a cocommutative comonoid which is also a monoid such that the multiplication and unit are comonoid morphisms. Furthermore, in a Cartesian monoidal category, every object has a unique cocommutative comonoid structure and every map is a comonoid morphism. As such, every monoid in a Cartesian monoidal category is uniquely a cocommutative bimonoid, since the multiplication and unit are comonoid morphisms. So in $\mathbf{Cppo}$, $\Sigma$ is a cocommutative bimonoid where the comultiplication $\Delta: \Sigma \to \Sigma \times \Sigma$ and counit $e: \Sigma \to \lbrace \bot \rbrace$ are defined as follows: 
\begin{align*}
    \Delta(\bot)=(\bot, \bot) && \Delta(\top)=(\top, \top) && e(\top) = \bot && e(\bot) = \bot 
\end{align*}
Therefore $\Sigma \times -$ is a symmetric bimonad on $\omega\text{-}\mathsf{CPPO}$. By Lemma \ref{fixedmonad}, to show that $\Sigma \times -$ is also a traced monad, it suffices to show that Conway operator of $\omega\text{-}\mathsf{CPPO}$ lifts to the Eilenberg-Moore category, in other words, that the parametrized fixed point of a $\Sigma$-module morphism is again a $\Sigma$-module morphism. To prove this we must first discuss \emph{strict} functions. A continuous function $h: X \to Y$ is strict if it preserves the bottom element, that is, $h(\bot) = \bot$. Furthermore, strict functions are compatible with the Conway operator, that is, if $h: X \to Y$ is strict, then any for continuous functions $f: A \times X \to X$ and $g: A \times Y \to Y$ such that $g \circ (1_A \times h) = h \circ f$, then $\mathsf{Fix}^Y_A(g) = h \circ \mathsf{Fix}^X_A(f)$ \cite[Section 5.3]{hasegawa2004uniformity}. Explicitly, if for all $a\in A$ and $x \in X$ the left side equality holds, then the right side equality holds: 
\begin{align}\label{strictfixed}
g(a, h(x)) = h(f(a,x)) \Longrightarrow \mathsf{fix}(g(a,-)) = h(\mathsf{fix}(f(a,-))  
\end{align}
Next, let us discuss $\Sigma$-modules. A $\Sigma$-module $(A, (-) \bullet (-))$ amounts to a $\omega$-cppo $A$ equipped with a continuous map $(-) \bullet (-):\Sigma\times A \to A$ such that for all $a \in A$ and $y_1,y_2 \in \Sigma$: 
\begin{align*}
\top \bullet a = a && (y_1 \wedge y_2) \bullet a = y_1 \bullet (y_2 \bullet a)
\end{align*}
Now let $(A, (-) \bullet (-))$ and $(X, (-) \ast (-))$ be $\Sigma$-modules, and let 
\[f: (A, (-) \bullet (-)) \times^\Sigma (X, (-) \ast (-)) \to (X, (-) \ast (-))\] 
be a $\Sigma$-module morphisms, that is, the following equality holds for all $y \in \Sigma$, $a \in A$, and $x \in X$: 
\begin{align}\label{sigmamorph}
 f \left( y \bullet a, y \ast x \right) = y \ast f(x,a)    
\end{align}
We need to show that $\mathsf{Fix}^X_A(f)$ is also a $\Sigma$-module morphism. Now for all $y \in \Sigma$, since $(y,\bot) \leq (\top, \bot)$ it follows that $y \ast \bot \leq \top \ast \bot = \bot$, but since $\bot$ is the bottom element, this implies that $y \ast \bot = \bot$, and therefore $y \ast (-): X \to X$ is strict. Therefore applying (\ref{strictfixed}) to (\ref{sigmamorph}), we then have the following: 
\begin{align}
 f \left( y \bullet a, y \ast x \right) = y \ast f(x,a) \Rightarrow  \mathsf{fix}\left(  f \left( y \bullet a, - \right) \right) = y \ast \mathsf{fix}(f(a,-)) 
\end{align}
Re-expressing the right-hand-side, we get that: 
\[ \mathsf{Fix}^X_A(f)(y \bullet a) = y \ast \mathsf{Fix}^X_A(f)(a) \] 
So $\mathsf{Fix}^X_A(f): (A, (-) \bullet (-)) \to (X, (-) \ast (-))$ is a $\Sigma$-module morphism. So we conclude that $\Sigma \times -$ is a traced monad on $\mathbf{Cppo}$. However, $\Sigma \times -$ is not a Hopf monad since $\Sigma$ is not a Hopf monoid. Indeed, if $\Sigma$ was a Hopf monoid, there would be a continuous map $S: \Sigma \to \Sigma$ where in particular $S(\bot) \wedge \bot = \top$. However, by definition $x \wedge \bot = \bot$ for all $x \in \Sigma$. Since $\bot \neq \top$, such an $S$ cannot exist and therefore $\Sigma$ cannot have an antipode, and so $\Sigma$ is not a Hopf monoid. Therefore, $\Sigma \times -$ is a traced monad which is not a Hopf monad. 
\end{example}

The above example shows that for a cocommutative bimonoid, it is not necessary for it to also be a cocommutative Hopf monoid for its induced symmetric bimonad to be a traced monad. That said, here is an example of a cocommutative bimonoid whose induced symmetric bimonad is not a traced monad. 

\begin{example} \normalfont Consider again the Sierpinski space $\Sigma$ in the traced Cartesian monoidal category $\omega\text{-}\mathsf{CPPO}$ as in the previous example. The Sierpinski space $\Sigma$ has another cocommutative bimonoid structure, where the comonoid structure is the same as in the above example, but where this time the monoid structure is given by unit element being the bottom element $\bot$ and the multiplication $(-) \vee (-): \Sigma \times \Sigma \to \Sigma$ defined as follows: 
\begin{align*}
\top \vee \top = \top && \top \vee \bot = \top && \bot \vee \top = \top && \bot \vee \bot = \bot     
\end{align*}
To avoid confusion, we will denote the Sierpinski space with this cocommutative bimonoid structure by $\Sigma_\vee$. Then $\Sigma_\vee \times -$ is a symmetric bimonad on $\omega\text{-}\mathsf{CPPO}$. In general, every monoid is a module over itself where the action is the multiplication, so $(\Sigma_\vee, (-) \vee (-))$ is a $\Sigma_\vee$-module. We also have that the projection:
\[\pi_0: (\Sigma_\vee, (-) \vee (-)) \times^{\Sigma_\vee} (\Sigma_\vee, (-) \vee (-)) \to (\Sigma_\vee, (-) \vee (-))\]
is a $\Sigma_\vee$-module morphism. Now the parametrized fixed point of the projection \[\mathsf{Fix}^{\Sigma_\vee}_{\Sigma_\vee}(\pi_0): \Sigma_\vee \to \Sigma_\vee\] 
is the map which sends everything to the bottom element, $\mathsf{Fix}^{\Sigma_\vee}_{\Sigma_\vee}(\pi_0)(x) = \bot$. However, $\mathsf{Fix}^{\Sigma_\vee}_{\Sigma_\vee}(\pi_0)$ is not a $\Sigma_\vee$-module morphism $(\Sigma_\vee, (-) \vee (-)) \to (\Sigma_\vee, (-) \vee (-))$. Indeed, on the one hand $\mathsf{Fix}^{\Sigma_\vee}_{\Sigma_\vee}(\pi_0)(\top \vee \top)= \bot$, but on the other hand $\top \vee \mathsf{Fix}^{\Sigma_\vee}_{\Sigma_\vee}(\pi_0)(\top) = \top \vee \bot = \top$. Therefore, $\mathsf{Fix}^{\Sigma_\vee}_{\Sigma_\vee}(\pi_0)(\top \vee \top) \neq \top \vee \mathsf{Fix}^{\Sigma_\vee}_{\Sigma_\vee}(\pi_0)(\top)$, so $\mathsf{Fix}^{\Sigma_\vee}_{\Sigma_\vee}(\pi_0)$ is not a $\Sigma_\vee$-module morphism. Thus the Conway operator $\mathsf{Fix}$ does not lift to the category of $\Sigma_\vee$-modules and so, as explained in Section \ref{carttracesec}, neither does the trace operator. So $\Sigma_\vee \times -$ is not a traced monad. 
\end{example}

Providing an example of a symmetric Hopf monad that is not a traced monad is surprisingly not as straightforward as one would hope. Indeed, the fact that every symmetric Hopf monad on a compact closed category is a traced monad covers a large class of examples of traced symmetric monoidal categories, while the fact that every cocommutative Hopf monoid in a traced symmetric monoidal category induces a traced monad covers a large class of examples of symmetric Hopf monads. Therefore we must look beyond these sorts of examples, and unfortunately, the desired sort of example does not seem to arise naturally or appear in the literature. 

There are two key concepts to understanding the following example. The first is the notion of a symmetric monoidal category with coproducts such that the monoidal product distributes over the coproduct. The second is the fact that a symmetric monoidal category can have more than one trace operator. 

\begin{example} \normalfont Let $\mathbb{X}$ be a category with finite coproducts $\oplus$, injection maps $\iota_j$, copairing operator $[-,-]$, and initial object $\bot$. For every object $A$, define the codiagonal map $\nabla_A: A \oplus A \to A$ as the copairing of the identity map with itself, $\nabla_a = [1_A, 1_A]$. Consider the category $\mathbb{X} \times \mathbb{X}$, whose objects and maps are pairs of objects and maps of $\mathbb{X}$, and where composition and identities are defined pointwise. Define the endofunctor $T: \mathbb{X} \times \mathbb{X} \to \mathbb{X} \times \mathbb{X}$ on objects as $T(A, B) = (A \oplus B, A \oplus B)$ and on maps as $T(f,g) = (f \oplus g, f \oplus g)$, and also define the natural transformations:
\begin{align*}
    \mu_{(A,B)}: TT(A,B) \to T(A,B) && \eta_{(A,B)}: (A,B) \to T(A,B)
\end{align*}
 respectively as follows: 
\[ \mu_{(A,B)} :=   \xymatrixcolsep{5pc}\xymatrix{ TT(A,B) = \left( (A \oplus B) \oplus (A \oplus B), (A \oplus B) \oplus (A \oplus B) \right) \ar[d]^-{\left( \nabla_{A \oplus B}, \nabla_{A \oplus B} \right) } \\  (A \oplus B, A \oplus B) = T(A,B)  } \] 
\[ \eta_{(A,B)} :=  \xymatrixcolsep{5pc}\xymatrix{ (A,B) \ar[r]^-{\left( \iota_0, \iota_1  \right) } & (A \oplus B, A \oplus B) = T(A,B)  }\] 
Then $(T, \mu, \eta)$ is a monad on $\mathbb{X} \times \mathbb{X}$. The Eilenberg-Moore category is isomorphic to the base category, that is, $(\mathbb{X} \times \mathbb{X})^T \cong \mathbb{X}$, and so the forgetful functor can be identified with the diagonal functor $\Delta: \mathbb{X} \to \mathbb{X} \times \mathbb{X}$, defined on objects as $\Delta(X) =(X,X)$ and on maps as $\Delta(f) = (f,f)$. In particular, note that we may re-express $T$ using $\Delta$ by $T(A,B) = \Delta(A \oplus B)$ and $T(f,g) = \Delta(f \oplus g)$, and so we may write the monad multiplicaiton as $\mu_{(A,B)} = \Delta\left(\nabla_{A \oplus B}\right)$. Now suppose that $\mathbb{X}$ is also a symmetric monoidal category with monoidal product $\otimes$ and monoidal unit $I$. Then $\mathbb{X} \times \mathbb{X}$ is also a symmetric monoidal category whose structure is defined pointwise, that is, the monoidal product is $(A,B) \otimes (C,D) = (A \otimes B, C \otimes D)$ and the monoidal unit $(I,I)$. Now suppose that in $\mathbb{X}$, the coproduct structure distributes over the monoidal structure. Explicitly, the natural transformation $\partial_{A,B,C}: (A \otimes C) \oplus (B \otimes C) \to (A \oplus B) \otimes C$ defined as: 
\[ \partial_{A,B,C} :=   \xymatrixcolsep{5pc}\xymatrix{ (A \otimes C) \oplus (B \otimes C) \ar[r]^-{ [ \iota_0 \otimes 1_C, \iota_1 \otimes 1_C ] } & (A \oplus B) \otimes C }\] 
and the unique map from the initial object $i_{A \otimes \bot}: \bot \to \bot \otimes A$ are isomorphisms, so $ (A \otimes C) \oplus (B \otimes C) \cong (A \oplus B) \otimes C$ and $\bot \cong \bot \otimes A$. Then define the natural transformation 
\[m_{(A,B),(C,D)}:  T\left( (A,B) \otimes (C,D) \right) \to  T(A,B) \otimes T(C,D)\] 
and the map $m_(I,I): T(I,I) \to (I,I)$ respectively as follows (as to not overload notation, we write them using $\Delta$): 
\[ m_{(A,B),(C,D)} :=   \xymatrixcolsep{5pc}\xymatrix{\Delta\left(  (A \otimes C) \oplus (B \otimes D) \right) \ar[r]^-{\Delta \left( [ \iota_0 \otimes \iota_0, \iota_1 \otimes \iota_1 ] \right) } &  \Delta\left((A \oplus B) \otimes (C \oplus D)  \right) } \]
\[ m_{(I,I)} :=   \xymatrixcolsep{5pc}\xymatrix{\Delta\left( I \oplus I \right) \ar[r]^-{\Delta \left( \nabla_{I} \right) } & \Delta(I)  } \]
Then $(T, \mu, \eta, m, m_{(I,I)})$ is a symmetric Hopf monad on $\mathbb{X} \times \mathbb{X}$ where the left fusion operator and its inverse: 
\begin{align*}
    {\mathsf{h}^l}_{(A,B), (C,D)}: T\left( (A,B) \otimes T(C,D) \right) \to T(A,B) \otimes T(C,D) \\ 
{\mathsf{h}^l}^{-1}_{(A,B), (C,D)}: T(A,B) \otimes T(C,D) \to T\left( (A,B) \otimes T(C,D) \right)    
\end{align*}
are given by the distribution isomorphism (again writing them using $\Delta$): 
\[ \mathsf{h}^l_{(A,B), (C,D)} :=   \xymatrixcolsep{5pc}\xymatrix{\Delta\left( \left(A \otimes (C \oplus D) \right) \oplus \left(B \otimes (C \oplus D) \right) \right) \ar[r]^-{\Delta \left( \partial_{A,B,C \oplus D} \right) } &  \Delta\left((A \oplus B) \otimes (C \oplus D)  \right) } \]
\[ {\mathsf{h}^l}^{-1}_{(A,B), (C,D)} :=   \xymatrixcolsep{5pc}\xymatrix{  \Delta\left((A \oplus B) \otimes (C \oplus D)  \right) \ar[r]^-{\Delta \left( \partial^{-1}_{A,B,C \oplus D} \right) } &  \Delta\left( \left(A \otimes (C \oplus D) \right) \oplus \left(B \otimes (C \oplus D) \right) \right)  } \]
The induced symmetric monoidal structure on $\mathbb{X}$ (seen as the Eilenberg-Moore category) is precisely the one we started with on $\mathbb{X}$. Lastly, suppose that $\mathbb{X}$ has two distinct trace operators $\mathsf{Tr}$ and $\overline{\mathsf{Tr}}$. This induces a trace operator $\mathsf{Tr} \times \overline{\mathsf{Tr}}$ on $\mathbb{X} \times \mathbb{X}$ defined pointwise, that is:
\[(\mathsf{Tr} \times \overline{\mathsf{Tr}})(f,g) = (\mathsf{Tr}(f), \overline{\mathsf{Tr}}(g))\] 
So $\mathbb{X} \times \mathbb{X}$ is a traced symmetric monoidal category. For either of the trace operators on $\mathbb{X}$, the diagonal functor $\Delta: \mathbb{X} \to \mathbb{X} \times \mathbb{X}$, interpreted as the forgetful functor, does not preserve the trace operator since clearly: 
\begin{align*}
 (\mathsf{Tr} \times \overline{\mathsf{Tr}})\left( \Delta(f) \right) = (\mathsf{Tr} \times \overline{\mathsf{Tr}})(f,f) = (\mathsf{Tr}(f), \overline{\mathsf{Tr}}(f)) \neq (\mathsf{Tr}(f), \mathsf{Tr}(f)) = \Delta(\mathsf{Tr}(f) ) \\  
  (\mathsf{Tr} \times \overline{\mathsf{Tr}})\left( \Delta(f) \right) = (\mathsf{Tr} \times \overline{\mathsf{Tr}})(f,f) = (\mathsf{Tr}(f), \overline{\mathsf{Tr}}(f)) \neq (\mathsf{Tr}(f), \mathsf{Tr}(f)) = \Delta(\mathsf{Tr}(f) )
\end{align*}
So we conclude that $(T, \mu, \eta, m, m_{(I,I)})$ is not a traced monad (or equivalently trace-coherent) with respect to the trace operator $\mathsf{Tr} \times \overline{\mathsf{Tr}}$. Of course, taking instead the trace operator $\mathsf{Tr} \times \mathsf{Tr}$ or $\overline{\mathsf{Tr}} \times \overline{\mathsf{Tr}}$, or if $\mathsf{Tr} = \overline{\mathsf{Tr}}$, then $(T, \mu, \eta, m, m_{(I,I)})$ would be a traced monad. 
\end{example}

Truthfully, the above example is somewhat complex. In fact, finding a symmetric monoidal category with distributive coproducts and two distinct trace operators is not an obvious task either. There are examples with two distinct trace operators, such as the free traced symmetric monoidal category where one of the trace operators is the free one while the other is the one induced from the free symmetric monoidal structure \cite{abramsky2005abstract}, and also the category of continuous lattices where one trace operator is given by the least fixed point operator and the other given by the greatest continuous fixed point operator \cite{simpson1993characterisation}. To obtain a traced symmetric monoidal category with distributive coproducts, one can consider the biproduct completion (of the semi-additive category obtained as the direct image of the powerset functor),
as suggested by Naohiko Hoshino, in private communications with the authors. Given a category $\mathbb{X}$, we define a category $B\mathbb{X}$ whose objects are finite families $\{X_i\}_{i\in I}$ of objects of $\mathbb{X}$ and whose maps $\{X_i\}_{i\in I}\to \{Y_j\}_{j\in J}$ are families $\{f_{i,j}: X_i \to Y_j\}_{(i,j)\in I\times J}$ of maps in $\mathbb{X}$. This new category $B\mathbb{X}$ has finite biproducts where the zero object is given by the empty family, while the biproduct is given by the disjoint union of families. When $\mathbb{X}$ is a traced symmetric monoidal category, $B\mathbb{X}$ is again a traced symmetric monoidal category with tensor unit $\{I\}$ and the tensor product
$\{X_i\}_{i\in I}\otimes \{Y_j\}_{j\in J}=\{X_i\otimes Y_j\}_{(i,j)\in I\times J}$, which this time has distributive biproducts. The biproduct completion is independent of the trace operator, so if one starts with a symmetric monoidal category with two distinct trace operators, its biproduct completion will also have two distinct trace operators. Therefore, using this we obtain a suitable starting traced symmetric monoidal category on which to apply the construction of the above example and obtain a symmetric Hopf monad which is not a traced monad. 

Ideally, it would be of interest to find an example that is not as complicated. A potential example would be to find a non-idempotent symmetric Hopf monad on a coCartesian monoidal category, which can therefore not be trace-coherent by Corollary \ref{tracecoherentcocart}. 

\bibliographystyle{plainnat}
\bibliography{tracereferences}

\end{document}

%% file: fig-def.tex
\newcounter{doubleV}
\newcounter{XplusH}
\newcounter{YplusV}

\newcount\mycount
\newcommand{\trace}[4]
    {\setcounter{doubleV}{#4}
     \addtocounter{doubleV}{#4}
     \setcounter{XplusH}{#1}
     \addtocounter{XplusH}{#3}
     \setcounter{YplusV}{#2}
     \addtocounter{YplusV}{#4}
     \put(#1,#2){\oval(\value{doubleV},\value{doubleV})[l]}
     \put(#1,\value{YplusV}){\line(1,0){#3}}
     \put(\value{XplusH},#2){\oval(\value{doubleV},\value{doubleV})[r]}
     }

\newdimen\cboxwidth
\cboxwidth=\unitlength
\newdimen\cboxheight
\cboxheight=\unitlength
\newcommand{\rgbfbox}[8]{%
\cboxwidth=\unitlength
\cboxheight=\unitlength
\multiply \cboxwidth by #3
\multiply \cboxheight by #4
{\color[rgb]{#5,#6,#7}\put(#1,#2){\rule{\cboxwidth}{\cboxheight}}}
\put(#1,#2){\framebox(#3,#4){#8}}
\cboxwidth=\unitlength
\cboxheight=\unitlength
}
\newcommand{\redfbox}[5]{\rgbfbox{#1}{#2}{#3}{#4}{1}{.7}{.8}{#5}}
\newcommand{\greenfbox}[5]{\rgbfbox{#1}{#2}{#3}{#4}{0}{1}{0}{#5}}
\newcommand{\violetfbox}[5]{\rgbfbox{#1}{#2}{#3}{#4}{.8}{.8}{1}{#5}}
\newcommand{\bluefbox}[5]{\rgbfbox{#1}{#2}{#3}{#4}{0}{1}{1}{#5}}
\newcommand{\yellowfbox}[5]{\rgbfbox{#1}{#2}{#3}{#4}{1}{1}{0}{#5}}

\newdimen\cboxwidth
\cboxwidth=\unitlength
\newdimen\cboxheight
\cboxheight=\unitlength
\newcommand{\rgbbox}[8]{%
\cboxwidth=\unitlength
\cboxheight=\unitlength
\multiply \cboxwidth by #3
\multiply \cboxheight by #4
{\color[rgb]{#5,#6,#7}\put(#1,#2){\rule{\cboxwidth}{\cboxheight}}}
\put(#1,#2){\makebox(#3,#4){#8}}
\cboxwidth=\unitlength
\cboxheight=\unitlength
}
\newcommand{\graybox}[5]{\rgbbox{#1}{#2}{#3}{#4}{.9}{.9}{.9}{#5}}

\newcommand{\lineseg}[4]{\qbezier(#1,#2)(#1,#2)(#3,#4)}
\newcommand{\idline}[3]{%
\put(#1,#2){\line(1,0){#3}}
}

\newcounter{XX}
\newcounter{YY}
\newcommand{\mul}[2]{%
\setcounter{XX}{#1}%
\setcounter{YY}{#2}%
\addtocounter{XX}{5}%
\addtocounter{YY}{5}%
\rgbfbox{#1}{#2}{5}{10}{1}{1}{0}{}%
{\color[rgb]{1,.8,.8}\put(\value{XX},\value{YY}){\circle*{10}}}%
\put(\value{XX},\value{YY}){\circle{10}}%
\rgbbox{#1}{#2}{5}{10}{1}{.8}{.8}{}%
}
\newcommand{\mulunit}[2]{%
\setcounter{XX}{#1}%
\setcounter{YY}{#2}%
\addtocounter{XX}{5}%
\addtocounter{YY}{5}%
{\color[rgb]{1,.8,.8}\put(\value{XX},\value{YY}){\circle*{10}}}%
\put(\value{XX},\value{YY}){\circle{10}}%
}
\newcommand{\comul}[2]{%
\setcounter{XX}{#1}%
\setcounter{YY}{#2}%
\addtocounter{XX}{5}%
\addtocounter{YY}{5}%
\rgbfbox{\value{XX}}{#2}{5}{10}{0}{.8}{.8}{}%
{\color[rgb]{0,.8,.8}\put(\value{XX},\value{YY}){\circle*{10}}}%
\put(\value{XX},\value{YY}){\circle{10}}%
\rgbbox{\value{XX}}{#2}{5}{10}{0}{.8}{.8}{}%
}
\newcommand{\comulunit}[2]{%
\setcounter{XX}{#1}%
\setcounter{YY}{#2}%
\addtocounter{XX}{5}%
\addtocounter{YY}{5}%
{\color[rgb]{0,.8,.8}\put(\value{XX},\value{YY}){\circle*{10}}}%
\put(\value{XX},\value{YY}){\circle{10}}%
}
\newcommand{\compactunit}[2]{%
\setcounter{XX}{#1}%
\setcounter{YY}{#2}%
\addtocounter{XX}{5}%
\addtocounter{YY}{5}%
\rgbfbox{\value{XX}}{#2}{5}{10}{.8}{.8}{.8}{}%
{\color[rgb]{.8,.8,.8}\put(\value{XX},\value{YY}){\circle*{10}}}%
\put(\value{XX},\value{YY}){\circle{10}}%
\rgbbox{\value{XX}}{#2}{5}{10}{.8}{.8}{.8}{\small$\!\!\cap$}%
}
\newcommand{\compactcounit}[2]{%
\setcounter{XX}{#1}%
\setcounter{YY}{#2}%
\addtocounter{XX}{5}%
\addtocounter{YY}{5}%
\rgbfbox{#1}{#2}{5}{10}{.8}{.8}{.8}{}%
{\color[rgb]{.8,.8,.8}\put(\value{XX},\value{YY}){\circle*{10}}}%
\put(\value{XX},\value{YY}){\circle{10}}%
\rgbbox{#1}{#2}{5}{10}{.8}{.8}{.8}{\small$~\cup$}%
}

\newcommand{\Hcolor}{\color[rgb]{0,0,1}}